\documentclass[11pt,reqno]{amsart}
\usepackage{hyperref}

\usepackage{amsmath} 

\usepackage{amsthm}
\usepackage{amssymb}
\usepackage{tikz}
\usepackage{here}
\usepackage{latexsym}
\usepackage{comment}
\numberwithin{equation}{section}
\newtheorem{thm}{Theorem}[section]
\newtheorem{prop}[thm]{Proposition}
\newtheorem{lemm}[thm]{Lemma}
\newtheorem{cor}[thm]{Corollary}

\theoremstyle{remark}
\newtheorem{rem}[thm]{Remark}
\newtheorem{defn}[thm]{Definition}

\newcommand{\BBB}{\mathbb}
\newcommand{\R}{{\BBB R}}
\newcommand{\Z}{{\BBB Z}}
\newcommand{\T}{{\BBB T}}

\newcommand{\N}{{\BBB N}}
\newcommand{\C}{{\BBB C}}

\newcommand{\ZZ}{\mathcal{Z}}

\newcommand{\ee}{\mbox{\boldmath $1$}}

\newcommand{\lec}{{\ \lesssim \ }}
\newcommand{\gec}{{\ \gtrsim \ }}

\newcommand{\CC}{\mathcal{C}}
\newcommand{\RR}{\mathcal{R}}

\newcommand{\al}{\alpha}
\newcommand{\be}{\beta}
\newcommand{\ga}{\gamma}

\newcommand{\ka}{\kappa}

\newcommand{\eps}{\varepsilon}

\newcommand{\x}{\xi}

\newcommand{\de}{\delta}
\newcommand{\om}{\omega}

\newcommand{\supp}{\operatorname{supp}}

\newcommand{\F}{\mathcal{F}}
\newcommand{\1}{{\mathbf 1}}
\newcommand{\HHH}{\mathcal{H}}

\newcommand{\ti}{\widetilde}

\newcommand{\cj}[1]{\overline{#1}}
\renewcommand{\H}{\mathcal H}

\newcommand{\mus}{\mu (\sigma_1,\sigma_2,\sigma_3)}
\newcommand{\kks}{\kappa (\sigma_1,\sigma_2,\sigma_3)}

\title[WP and IP for a system of periodic qDNLS]{Well-posedness and ill-posedness for
\\
a system of periodic quadratic derivative
\\
nonlinear Schr\"odinger equations
}

\author[H. Hirayama]{Hiroyuki Hirayama}
\address[H. Hirayama]{Faculty of Education,
University of Miyazaki, 1-1, Gakuenkibanadai-nishi,
Miyazaki 889-2192, Japan}
\email[H. Hirayama]{h.hirayama@cc.miyazaki-u.ac.jp}

\author[S. Kinoshita]{Shinya Kinoshita}
\address[S. Kinoshita]{Department of Mathematics,
Institute of Science Tokyo, Meguro-Ku,
Tokyo 152-8551, Japan}
\email[S. Kinoshita]{kinoshita@math.titech.ac.jp}

\author[M. Okamoto]{Mamoru Okamoto}
\address[M. Okamoto]{Department of Mathematics,
Graduate School of Advanced Science and Engineering,
Hiroshima University,
1-3-1 Kagamiyama, Higashi-Hiroshima, 739-8526, Japan}
\email[M. Okamoto]{mokamoto@hiroshima-u.ac.jp}


\begin{document}
\begin{abstract}
We consider the Cauchy problem of a system of quadratic derivative nonlinear 
Schr\"odinger equations which was introduced by M. Colin and T. Colin (2004) as a model of laser-plasma interaction.
For the nonperiodic setting, the authors proved some well-posedness results, 
which contain the scaling critical case for $d\geq 2$.  
In the present paper, we prove the well-posedness of this system for the periodic setting. 
In particular, well-posedness is proved at the scaling critical regularity for $d\geq 3$ 
under some conditions for the coefficients of the Laplacian.  
We also prove some ill-posedness results.
As long as we use an iteration argument,
our well-posedness results
are optimal except for some critical cases.
\end{abstract}
\maketitle

\tableofcontents

%%%%%%%%%%%%%%%%%%%%%%%%%%%%%%%%%%%%%%%%%%%%%%%%%%%%%%%%%%%%%%%%%%%%%%%%%%%%%%%%%
%%%%%%%%%%%%%%%%%%%%%%%%%%%%%%%%%%%%%%%%%%%%%%%%%%%%%%%%%%%%%%%%%%%%%%%%%%%%%%%%%
%%%%%%%%%%%%%%%%%%%%%%%%%%%%%%%%%%%%  Section 1   %%%%%%%%%%%%%%%%%%%%%%%%%%%%%%%%%%
%%%%%%%%%%%%%%%%%%%%%%%%%%%%%%%%%%%%%%%%%%%%%%%%%%%%%%%%%%%%%%%%%%%%%%%%%%%%%%%%%
%%%%%%%%%%%%%%%%%%%%%%%%%%%%%%%%%%%%%%%%%%%%%%%%%%%%%%%%%%%%%%%%%%%%%%%%%%%%%%%%%

\section{Introduction\label{intro_torus}}
We consider the Cauchy problem of the system of nonlinear Schr\"odinger equations:
\begin{equation}\label{NLS_sys_torus}
\begin{cases}
\displaystyle (i\partial_{t}+\alpha \Delta )u=-(\nabla \cdot w )v,\hspace{2ex}t>0,\  x\in \T^{d}, \\
\displaystyle (i\partial_{t}+\beta \Delta )v=-(\nabla \cdot \overline{w})u,\hspace{2ex}t>0,\  x\in \T^{d}, \\
\displaystyle (i\partial_{t}+\gamma \Delta )w =\nabla (u\cdot \overline{v}),\hspace{2ex}t>0,\  x\in \T^{d}, \\
(u(0,x), v(0,x), w(0,x))=(u_{0}(x), v_{0}(x), w_{0}(x)),\hspace{2ex}x\in \T^{d},
\end{cases}
\end{equation}
where $\alpha$, $\beta$, $\gamma\in \R\backslash \{0\}$, 
$\T=\R/2\pi\Z$, and the unknown functions $u$, $v$, $w$ are $\C^d$ valued. 
The initial data $(u_0,v_0,w_0)$ is given in the Sobolev space
\[
\mathcal{H}^s(\T^d):=(H^s(\T^d))^d\times (H^s(\T^d))^d\times (H^s(\T^d))^d. 
\]
The system (\ref{NLS_sys_torus}) was introduced by Colin and Colin in \cite{CC04} 
as a model of laser-plasma interaction.

The aim of this paper is to
%classify the Sobolev regularity to be well-posed for \eqref{NLS_sys_torus}.
classify the property of the flow map of \eqref{NLS_sys_torus} in terms of the Sobolev regularity.
One of the threshold values is coming from the scaling transformation,
which is called the scaling critical regularity.
Here, we note that \eqref{NLS_sys_torus} (on $\R^d$) is invariant under the following scaling transformation:
\[
A_{\lambda}(t,x)=\lambda^{-1}A(\lambda^{-2}t,\lambda^{-1}x),
\]
where
$A=(u,v,w)$ and $\lambda>0$.
Hence, the scaling critical regularity is
\begin{equation}
s_{c}= \frac d2 -1.
\label{scal}
\end{equation}

First, we introduce some known results for related problems. 
The system (\ref{NLS_sys_torus}) has quadratic nonlinear terms which contain a derivative. 
A derivative loss arising from the nonlinearity makes the problem difficult. 
In fact, Chihara (\cite{Chi02}) and Christ (\cite{Ch}) proved that the flow map of the Cauchy problem:
\[
\begin{cases}
i\partial_{t}u-\partial_{x}^{2}u=u\partial_{x}u,\ t>0,\ x\in \T,\\
u(0,x)=u_{0}(x),\ x\in \T
\end{cases}
\]
is not continuous on $H^{s}(\T )$ for any $s\in \R$.
See \cite{CGKO17} for the well-posedness for mean-zero initial data.
Moreover,
see also \cite{KoOk1, KoOk2} for ill-posedness results on $\T$.
%There are also positive results for the Cauchy problem:
%\begin{equation*}\label{ddqdnls_T}
%\begin{cases}
%i\partial_{t}u-\Delta u=\partial_j (\overline{u}^{2}),\ t\in \R,\ x\in \T^{d},\\
%u(0,x)=u_{0}(x),\ x\in \T^{d}, 
%\end{cases}
%\end{equation*}
%where $\partial_j =\partial /\partial x_j$ $(j=1,\cdots ,d)$. Gr\"unrock (\cite{Gr}) proved that (\ref{ddqdnls_T}) is globally well-posed in $L^{2}(\T^d )$ for $d=1$ and 
%locally well-posed in $H^{s}(\T^d )$ for $d\geq 2$ and $s>s_{c}$ ($=\frac d2-1$). 
On the other hand, 
for the Cauchy problem of the cubic derivative nonlinear Schr\"odinger equation:
\[
\begin{cases}
i\partial_{t}u+\partial_{x}^{2}u=i \partial_{x}(|u|^{2}u),\ t>0,\ x\in \T,\\
u(0,x)=u_{0}(x),\ x\in \T,
\end{cases}
\]
Herr (\cite{He06}) proved the local well-posedness in $H^{s}(\T )$ for $s\geq \frac 12$ by using the gauge transform 
and Win (\cite{Wi10}) proved the global well-posedness in $H^{s}(\T )$ for $s> \frac 12$.
For the nonperiodic case, there are many results for the well-posedness of the nonlinear Schr\"odinger equations with derivative nonlinearity.
See, for example, \cite{Be08}, \cite{BL01}, \cite{Chi99}, \cite{CKSTT02}, \cite{HKNV}, \cite{IKO16}, \cite{KNV23}, \cite{MWX11}, \cite{Tak01},
and references therein. 

Next, we mention some known results for 
the well-posedness of (\ref{NLS_sys_torus}). 
We set
\begin{equation}\label{coeff_condition}
\mu :=\alpha\beta\gamma \left(\frac{1}{\alpha}-\frac{1}{\beta}-\frac{1}{\gamma}\right),\ \ 
\kappa :=(\alpha -\beta)(\alpha -\gamma)(\beta +\gamma),\ \ 
\widetilde{\kappa} :=(\alpha -\gamma)(\beta +\gamma). 
\end{equation}
For the nonperiodic case,
in \cite{Hi} and \cite{HK},
the first and second 
authors proved the well-posedness of (\ref{NLS_sys_torus}) 
in $\H^s(\R^d)$ under the condition $\kappa \ne 0$, 
where $s$ is given in Table~\ref{WP_NLS_sys} below. 
\begin{table}[h]
\begin{center}
\begin{tabular}{|l|l|l|l|l|l|}
\hline
\multicolumn{2}{|c|}{} & \multicolumn{1}{|c|}{$d=1$} & 
\multicolumn{1}{|c|}{$d=2$} & 
\multicolumn{1}{|c|}{$d=3$} &
\multicolumn{1}{|c|}{$d\geq 4$}\\
\hline
\multicolumn{2}{|c|}{$\mu > 0$} & \multicolumn{1}{|c|}{$s\geq 0$} & \multicolumn{3}{|c|}{$s\geq s_{c}$} \\              
\cline{1-5}
\multicolumn{2}{|c|}{$\mu =0$} & \multicolumn{3}{|c|}{$s\geq 1$} 
&
%\multicolumn{1}{|c|}{$s\geq s_{c}$}
\\
\cline{1-5}
\multicolumn{2}{|c|}{$\mu < 0$, $\kappa \ne 0$} & \multicolumn{2}{|c|}{$s\geq \frac 12$} & 
\multicolumn{1}{|c|}{$s>s_c$} & \\
\hline
\end{tabular}
\caption{Regularities to be well-posed in \cite{Hi} and \cite{HK}%
\label{WP_NLS_sys}}
\end{center}
\end{table}

\noindent In \cite{Hi} and \cite{HKO2021}, the authors also considered 
the case $\kappa =0$ 
and proved the 
well-posedness of \rm (\ref{NLS_sys_torus}) 
in $\H^s(\R^d)$, where $s$ is given in Table~2%\ref{WP_table22}
\ below. 
\begin{table}[h]
\begin{center}
\begin{tabular}{|l|l|l|l|l|}
\hline
\multicolumn{2}{|c|}{} & \multicolumn{1}{|c|}{$d=1,2$} & \multicolumn{1}{|c|}{$d \ge 3$}\\
\hline
\multicolumn{2}{|c|}{$\alpha-\beta=0$,\ $\widetilde{\kappa}\ne 0$} & \multicolumn{1}{|c|}{$s\geq \frac 12$} &\multicolumn{1}{|c|}{$s>s_c$} \\
\hline
\end{tabular}
\caption{Regularities to be well-posed in \cite{Hi} and \cite{HKO2021}%
\label{WP_table22}}
\end{center}
\end{table}
On the other hand, 
the first author proved in \cite{Hi} that
the flow map is not $C^2$ for $s<1$ if $\mu= 0$, 
for $s< \frac 12$ if $\mu< 0$ and 
$\widetilde{\kappa}\ne 0$, 
and for any $s\in \R$ if $\widetilde{\kappa} =0$. 
Furthermore, the authors proved in \cite{HKO2021} that 
the flow map is not $C^3$ for $s<0$ if $\mu >0$. 
Therefore, the well-posedness of (\ref{NLS_sys_torus}) 
in $\H^s(\R^d)$ is optimal except for some scaling critical cases 
if we use an iteration argument. 
By using the modified energy method, 
the authors in \cite{HKO22} also obtained the well-posedness in $\H^s(\R^d)$
for $s> \frac d2+3$ under the condition $\beta +\gamma \ne 0$, 
which result contained the case $\alpha -\gamma =0$. 
The well-posedness for radial initial data is also considered in \cite{HKO2020}.

Now, we give the main results in the present paper.
Recall that the scaling critical regularity $s_{c}$ is given by \eqref{scal} and $\mu$, $\kappa$, and $\widetilde{\kappa}$ are given in (\ref{coeff_condition}). 
%For a Banach space $\HHH$ and $r>0$, we define $B_r(\HHH):=\{ f\in \HHH \,|\, \|f\|_\HHH \le r \}$. 
%
We note that if $\alpha$, $\beta$, $\gamma \in \R\backslash \{0\}$ satisfy $\mu \ge 0$, 
then $\kappa \neq 0$  holds. 
\begin{thm}[Critical case]\label{wellposed_T}
We assume $\alpha$, $\beta$, $\gamma \in \R\backslash \{0\}$. \\
{\rm (i)}\ If $d\geq 3$ and $\mu >0$,  
then {\rm (\ref{NLS_sys_torus})} is locally well-posed in $\mathcal{H}^{s_{c}}(\T^d )$. 
More precisely, for any $(u_{0}, v_{0}, w_{0})\in \mathcal{H}^{s_c}(\T^d)$, 
there exist $T>0$ and a solution
\[
(u,v,w)\in C \big( [0,T);\mathcal{H}^{s_{c}}(\T^d ) \big)
\]
to the system {\rm (\ref{NLS_sys_torus})} on $(0, T)$. 
Such solution is unique in $X^{s_{c}}([0,T))$ which is a closed subspace of $C\left([0,T);\mathcal{H}^{s_{c}}(\T^d )\right)$ {\rm (see (\ref{Xs_norm_T}) and Definition~\ref{YZ_space_T})}. 
Moreover, the flow map
\[
\mathcal{H}^{s_c}(\T^d )\ni (u_{0},v_{0},w_{0})\mapsto (u,v,w)\in X^{s_{c}}([0,T))
\]
is Lipschitz continuous. \\
{\rm (ii)}\ If $d\geq 4$ and $\mu =0$, 
then {\rm (\ref{NLS_sys_torus})} is locally well-posed in $\mathcal{H}^{s_{c}}(\T^d )$. \\
{\rm (iii)}\ If $d\geq 5$, $\mu < 0$, and $\kappa \neq 0$, 
then {\rm (\ref{NLS_sys_torus})} is locally well-posed in $\mathcal{H}^{s_{c}}(\T^d )$.
\end{thm}
\begin{thm}[Subcritical case]\label{wellposed_sub_T}Let $d\ge 1$, $\alpha$, $\beta$, $\gamma \in \R\backslash \{0\}$, and $s >s_c$. 
If one of
\begin{enumerate}
\item[{\rm (i)}] $\mu>0$ and $s>0;$ 
\item[{\rm (ii)}] $\mu \le 0$, $\widetilde{\kappa} \neq 0$, and $s\ge 1$
except for $\mu <0$ and $(d,s)\ne (3,1)$
\end{enumerate}
is satisfied, 
then {\rm (\ref{NLS_sys_torus})} is locally well-posed in $\mathcal{H}^{s}(\T^d )$.
\end{thm}
\begin{table}[h]
\begin{center}
\begin{tabular}{|l|l|l|l|l|l|}
\hline
\multicolumn{2}{|c|}{} & \multicolumn{1}{|c|}{$d=1,2$} & 
\multicolumn{1}{|c|}{$d=3$} & 
\multicolumn{1}{|c|}{$d=4$} & 
\multicolumn{1}{|c|}{$d\geq 5$}\\
\hline
\multicolumn{2}{|c|}{$\mu > 0$} & \multicolumn{1}{|c|}{$s>0$} &\multicolumn{3}{|c|}{} \\              
\cline{1-4}
\multicolumn{2}{|c|}{$\mu =0$} & \multicolumn{2}{|c|}{} & 
\multicolumn{2}{c|}{$s\geq s_{c}$}
\\
\cline{1-2} \cline{4-5}
\multicolumn{2}{|c|}{$\mu < 0$, $\kappa \ne 0$} & \multicolumn{1}{|c|}{$s\geq 1$}
&\multicolumn{1}{|c|}{$s> 1$}
&\multicolumn{1}{|c|}{$s>s_c$} & \\
\cline{1-2} \cline{6-6}
\multicolumn{2}{|c|}{$\alpha-\beta =0$, $\widetilde{\kappa} \ne 0$} & \multicolumn{1}{|c|}{} 
&\multicolumn{1}{|c|}{}
&\multicolumn{2}{|c|}{}\\
\hline
\end{tabular}
\caption{Regularities to be well-posed in Theorems~\ref{wellposed_T} and ~\ref{wellposed_sub_T}%
\label{WP_NLS_sys_Td}}
\end{center}
\end{table}

\begin{rem}
The dependence of the existence time $T$ on the initial data differs between the critical and subcritical cases.
On the one hand, $T$ depends on the norm of initial data in the subcritical case (Theorem \ref{wellposed_sub_T}),
but on the other hand, $T$ also depends on the profile of the initial data in the critical case (Theorem \ref{wellposed_T}).
\end{rem}

\begin{rem}
The condition $\mu>0$ yields
that the dispersive effect in the nonlinear terms does not vanish, which is called nonresonance.
Moreover, $\kappa \neq 0$ is the nonresonance condition under the High-Low interaction.

Oh {\rm (\cite{Oh09})} studied the resonance and the nonresonance for the system of KdV equations. 
He proved that if the coefficient of the linear term of the system satisfies the nonresonance condition, 
then the well-posedness of the system is obtained at lower regularity than the 
regularity for the coefficient satisfying the resonance condition. 
\end{rem}

\begin{rem}
The well-posedess result for (i) in Theorem \ref{wellposed_sub_T} is not novel.
Indeed,
Gr\"unrock \cite{Gru00} proved that
the Cauchy problem for the quadratic derivative nonlinear Schr\"odinger equation
\[
(i\partial_{t}+\Delta )u= \partial_{x_1}(\cj u^2)
\]
is well-posed in $H^s(\T^d)$ when $s \ge 0$ and $s> s_c$.
A similar argument as in \cite{Gru00} applies to \eqref{NLS_sys_torus} with $\mu>0$.
Specifically,
\eqref{NLS_sys_torus} with $\mu>0$ is well-posed in $H^s(\T^d)$ for $s \ge 0$ and $s> s_c$.
In particular,
with the $L^2$-conservation law below,
\eqref{NLS_sys_torus} is globally well-posed in $H^s(\T)$ for $s \ge 0$.
However,
since the proof of Theorem \ref{wellposed_sub_T} (i) relies on the Littlewood-Paley decomposition and the bilinear Strichartz estimate,
the case $s=0$ is excluded to avoid logarithmic divergences.
\end{rem}

The system (\ref{NLS_sys_torus}) has the following conserved quantities (see Proposition 7.1 in \cite{Hi}):
\[
\begin{split}
Q(u,v,w)&:=2\|u\|_{L^{2}_{x}}^{2}+\|v\|_{L^{2}_{x}}^{2}+\|w\|_{L^{2}_{x}}^{2},\\
H(u,v,w)&:=\alpha \|\nabla u\|_{L^{2}_{x}}^{2}+\beta \|\nabla v\|_{L^{2}_{x}}^{2}+\gamma \|\nabla w\|_{L^{2}_{x}}^{2}+2{\rm Re}(w ,\nabla (u\cdot \overline{v}))_{L^{2}_{x}}. 
\end{split}
\]
By using these quantities, we obtain the following result.
\begin{thm}\label{global_extend_T}
Let $d\ge 1$. 
We assume that $\alpha$, $\beta$, $\gamma \in \R \backslash \{0\}$ have the same sign 
and satisfy the following {\rm (i)} or {\rm (ii)}:
\begin{enumerate}
\item[{\rm (i)}] $1\le d\le 4$ and $\mu\ge 0;$
\item[{\rm (ii)}] $1 \le d \le 2$, $\mu <0$, and $\widetilde{\kappa} \ne 0$.
\end{enumerate}
Then, {\rm (\ref{NLS_sys_torus})} is globally well-posed for small data in $\mathcal{H}^{1}(\T^d )$. 
\end{thm}

\begin{rem}
If the initial data are small enough,
we obtain the solution to \eqref{NLS_sys_torus}
on the time interval $[0,1)$,
even in the scaling critical case.
See Subsection \ref{subsec:WP11} below.
Therefore,
{\rm Theorem~\ref{global_extend_T}} follows from 
a priori estimate of the $\H^1$-norm which is obtained by the conservation quantities.
Proof of the a priori estimate is the same as that in the nonperiodic case {\rm (see Proposition 7.2 in \cite{Hi})}. 
\end{rem}

The main tools of above well-posedness results 
are the Strichartz  and bilinear Strichartz estimates 
with Fourier restriction method. 
The Strichartz estimate on tori was proved by
%Bourgain (\cite{Bo93}) and improved by Bourgain (\cite{Bo13}) and Bourgain and Demeter (\cite{BD15}). 
\cite{Bo93, Bo13, BD15, KV16}.
Because our results contain the scaling critical case, 
we will use $U^{p}$ and $V^{p}$ type spaces as the resolution space.  

To obtain the well-posedness of (\ref{NLS_sys_torus}) at the scaling critical regularity, 
we will show the following bilinear estimate. (The definition of $Y^0_{\sigma}$ will be given in Definition~\ref{YZ_space_T}.)
\begin{equation}\label{Ybe_ext_T}
\begin{split}
&\|\eta (t)P_{N_{3}}(P_{N_{1}}u_{1}\cdot P_{N_{2}}u_{2})\|_{L^{2}(\R \times \T^{d})}\\
&\lesssim N_{\min}^{s_c}\left(\frac{N_{\min}}{N_{\max}}+\frac{1}{N_{\min}}\right)^{\delta}\|P_{N_{1}}u_{1}\|_{Y^{0}_{\sigma_{1}}}\|P_{N_{2}}u_{2}\|_{Y^{0}_{\sigma_{2}}},
\end{split}
\end{equation}
where $P_N$ denotes the Littlewood-Paley projection, $\de>0$, and
\begin{equation}
\eta (t) := \Big( \frac{\sin \frac{\pi t}2}{\pi t} \Big)^2.
\label{cutoff}
\end{equation}
Wang (\cite{Wa}) proved 
a similar bilinear estimate for the case $N_{1}\sim N_{3}\gtrsim N_{2}$
by using the decomposition for the Fourier support of $u_1$ 
into the stripes which are contained in some cube with side-length $N_2$.
See also Lemma 3.3 in \cite{KV16}.
To prove (\ref{Ybe_ext_T}) for the case $N_1\sim N_2\gg N_3$, 
we will use the decomposition for both $u_1$ and $u_2$. 
This is a different point from the case $N_{1}\sim N_{3}\gtrsim N_{2}$.

The well-posedness in the subcritical cases follows from a slight modification of the critical cases,
except for the cases $d=1,2$ and $s=1$.
When $d=1,2$ and $s=1$,
we use a convolution estimate.
See Subsection \ref{subsec:ResIII}.
Moreover,
we employ the dyadic decomposition of modulation parts.
We then use Besov-type Fourier restriction norm spaces (instead of $U^2$-type spaces) to prove the well-posedness in $\H^1(\T^d)$ for $d=1,2$.
See Subsection \ref{subsec:ds21}.

From tables
\ref{WP_NLS_sys}, \ref{WP_table22}, and \ref{WP_NLS_sys_Td},
there are some differences between Sobolev regularities to be well-posed for \eqref{NLS_sys_torus} on $\R^d$ and $\T^d$.
In fact,
the well-posedness in $\H^s(\R^d)$ holds for $s> \frac 12$
at least when $\mu \le 0$, $\widetilde \kappa \neq 0$, and $d=1,2,3$.
However,
we can not prove the well-posedness in $\H^s(\T^d)$ for $\mu \le 0$ and $s<1$
by using an iteration argument.
See Theorem \ref{thm:IP2} below.
Moreover,
the well-posedness in $\H^1(\T^d)$ for $\mu<0$ and $\kappa \neq 0$ is unsolved
even when $d=3$,
since the Strichartz estimate contains a derivative loss.
Indeed,
if the $L^3$-Strichartz estimate without a derivative loss holds,
we can show the well-posedness in $\H^1(\T^3)$ for $\mu<0$ and $\kappa \neq 0$.

\begin{rem}
Since the Strichartz estimate for irrational tori is valid (\cite{BD15, KV16}),
our main results also hold for irrational tori.
Namely,
with straightforward modifications to our proof,
we can replace $\T^d$ in this paper with
\[
\T^d_\theta := \prod_{j=1}^d (\R / 2\pi \theta _j \Z)
\]
where $\theta = (\theta_1, \dots, \theta_d) \in (0,\infty)^d$.
See Remarks \ref{rem:irrm1} and \ref{rem:ConvEstTori} below.
\end{rem}

We also obtain some negative results as follows.  
\begin{thm}
\label{thm:IP}
Let $\al$, $\be$, $\ga \in \R\backslash \{0\}$ and $s \in \R$ satisfy one of the followings:
\begin{enumerate}
\item[{\rm (i)}]
$\be + \ga=0$ and $s \neq 0;$

\item[{\rm (ii)}]
$\al - \ga=0$ and $s<0$.
\end{enumerate}
Then,
we have the norm inflation in $\H^s(\T^d)$ for \eqref{NLS_sys_torus}.
More precisely,
there exist a sequence $\{ (u_n, v_n, w_n) \}$ of solutions to \eqref{NLS_sys_torus}
and a sequence $\{ t_n \}$ of positive numbers such that
\begin{align*}
&\lim_{n \to \infty} t_n =0,
\quad
\lim_{n \to \infty} \| (u_n(0), v_n(0), w_n(0)) \|_{\H^s}=0,
\\
&
\lim_{n \to \infty}
\| (u_n(t_n), v_n(t_n), w_n(t_n)) \|_{\H^s} = \infty.
\end{align*}
\end{thm}
\begin{rem}
The norm inflation in $\H^s(\T^d)$ for $\be+\ga=0$ and $s>0$ comes from the high$\times$low$\to$high interaction.
On the other hand,
the norm inflation in $\H^s(\T^d)$ for ($\be+\ga=0$ and $s<0$) or ($\al-\ga=0$ and $s<0$) comes from the high$\times$high$\to$low interaction.
The norm inflation implies a discontinuity of the flow map of \eqref{NLS_sys_torus}.
In particular,
the Cauchy problem \eqref{NLS_sys_torus} is ill-posed in $\H^s(\T^d)$ for 
the case (i) or (ii) in Theorem \ref{thm:IP}.
\end{rem}
Because of the $L^2$-conservation,
the norm inflation in $\H^0(\T^d)$ for \eqref{NLS_sys_torus} does not occur.
However, we obtain the discontinuity when $\beta+\gamma=0$ as follows.

\begin{thm}
\label{thm:IP2}
%Let $\al$, $\be$, $\ga \in \R\backslash \{0\}$ satisfy
%\[
%\be + \ga=0.
%\]
Let $\al$, $\be$, $\ga \in \R\backslash \{0\}$ and $s \in \R$ satisfy $\be + \ga=0$ Then,
the flow map of \eqref{NLS_sys_torus} is discontinuous in $\H^s(\T^d)$.
\end{thm}

We also obtain that the flow map is not locally uniformly continuous in $\H^s(\T^d)$ in some cases.

\begin{thm}
\label{thm:uni1}
Let $\al$, $\be$, $\ga \in \R\backslash \{0\}$ and $s \in \R$ satisfy one of the followings:
\begin{enumerate}
\item[{\rm (i)}]
$\al-\ga=0$ and $s \ge 0;$

\item[{\rm (ii)}]
%There exists $k \in \Q$ satisfying
%\begin{equation}
%\label{keq1}
%\al k^2 - \be (k-1)^2-\ga=0.
%\end{equation}
%Moreover, $s<1$.
$\mu \le 0$ and $s<1$.
\end{enumerate}
Then, the flow map for \eqref{NLS_sys_torus} fails to be locally uniformly continuous in $\H^s(\T^d)$.
More precisely,
there exist sequences $\{ (u_n, v_n, w_n)\}$, $\{(\ti u_n, \ti v_n, \ti w_n)\}$ of solutions to \eqref{NLS_sys_torus}
and a sequence $\{ t_n \}$ of positive numbers such that
\begin{align*}
&
\lim_{n \to \infty} t_n =0,
\\
&
\sup_{n \in \N}
\big(
\| (u_n(0), v_n(0), w_n(0)) \|_{\H^s}
+
\| (\ti u_n(0), \ti v_n(0), \ti w_n(0)) \|_{\H^s} \big)
\lec 1,
\\
&
\lim_{n \to \infty}
\| (u_n(0), v_n(0), w_n(0)) - (\ti u_n(0), \ti v_n(0), \ti w_n(0)) \|_{\H^s}
=0,
\\
&
\lim_{n \to \infty}
\|
(u_n(t_n), v_n(t_n), w_n(t_n)) - (\ti u_n(t_n), \ti v_n(t_n), \ti w_n(t_n)) \|_{\H^s} \gec 1.
\end{align*}
\end{thm}

Theorem \ref{thm:uni1} implies that
the well-posedness in $\H^s(\T^d)$ does not follow from an iteration argument.
As mentioned before, the authors \cite{HKO22} proved the well-posedness in $\H^s(\T^d)$ for $\al-\ga=0$ and $s>\frac{d}{2}+3$
by using the energy method.%
\footnote{Strictly speaking, the nonperiodic cases are treated in \cite{HKO22}.
However, the same argument works for the periodic cases.}
Namely, the flow map of \eqref{NLS_sys_torus} is continuous in $\H^s(\T^d)$ for $\al-\ga=0$ and $s>\frac{d}{2}+3$.
When $\al-\ga=0$ and $0 \le s \le \frac{d}{2}+3$,
it is unsolved whether the flow map is continuous or not in $\H^s(\T^d)$.

The same argument of Proposition 5.1 in \cite{HKO2021} yields that
the flow map of \eqref{NLS_sys_torus} fails to be
%third-order continuously differentiable
$C^3$
if $d=1$, $\mu>0$, and $s<0$.%
\footnote{If $\frac \gamma \alpha$ is rational, the argument is the same as in \cite{HKO2021}.
If $\frac \gamma \alpha$ is irrational,
for any $N \in \N$, there exists a rational number $k_N$ such that
$|k_N - \frac \gamma \al|<\frac 1N$.
Then, the argument in \cite{HKO2021} with $k$ replaced by $k_N$ shows that the flow map is not $C^3$.
}
Therefore,
we obtain almost sharp well-posedness results of \eqref{NLS_sys_torus}
in $\H^s(\T^d)$ (except for some critical cases)
if we use an iteration argument. 

To prove Theorems \ref{thm:IP2} and \ref{thm:uni1},
we use an ODE approach as in \cite{BGT02} and \cite{CCT03}.
Since we consider the system \eqref{NLS_sys_torus},
the corresponding ODEs become a Hamiltonian system.
By using conserved quantities of the Hamiltonian system,
we study the asymptotic behavior of the ODEs.
See Sections \ref{sec:IP} and \ref{Sec:7}.

\begin{rem}
{\rm (i)}\ 
In the case $\al-\ga=0$ and $s>0$,
we consider the high$\times$low$\to$high interaction in the proof of Theorem \ref{thm:uni1}.
However,
the low-frequency part here is $v$,
while $u$ is the low-frequency part in Theorem \ref{thm:IP} for $\be+\ga=0$ and $s>0$.
Because of this difference,
our argument does not yield ill-posedness in $H^s(\T^d)$ for $\al-\ga=0$ and $s>0$.
See also Remark \ref{rem:notip}.\\
%The failure of the locally uniformly continuity for $s<1$
%
{\rm (ii)}\
Theorem~\ref{thm:uni1} (ii) 
comes from the high$\times$high$\to$high interaction.
\\
{\rm (iii)}\
Theorem~\ref{thm:uni1} (ii) contains the case $\mu<0$. 
For the nonperiodic setting under the condition $\mu<0$ and $\ti \ka \neq 0$,
the well-posedness was obtained 
for $d=1,2$ and $s\ge \frac 12$ by the iteration argument {\rm (see Table~\ref{WP_NLS_sys})}. 
In particular, the flow map is analytic. 
This is a different point between periodic and nonperiodic settings. 
\end{rem}

%\kuuhaku \\

\noindent {\bf Notation.} 
We define the integral on $\T^{d}$:
\[
\int_{\T^{d}}f(x)dx:=\int_{[0,2\pi ]^{d}}f(x)dx .
\]
We denote the spatial Fourier coefficients for the function on $\T^{d}$ as
\[
\F_{x}[f](\xi )=\widehat{f}(\xi ):=\int_{\T^{d}}f(x)e^{-i\xi \cdot x}dx,\ \xi \in \Z^{d}
\]
and the space time Fourier transform as
\[
\F[f](\tau ,\xi)
%=\widetilde{f}(\tau ,\xi )
:=\int_{\R}\int_{\T^{d}}f(t,x)e^{-it\tau}e^{-ix\cdot \xi}dxdt,\ \tau \in \R, \ \xi \in \Z^{d}. 
\]
For $\sigma \in \R$, the free evolution $e^{it\sigma \Delta}$ on $L^{2} (\T^d)$ is given as a Fourier multiplier
\[
\F_{x}[e^{it\sigma \Delta}f](\xi )=e^{-it\sigma |\xi |^{2}}\widehat{f}(\xi ). 
\]
We will use $A\lesssim B$ to denote an estimate of the form $A \le CB$ for some constant $C$ and write $A \sim B$ to mean $A \lesssim B$ and $B \lesssim A$. 
We will use the convention that capital letters denote dyadic numbers, e.g. $N=2^{n}$ for $n\in \N_0 := \N \cup \{ 0 \}$ and for a dyadic summation we write
$\sum_{N}a_{N}:=\sum_{n\in \N_0}a_{2^{n}}$ and $\sum_{N\geq M}a_{N}:=\sum_{n\in \N_0, 2^{n}\geq M}a_{2^{n}}$ for brevity. 
Let $\chi \in C^{\infty}_{0}((-2,2))$ be an even, non-negative function such that $\chi (s)=1$ for $|s|\leq 1$. 
We define $\psi_{1}(s):=\chi (s)$ and $\psi_{N}(s):=\psi_{1}(N^{-1}s)-\psi_{1}(2N^{-1}s)$ for $N\geq 2$. 
We define frequency and modulation projections
\begin{align*}
%&\widehat{P_{S}u}(\xi ):=\ee_{S}(\xi )\widehat{u}(\xi ),\quad
%\widehat{P_{N}u}(\xi ):=\psi_{N}(|\xi |)\widehat{u}(\xi ),\\
&\F_x[ P_{S}u ](\xi ):=\ee_{S}(\xi ) \F_x [u](\xi ),\quad
\F_x[ P_{N}u](\xi ):=\psi_{N}(|\xi |) \F_x[ u](\xi ),\\
&\F[ Q_{M}^{\sigma}u ](\tau ,\xi ):=\psi_{M}(\tau +\sigma |\xi|^{2}) \F [u](\tau ,\xi )
\end{align*}
for a set $S\subset \Z^{d}$ and dyadic numbers $N$, $M$, where $\ee_{S}$ is the characteristic function of $S$. 
Furthermore, we define $Q_{\geq M}^{\sigma}:=\sum_{N\geq M}Q_{N}^{\sigma}$ and $Q_{<M}^{\sigma}:=Id -Q_{\geq M}^{\sigma }$. 
For $s\in \R$, we define the Sobolev space $H^{s}(\T^{d})$ as the space of all 
periodic distributions for which the norm
\[
\|f\|_{H^{s}}:=\left(\sum_{\xi \in \Z^{d}}\langle \xi \rangle^{2s}|\widehat{f}(\xi )|^{2}\right)^{\frac 12}\sim \left(\sum_{N\geq 1}N^{2s}\|P_{N}f\|_{L^{2}(\T^{d})}^{2}\right)^{\frac 12}
\]
is finite. 

The rest of this paper is planned as follows.
In Section \ref{func_sp_T}, we will give the definition and properties of the $U^{p}$ space and $V^{p}$ space. 
In Section \ref{b_Stri_T}, we will introduce some Strichartz estimates on tori and prove the bilinear estimates. 
In Section \ref{tri_est_T}, we will give the trilinear estimates. 
In Section \ref{WP2},  we will prove the well-posedness results (Theorems~\ref{wellposed_T} 
and ~\ref{wellposed_sub_T}). 
In Sections \ref{sec:IP} and \ref{Sec:7}, we will give some counter examples of well-posedness. 
In particular, 
we will prove the ill-posedness results (Theorems~\ref{thm:IP} and ~\ref{thm:IP2}) 
in Section \ref{sec:IP} 
and the failure of the uniform continuity of the flow map (Theorem~\ref{thm:uni1}) in Section \ref{Sec:7}. 
%
%%%%%%%%%%%%%%%%%%%%%%%%%%%%%%%%%%%%%%%%%%%%%%%%%%%%%%%%%%%%%%%%%%%%%%%%%%%%%%%%%
%%%%%%%%%%%%%%%%%%%%%%%%%%%%%%%%%%%%%%%%%%%%%%%%%%%%%%%%%%%%%%%%%%%%%%%%%%%%%%%%%
%%%%%%%%%%%%%%%%%%%%%%%%%%%%%%%%%%%%  Section 2   %%%%%%%%%%%%%%%%%%%%%%%%%%%%%%%%%%
%%%%%%%%%%%%%%%%%%%%%%%%%%%%%%%%%%%%%%%%%%%%%%%%%%%%%%%%%%%%%%%%%%%%%%%%%%%%%%%%%
%%%%%%%%%%%%%%%%%%%%%%%%%%%%%%%%%%%%%%%%%%%%%%%%%%%%%%%%%%%%%%%%%%%%%%%%%%%%%%%%%

\section{$U^{p}$, $V^{p}$ spaces  and their properties \label{func_sp_T}}
In this section, we define the $U^{p}$ space and the $V^{p}$ space, 
and mention the properties of these spaces which are proved in \cite{HHK09}
and \cite{HTT11} (see, also \cite{HHK10}).
Throughout this section, $\HHH$ denotes a separable Hilbert space over $\C$. 

We define the set of finite partitions $\ZZ$ as
\[
\ZZ :=\left\{ \{t_{k}\}_{k=0}^{K}| \ K\in \N , \ -\infty <t_{0}<t_{1}<\cdots <t_{K}\leq \infty \right\}
\]
and we put $v(\infty):=0$ for all functions $v:\R \rightarrow \HHH$. 
\begin{defn}\label{upsp_T}
Let $1\leq p <\infty$.
We call a function $a:\R\rightarrow \HHH$ 
a ``$U^{p}${\rm -atom}'' if
there exist
$\{t_{k}\}_{k=0}^{K}\in \ZZ$ and $\{\phi_{k}\}_{k=0}^{K-1}\subset \HHH$ with 
$\sum_{k=0}^{K-1}\|\phi_{k}\|_{\HHH}^{p}=1$ such that
\[
a(t)=\sum_{k=1}^{K}\ee_{[t_{k-1},t_{k})}(t)\phi_{k-1}.
\]
Furthermore, we define the atomic space 
\[
U^{p}(\R ;\HHH):=\left\{ \left.
%u=
\sum_{j=1}^{\infty}\lambda_{j}a_{j}
\text{ in } L^{\infty}_{t}(\R ;\HHH)
\right| \ a_{j}:U^{p}\text{-}{\rm atom},\ \{ \lambda_{j} \} \in l^1 \right\}
\]
with the norm
\[
\|u\|_{U^{p}(\R ;\HHH)}:=\inf \left\{\sum_{j=1}^{\infty}|\lambda_{j}|\left|
\
u=\sum_{j=1}^{\infty}\lambda_{j}a_{j} \text{ in } L^{\infty}_{t}(\R ;\HHH),\ 
a_{j}:U^{p}\text{-}{\rm atom},\ \{ \lambda_{j} \} \in l^1 \right.\right\}.
\]
Here, $\l^1$ denotes the space of all absolutely summable $\C$-valued sequences.
\end{defn}
\begin{defn}\label{vpsp_T}
Let $1\leq p <\infty$. We define the space of the bounded $p$-variation 
\[
V^{p}(\R ;\HHH):=\{ v:\R\rightarrow \HHH|\ \|v\|_{V^{p}(\R ;\HHH)}<\infty \}
\]
with the norm
\[
\|v\|_{V^{p}(\R ;\HHH)}:=\sup_{\{t_{k}\}_{k=0}^{K}\in \ZZ}\left(\sum_{k=1}^{K}\|v(t_{k})-v(t_{k-1})\|_{\HHH}^{p}\right)^{1/p}.
\]
Likewise, let $V^{p}_{-, rc}(\R ;\HHH)$ denote the closed subspace of all right-continuous functions $v\in V^{p}(\R ;\HHH)$ with 
$\lim_{t\rightarrow -\infty}v(t)=0$, endowed with the same norm $\|\cdot \|_{V^{p}(\R ;\HHH)}$.
\end{defn}
\begin{prop}[\cite{HHK09} Propositions\ 2,2,\ 2.4,\ Corollary\ 2.6]\label{upvpprop_T}
Let $1\leq p<q<\infty$. \\
{\rm (i)} $U^{p}(\R ;\HHH)$, $V^{p}(\R ;\HHH)$, and $V^{p}_{-, rc}(\R; \HHH)$ are Banach spaces. \\ 
{\rm (ii)} The embeddings $U^{p}(\R ;\HHH)\hookrightarrow V^{p}_{-,rc}(\R ;\HHH)\hookrightarrow U^{q}(\R ;\HHH)\hookrightarrow L^{\infty}_{t}(\R ;\HHH)$ are continuous. 
\end{prop}
\begin{defn}
Let $1\leq p<\infty$, $s\in \R$, and $\sigma \in \R \backslash \{0\}$. We define
\[
U^{p}_{\sigma}H^s:=
\{ u:\R\rightarrow H^s|\ e^{-it \sigma \Delta}u\in U^{p}(\R ;H^s(\T^d ))\}
%\{ e^{it \sigma \Delta}u |\ u \in U^{p}(\R ;H^s(\T^d )) \}
\]
with the norm $\|u\|_{U^{p}_{\sigma}H^s}:=\|e^{-it \sigma \Delta}u\|_{U^{p}(\R ;H^s)}$
and
\[
V^{p}_{\sigma}H^s:=
\{ v:\R\rightarrow H^s|\ e^{-it \sigma \Delta}v\in V^{p}_{-,rc}(\R ;H^s(\T^d ))\}
%\{ e^{it \sigma \Delta}v |\ v \in V^{p}_{-,rc}(\R ;H^s(\T^d ))\}
\]
with the norm $\|v\|_{V^{p}_{\sigma}H^s}:=\|e^{-it \sigma \Delta}v\|_{V^{p}(\R ;H^s)}$.
\end{defn}
\begin{rem}
We note that $\|\overline{u}\|_{U^{p}_{\sigma}H^s}=\|u\|_{U^{p}_{-\sigma}H^s}$ and $\|\overline{v}\|_{V^{p}_{\sigma}H^s}=\|v\|_{V^{p}_{-\sigma}H^s}$. 
\end{rem}
\begin{prop}[\cite{HHK09} Corollary\ 2.18]\label{projest_T}
Let $1< p<\infty$ and $\sigma \in \R \backslash \{0\}$. We have
\begin{align}
&\|Q_{\geq M}^{\sigma}u\|_{L^{2}(\R ;L^2)}\lesssim M^{-\frac 12}\|u\|_{V^{2}_{\sigma}L^2},\label{highMproj_T}\\
&\|Q_{<M}^{\sigma}u\|_{V^{p}_{\sigma}L^2}\lesssim \|u\|_{V^{p}_{\sigma}L^2},\ \ \|Q_{\geq M}^{\sigma}u\|_{V^{p}_{\sigma}L^2}\lesssim \|u\|_{V^{p}_{\sigma}L^2}.\label{Vproj_T}
%&\|Q_{<M}^{\sigma}u\|_{U^{p}_{\sigma}L^{2}}\lesssim \|u\|_{U^{p}_{\sigma}L^{2}},\ \ \|Q_{\geq M}^{\sigma}u\|_{U^{p}_{\sigma}L^{2}}\lesssim \|u\|_{U^{p}_{\sigma}L^{2}}. \label{Uproj}
\end{align}
\end{prop}

By (\ref{highMproj_T}), we also obtain
\begin{equation}\label{highMproj_Tequal}
\|Q_{M}^{\sigma}u\|_{L^{2}(\R ;L^2)}\lesssim M^{-\frac 12}\|u\|_{V^{2}_{\sigma}L^2}
\end{equation}

\begin{prop}[\cite{HHK09} Proposition\ 2.19]\label{multiest_T}
Let 
\[
T_{0}:L^2(\T^d )\times \cdots \times L^2(\T^d )\rightarrow L^{1}_{loc}(\T^d )
\]
be an $m$-linear operator, $I\subset \R$ be an interval, 
and $\rho:I\rightarrow [0,\infty)$ be a continuous function. 
Assume that for some $1\leq p, q< \infty$, 
\[
\|\rho (t)T_{0}(e^{it\sigma_{1}\Delta}\phi_{1},\cdots ,e^{it\sigma_{m}\Delta}\phi_{m})\|_{L^{p}_{t}(I :L^{q}_{x})}\lesssim \prod_{i=1}^{m}\|\phi_{i}\|_{L^2}.
\]
Then, there exists $T:U^{p}_{\sigma_{1}}L^2\times \cdots \times U^{p}_{\sigma_{m}}L^2\rightarrow L^{p}_{t}(I ;L^{q}_{x}(\T^d ))$ satisfying
\[
\|\rho (t)T(u_{1},\cdots ,u_{m})\|_{L^{p}_{t}(I ;L^{q}_{x})}\lesssim \prod_{i=1}^{m}\|u_{i}\|_{U^{p}_{\sigma_{i}}L^2}
\]
such that $T(u_{1},\cdots ,u_{m})(t)(x)=T_{0}(u_{1}(t),\cdots ,u_{m}(t))(x)$ a.e.
\end{prop}

The original version of Proposition~\ref{multiest_T} (which is Proposition\ 2.19 in \cite{HHK09}) 
is given as $\rho (t)\equiv 1$.
By the same argument as in the case $\rho (t) \equiv 1$ in \cite{HHK09}, we can prove Proposition~\ref{multiest_T}.  

\begin{prop}[\cite{HHK09} Proposition\ 2.20]\label{intpol_T}
Let $q>1$, $E$ be a Banach space, and $T:U^{q}_{\sigma}L^2\rightarrow E$ be a bounded linear operator 
with $\|Tu\|_{E}\leq C_{q}\|u\|_{U^{q}_{\sigma}L^2}$ for all $u\in U^{q}_{\sigma}L^2$.  
In addition, assume that for some $1\leq p<q$ there exists $C_{p}\in (0,C_{q}]$ such that the estimate $\|Tu\|_{E}\leq C_{p}\|u\|_{U^{p}_{\sigma}L^2}$ holds true for all $u\in U^{p}_{\sigma}L^2$. Then, $T$ satisfies the estimate
\[
\|Tu\|_{E}\lesssim C_{p}\left( 1+\log \frac{C_{q}}{C_{p}}\right) \|u\|_{V^{p}_{\sigma}L^2}
\]
for $u\in V^{p}_{\sigma}L^2$,
where the implicit constant depends only on $p$ and $q$.
\end{prop}
Next, we define the function spaces which will be used to construct the solution. 
\begin{defn}\label{YZ_space_T}
Let $s$, $\sigma\in \R$.\\
{\rm (i)} We define $Z^{s}_{\sigma}$ as the space of all functions $u:\R \rightarrow H^{s}(\T^{d})$ such that for every $\xi \in \Z^{d}$ the map 
$t\mapsto e^{it\sigma |\xi |^{2}}\widehat{u(t)}(\xi )$ is in $U^{2}(\R ;\C )$, and for which the norm
\[
\|u\|_{Z^{s}_{\sigma}}:=\left(\sum_{\xi \in \Z^{d}}\langle \xi \rangle^{2s}\|e^{it\sigma |\xi|^{2}}\widehat{u(t)}(\xi )\|^{2}_{U^{2}(\R ;\C )}\right)^{\frac 12}
\]
is finite. \\
{\rm (ii)} We define $Y^{s}_{\sigma}$ as the space of all functions $u:\R \rightarrow H^{s}(\T^{d})$ such that for every $\xi \in \Z^{d}$ the map 
$t\mapsto e^{it\sigma |\xi |^{2}}\widehat{u(t)}(\xi )$ is in $V^{2}_{-,rc}(\R ;\C )$, and for which the norm
\[
\|u\|_{Y^{s}_{\sigma}}:=\left(\sum_{\xi \in \Z^{d}}\langle \xi \rangle^{2s}\|e^{it\sigma |\xi|^{2}}\widehat{u(t)}(\xi )\|^{2}_{V^{2}(\R ;\C )}\right)^{\frac 12}
\]
is finite. 
\end{defn}
\begin{rem}[\cite{HHK09} Remark\ 2.23]
\label{rem:ristI}
%Let $E$ be a Banach space of continuous functions $f:\R\rightarrow \HHH$, for some Hilbert space $\HHH$. 
We also consider the restriction space of $Z^s_{\sigma}$ to an interval $I\subset \R$ by
\[
Z^s_{\sigma}(I)=\{u\in C(I,H^s(\T^d))| \ \text{there exists } v \in Z^s_{\sigma}\ \text{such that}\ v(t)=u(t) \ (t\in I) \}
\]
endowed with the norm $\|u\|_{Z^s_{\sigma}(I)}=\inf \{\|v\|_{Z^s_{\sigma}}|\ v \in Z^s_{\sigma}, \, v(t)=u(t) \ (t\in I) \}$.
The restriction space $Y^s_{\sigma}(I)$ is also defined in the same way. 
\end{rem}
\begin{prop}[\cite{HTT11} Proposition\ 2.8, Corollary\ 2.9]\label{UVembedpr}
The embeddings
\[
U^{2}_{\sigma}H^{s}\hookrightarrow Z^{s}_{\sigma}\hookrightarrow Y^{s}_{\sigma}\hookrightarrow V^{2}_{\sigma}H^{s}
\] 
are continuous. Furthermore, if $\Z^{d}=\bigcup_{k \in \N} C_{k}$ be a partition of $\Z^{d}$, then
\begin{equation}
%\label{V_ortho_T}
\notag
\left(\sum_{k \in \N}\|P_{C_{k}}u\|^{2}_{V^{2}_{\sigma}H^{s}}\right)^{\frac 12}\lesssim \|u\|_{Y^{s}_{\sigma}}.
\end{equation}
\end{prop}
For $f\in L^{1}_{\rm loc}(\R ;L^{2}(\T^{d}))$ and $\sigma \in \R$, we define
\[
I_{\sigma}[f](t):=\int_{0}^{t}e^{i(t-t')\sigma \Delta}f(t')dt'
\]
for $t\geq 0$ and $I_{\sigma}[f](t)=0$ for $t<0$.
\begin{prop}[\cite{HTT11} Proposition\ 2.11]\label{duality_T}
For $s\ge 0$, $T>0$, $\sigma \in \R\backslash \{0\}$ and $f\in L^{1}([0,T);H^{s}(\T^d ))$ we have $I_{\sigma}[f]\in Z^{s}_{\sigma}([0,T))$ and
\[
\|I_{\sigma}[f]\|_{Z^{s}_{\sigma}([0,T))}\leq 
\sup_{v\in Y^{-s}_{\sigma}([0,T)), \|v\|_{Y^{-s}_{\sigma}}=1}\left|\int_{0}^{T}\int_{\T^{d}}f(t,x)\overline{v(t,x)} dxdt\right|. 
\]
\end{prop}
\section{Strichartz and bilinear Strichartz estimates \label{b_Stri_T}}
%
%%%%%%%%%%%%%%%%%%%%%%%%%%%%%%%%%%%%%%%%%%%%%%%%%%%%%%%%%%%%%%%%%%%%%%%%%%%%%%%%%
%%%%%%%%%%%%%%%%%%%%%%%%%%%%%%%%%%%%%%%%%%%%%%%%%%%%%%%%%%%%%%%%%%%%%%%%%%%%%%%%%
%%%%%%%%%%%%%%%%%%%%%%%%%%%%%%%%%%%%  Section 3   %%%%%%%%%%%%%%%%%%%%%%%%%%%%%%%%%%
%%%%%%%%%%%%%%%%%%%%%%%%%%%%%%%%%%%%%%%%%%%%%%%%%%%%%%%%%%%%%%%%%%%%%%%%%%%%%%%%%
%%%%%%%%%%%%%%%%%%%%%%%%%%%%%%%%%%%%%%%%%%%%%%%%%%%%%%%%%%%%%%%%%%%%%%%%%%%%%%%%%
%
%
In this section, we introduce some Strichartz estimates on tori proved in \cite{BD15}, \cite{HTT11}, \cite{Wa} and
the bilinear estimate proved in \cite{Wa}. 
%We note that
%\[
%\|e^{it\sigma \Delta}\varphi \|_{L^{p}([-1,1] \times \T^{d})}
%\le \|e^{it\sigma \Delta}\varphi \|_{L^{p}([-k\pi/\sigma,k\pi/\sigma] \times \T^{d})}
%=(k/\sigma)^{1/p}\|e^{it\Delta}\varphi \|_{L^{p}(\T \times \T^{d})}
%\]
%holds, where $k$ denotes the minimum integer such that $k\pi/\sigma \ge 1$. 
%Therefore, the Strichartz estimates for $e^{it\Delta}\varphi$ on $\T\times \T^{d}$
%can be rewritten as the estimates for $e^{it\sigma \Delta}\varphi$ on $[-1,1]\times \T^{d}$. 
We also show the bilinear Strichartz estimates (Proposition~\ref{beStest_T}).  

For a dyadic number $N\geq 1$, we define $\CC_{N}$ as the collection of 
disjoint cubes of the form
\[
\left(\xi_{0}+[-N,N]^{d}\right)\cap\Z^{d}
\]
with some $\xi_{0}\in \Z^{d}$. 

%disjoint cubes $C\subset \Z^{d}$ of side-length $N$
%with arbitrary center and orientation. 
 %
 First, we mention the $L^4$-Strichartz estimate for the one dimensional case. 
\begin{prop}[\cite{Bo93} Proposition 2.1]\label{L4stri_1d_T}
For $\sigma \in \R \backslash \{0\}$ and $0<T\le 1$, we have
\[
\|e^{it\sigma\Delta}\varphi\|_{L^4([0,T)\times \T)}\lesssim \|\varphi\|_{L^2(\T)}. 
\]
\end{prop}
Next, we give the Strichartz estimates for general settings. 
\begin{prop}[\cite{BD15} Theorem 2.4, Remark 2.5]\label{Stri_est_T}Let $d\ge 1$ and 
$\sigma \in \R \backslash \{0\}$. 
Assume 
\begin{alignat*}{3}
&s\ge \frac d2-\frac{d+2}p& &\quad{\rm if}\ p>\frac{2(d+2)}d,\\
&s>0& &\quad{\rm if}\ p= \frac{2(d+2)}d.
\end{alignat*}
{\rm (i)}\ For any $0<T\le 1$ and dyadic number $N\geq 1$, we have
\begin{equation}\label{Stri_1_T}
\|P_{N}e^{it\sigma \Delta}\varphi \|_{L^{p}([0,T) \times \T^{d})}\lesssim N^{s}\|P_{N}\varphi \|_{L^{2}(\T^{d})}.
\end{equation}
{\rm (ii)}\ For any $0<T\le 1$ and $C\in \CC_{N}$ with dyadic number $N\geq 1$, we have
\begin{equation}\label{Stri_2_T}
\|P_{C}e^{it\sigma \Delta}\varphi \|_{L^{p}([0,T) \times \T^{d})}\lesssim N^{s}\|P_{C}\varphi \|_{L^{2}(\T^{d})}.
\end{equation}
\end{prop}
\begin{rem}
(i)\ The estimate (\ref{Stri_2_T}) follows from (\ref{Stri_1_T}) and the Galilean transformation (see  (5.7) and (5.8) in \cite{Bo93}). \\
%(ii)\ Implicit constants in the estimates (\ref{Stri_1_T}),\ (\ref{Stri_2_T}) depend on $\sigma$. \\
(ii)\ The estimates (\ref{Stri_1_T}) and (\ref{Stri_2_T}) also hold for $s> 0$ and $1\le p<\frac{2(d+2)}d$ 
since the embedding $L^{\frac{2(d+2)}d}([0,T) \times \T^{d})\hookrightarrow L^{p}([0,T) \times \T^{d})$ holds for $1\le p<\frac{2(d+2)}d$. 
\end{rem}
For dyadic numbers $N\geq 1$ and $M\geq 1$, we define $\RR_{M}(N)$ as the 
collection of all sets of the form
\[
\left(\xi_{0}+[-N,N]^{d}\right)\cap\{\xi \in \Z^{d}|\ |a\cdot \xi -A|\leq M\}
\]
with some $\xi_{0}\in \Z^{d}$, $a\in \R^{d}$, $|a|=1$, and $A\in \R$. 
\begin{prop}[\cite{HTT11} Proposition 3.3, \cite{Wa} (3.4)]
Let $d\ge 1$ and $\sigma \in \R \backslash \{0\}$. For any $0<T\le 1$ and $R\in \RR_{M}(N)$ with 
dyadic numbers $N\geq M\geq 1$, we have
\begin{equation}\label{Stri_infi_T}
\|P_{R}e^{it\sigma \Delta}\varphi \|_{L^{\infty}([0,T) \times \T^{d})}\lesssim M^{\frac 12}N^{\frac{d-1}2}\|P_{R}\varphi \|_{L^{2}(\T^{d})}.
\end{equation}
\end{prop}
By using the H\"older inequality with (\ref{Stri_2_T}) for $p<4$ and (\ref{Stri_infi_T}), we have the following  $L^{4}$-Strichartz estimate. 
\begin{prop}
Let $d\ge 1$ and $\sigma \in \R \backslash \{0\}$. 
Assume 
\begin{alignat*}{3}
&s\ge s_c\ \Big( \!=\frac d2-1 \Big)& &\quad{\rm if}\  d\ge 3,\\
&s>0& &\quad{\rm if}\ d=1\ {\rm or}\ 2.  
\end{alignat*}
There exists $\delta >0$ such that
for any $0<T\le 1$ and $R\in \RR_{M}(N)$ with dyadic numbers $N\geq M\geq 1$, we have
\begin{equation}\label{Stri_3_T}
\|P_{R}e^{it\sigma \Delta}\varphi \|_{L^{4}([0,T) \times \T^{d})}\lesssim N^{\frac{s}{2}}\left(\frac{M}{N}\right)^{\delta}\|P_{R}\varphi \|_{L^{2}(\T^{d})}.
\end{equation}
\end{prop}
By Propositions~\ref{multiest_T}, ~\ref{L4stri_1d_T}, and ~\ref{Stri_est_T}, we have the followings:
\begin{cor}\label{U4_Stri_T_1d}
For $\sigma \in \R \backslash \{0\}$ and $0<T\le 1$, we have
\[
\|u\|_{L^4([0,T)\times \T)}\lesssim \|u\|_{U^4_{\sigma}L^2}. 
\]
\end{cor}
\begin{cor}\label{UV_Stri_T}
Let $\sigma \in \R \backslash \{0\}$.
Assume 
\begin{alignat*}{3}
&s\ge \frac d2-\frac{d+2}p& &\quad{\rm if}\  p>\frac{2(d+2)}d, \\
&s>0& &\quad{\rm if}\  1 \le p\le \frac{2(d+2)}d. 
\end{alignat*}
For any $0<T\le 1$, dyadic number $N\geq 1$, and $C\in \CC_{N}$, we have
\begin{align}
&\|P_{N}u\|_{L^{p}([0,T)\times \T^{d})}\lesssim N^{s}\|P_{N}u\|_{U_{\sigma}^{p}L^{2}},\label{U_Stri_T}\\
&\|P_{C}u\|_{L^{p}([0,T)\times \T^{d})}\lesssim N^{s}\|P_{C}u\|_{U_{\sigma}^{p}L^{2}}.
%\label{U_Stri_2_T}
\notag
\end{align}
\end{cor}
Next, we give the bilinear Strichartz estimates. 
Recall that $\eta$ is defined in \eqref{cutoff}.

\begin{prop}\label{beStest_T}
Let $d\ge 1$ and $\sigma_{1}$, $\sigma_{2}\in \R \backslash \{0\}$ with $\sigma_1+\sigma_2 \neq 0$. 
Assume 
\begin{alignat*}{3}
&s\ge s_c\ \Big( \! =\frac d2-1 \Big)& &\quad{\rm if}\  d\ge 3,\\
&s>0& &\quad{\rm if}\ d=1\ {\rm or}\ 2.  
\end{alignat*}
{\rm (i)}\ There exists $\delta >0$ such that 
for any 
dyadic numbers $H$, $L$ with $H\geq L\geq 1$, we have
\begin{equation}\label{beStest_T_1}
\begin{split}
&\|\eta (t)P_{H}(e^{it\sigma_{1} \Delta}\phi_{1})\cdot P_{L}(e^{it\sigma_{2} \Delta}\phi_{2})\|_{L^{2}(\R \times \T^{d})}\\
&\lesssim L^{s}\left(\frac{L}{H}+\frac{1}{L}\right)^{\delta}\|P_{H}\phi_{1}\|_{L^{2}(\T^{d})}
\|P_{L}\phi_{2}\|_{L^{2}(\T^{d})}. 
\end{split}
\end{equation}
{\rm (ii)}\ There exists $\delta>0$ such that 
for any dyadic numbers $L$, $H$, $H'$ with $H\sim H'\gg L\geq 1$, we have
\begin{equation}\label{beStest_T_2}
\begin{split}
&\|\eta (t)P_{L}[P_{H}(e^{it\sigma_{1} \Delta}\phi_{1})\cdot P_{H'}(e^{it\sigma_{2} \Delta}\phi_{2})]\|_{L^{2}(\R \times \T^{d})}\\
&\lesssim L^{s}\left(\frac{L}{H}+\frac{1}{L}\right)^{\delta}\|P_{H}\phi_{1}\|_{L^{2}(\T^{d})}
\|P_{H'}\phi_{2}\|_{L^{2}(\T^{d})}. 
\end{split}
\end{equation}
\end{prop}
\begin{rem}\label{bilin_eta_remark}
We note that $\eta$ defied in \eqref{cutoff} satisfies $\eta (t)\gtrsim 1$ for $0< t\le 1$. 
Therefore, we have
\[
\begin{split}
&\|P_{H}(e^{it\sigma_{1} \Delta}\phi_{1})\cdot P_{L}(e^{it\sigma_{2} \Delta}\phi_{2})\|_{L^{2}([0,T) \times \T^{d})}
\\
&\quad
\lesssim 
\|\eta (t)P_{H}(e^{it\sigma_{1} \Delta}\phi_{1})\cdot P_{L}(e^{it\sigma_{2} \Delta}\phi_{2})\|_{L^{2}(\R \times \T^{d})},\\
&\|P_{L}[P_{H}(e^{it\sigma_{1} \Delta}\phi_{1})\cdot P_{H'}(e^{it\sigma_{2} \Delta}\phi_{2})]\|_{L^{2}([0,T) \times \T^{d})}
\\
&\quad
\lesssim 
\|\eta (t)P_{L}[P_{H}(e^{it\sigma_{1} \Delta}\phi_{1})\cdot P_{H'}(e^{it\sigma_{2} \Delta}\phi_{2})]\|_{L^{2}(\R \times \T^{d})}
\end{split}
\]
for any $0<T\le 1$. 
\end{rem}
To prove Proposition~\ref{beStest_T}, 
we use the following lemma. 
\begin{lemm}\label{bilin_lemm_l4_eta_est}
Let $d\ge 1$ and $\sigma \in \R \backslash \{0\}$. 
Assume 
\begin{alignat*}{3}
&s\ge s_c\ \Big( \! =\frac d2-1 \Big)& &\quad{\rm if}\  d\ge 3,\\
&s>0& &\quad{\rm if}\ d=1\ {\rm or}\ 2.  
\end{alignat*}
There exists $\delta >0$ such that
for any $R\in \RR_{M}(N)$ with dyadic numbers $N\geq M\geq 1$, we have
\begin{equation}\label{Stri_3_T22}
\big\| \eta (t)^{\frac{1}{2}}P_{R}e^{it\sigma \Delta}\varphi \big\|_{L^{4}(\R \times \T^{d})}\lesssim N^{\frac{s}{2}}\left(\frac{M}{N}\right)^{\delta}\|P_{R}\varphi \|_{L^{2}(\T^{d})}.
\end{equation}
\end{lemm}
\begin{proof}
For $q\in \Z$, we put $I_q:=[q,q+1)$. Then we have
\begin{equation}\label{eta_pr_phi_decom}
\begin{split}
\big\| \eta (t)^{\frac{1}{2}}P_{R}e^{it\sigma \Delta}\varphi \big\|_{L^{4}(\R \times \T^{d})}^4
&=\sum_{q=-\infty}^{\infty} \big\| \eta (t)^{\frac{1}{2}}P_{R}e^{it\sigma \Delta}\varphi \big\|_{L^{4}(I_q \times \T^{d})}^4\\
&\le \sum_{q=-\infty}^{\infty}\|\eta (t)\|_{L^{\infty}_t(I_q)}^2
\|P_{R}e^{it\sigma \Delta}\varphi \|_{L^{4}(I_q \times \T^{d})}^4. 
\end{split}
\end{equation}
By changing variable $t\mapsto t+q$, it holds that
\[
\begin{split}
\|P_{R}e^{it\sigma \Delta}\varphi \|_{L^4(I_q\times \T^d)}
&=\left\|P_{R}\sum_{\xi \in \Z^d}e^{i\xi \cdot x}e^{-it\sigma|\xi|^2}\widehat{\varphi}(\xi)\right\|_{L^4(I_q\times \T^d)}\\
&=\left\|P_{R}\sum_{\xi \in \Z^d}e^{i\xi \cdot x}e^{-it\sigma|\xi|^2}e^{-iq\sigma|\xi|^2}\widehat{\varphi}(\xi)\right\|_{L^4([0,1)\times \T^d)}. 
\end{split}
\]
Therefore, by using (\ref{Stri_3_T}), we have
\[
\|P_{R}e^{it\sigma \Delta}\varphi \|_{L^4(I_q\times \T^d)}
\lesssim N^{\frac{s}{2}}\left(\frac{M}{N}\right)^{\delta}\left\|
P_R\sum_{\xi \in \Z^d}e^{i\xi \cdot x}e^{-iq\sigma|\xi|^2}\widehat{\varphi}(\xi)\right\|_{L^2_x}.
\]
Thanks to Parseval's identity, we obtain
\[
\left\|
P_R\sum_{\xi \in \Z^d}e^{i\xi \cdot x}e^{-iq\sigma|\xi|^2}\widehat{\varphi}(\xi)\right\|_{L^2_x}
\sim \left\|\left\{\ee_{R}(\xi)e^{-iq\sigma|\xi|^2}\widehat{\varphi}(\xi)\right\}_{\xi \in \Z^d}\right\|_{l^2_{\xi}}
\sim \|P_R\varphi\|_{L^2_x}
\]
for any $q\in \Z$. 
Therefore, we get
\begin{equation}\label{vl4q_est}
\sup_{q\in \Z}\|P_{R}e^{it\sigma \Delta}\varphi \|_{L^4(I_q\times \T^d)}^4
\lesssim N^{2s}\left(\frac{M}{N}\right)^{4\delta}\|P_R\varphi\|_{L^2_x}^4.
\end{equation}
On the other hand, it holds that
\begin{equation}\label{psisup_est}
\sum_{q=-\infty}^{\infty}\|\eta(t)\|_{L^{\infty}_t(I_q)}^2
=\sum_{q=-\infty}^{\infty}\left(\sup_{q\le t\le q+1}\frac{\sin \frac{\pi t}{2}}{\pi t}\right)^4
\lesssim \sum_{q=1}^{\infty}\frac{1}{q^4}<\infty. 
\end{equation}
The estimate  (\ref{Stri_3_T22}) follows from
(\ref{eta_pr_phi_decom}), (\ref{vl4q_est}), and (\ref{psisup_est}). 
\end{proof}
\begin{rem}
%By using
From
Proposition~\ref{Stri_est_T} and 
the same argument as in the proof of Lemma~\ref{bilin_lemm_l4_eta_est}, 
we also have
\[
\begin{split}
&\big\| \eta (t)^{\frac{1}{p}}P_Ne^{it\sigma \Delta}\varphi \big\|_{L^p(\R\times \T^d)}
\lesssim N^s\|P_N\varphi\|_{L^2(\T^d)},\\
&\big\| \eta (t)^{\frac{1}{p}}P_Ce^{it\sigma \Delta}\varphi \big\|_{L^p(\R\times \T^d)}
\lesssim N^s\|P_C\varphi\|_{L^2(\T^d)}
\end{split}
\]
for any dyadic number $N\ge 1$ and $C\in \mathcal{C}_N$, where
\begin{alignat*}{3}
&s\ge \frac d2-\frac{d+2}p& &\quad{\rm if}\  p>\frac{2(d+2)}d, \\
&s>0& &\quad{\rm if}\  1\le p\le \frac{2(d+2)}d. 
\end{alignat*}
Furthermore, by applying Proposition~\ref{multiest_T}, we obtain
\begin{align}
\big\| \eta (t)^{\frac{1}{p}}P_Nu \big\|_{L^p(\R\times \T^d)}
\lesssim N^s\|P_Nu\|_{U^p_{\sigma}L^2}, \label{eta_lp_up_est}\\
\big\| \eta (t)^{\frac{1}{p}}P_Cu \big\|_{L^p(\R\times \T^d)}
\lesssim N^s\|P_Cu\|_{U^p_{\sigma}L^2}.\label{eta_lp_up_est_C}
\end{align}
\end{rem}
\begin{proof}[Proof of {\rm Proposition~\ref{beStest_T}}]
We put $u_{j}=e^{it\sigma_{j}\Delta}\phi_{j}$ $(j=1,2)$. 
To prove (\ref{beStest_T_1}) and (\ref{beStest_T_2}), 
we use the argument in [\cite{HTT11} Proposition\ 3.5]. 
Because the proof of (\ref{beStest_T_1}) is simpler 
(decomposition for $u_2$ is not needed),
we only give the proof of (\ref{beStest_T_2}). 

We decompose $P_{H}u_{1}=\sum_{C_{1}\in \CC_{L}}P_{C_{1}}P_{H}u_{1}$. 
For fixed $C_{1}\in \CC_{L}$, let $\xi_{0}=\xi_0(C_1)$ be the center of $C_{1}$.
Note that $|\xi_0| \sim H$.
Since $\xi_{1}\in C_{1}$ and $|\xi_{1}+\xi_{2}|\leq 2L$ imply $|\xi_{2}+\xi_{0}|\leq 3L$, we obtain
\[
\|\eta (t)P_{L}(P_{C_{1}}P_{H}u_{1}\cdot P_{H'}u_{2})\|_{L^{2}(\R \times \T^{d})}
\leq \|\eta (t)P_{C_{1}}P_{H}u_{1}\cdot P_{C_{2}(C_{1})}P_{H'}u_{2}\|_{L^{2}(\R \times \T^{d})},
\] 
where $C_{2}(C_{1})$ is a cube contained in $\{\xi_{2}\in \Z^{d}|\ |\xi_{2}+\xi_{0}|\leq 3L\}$. 
If we prove
\begin{equation}\label{be_C_decom_T}
\begin{split}
&\|\eta (t)P_{C_{1}}P_{H}u_{1}\cdot P_{C_{2}(C_{1})}P_{H'}u_{2}\|_{L^{2}(\R \times \T^{d})}\\
&\lesssim L^{s}\left(\frac{L}{H}+\frac{1}{L}\right)^{\delta}\|P_{C_{1}}P_{H}\phi_{1}\|_{L^{2}(\T^d)}\|P_{C_{2}(C_{1})}P_{H'}\phi_{2}\|_{L^{2}(\T^d)}
\end{split}
\end{equation}
for some $\delta >0$, then we obtain
\[
\begin{split}
&\|\eta (t)P_{L}(P_{H}u_{1}\cdot P_{H'}u_{2})\|_{L^{2}(\R \times \T^{d})}\\
&\lesssim \sum_{C_{1}\in \CC_{L}}L^{s}\left(\frac{L}{H}+\frac{1}{L}\right)^{\delta}\|P_{C_{1}}P_{H}\phi_{1}\|_{L^{2}(\T^d)}\|P_{C_{2}(C_{1})}P_{H'}\phi_{2}\|_{L^{2}(\T^d)}\\
&\lesssim L^{s}\left(\frac{L}{H}+\frac{1}{L}\right)^{\delta}\left(\sum_{C_{1}\in \CC_{L}}\|P_{C_{1}}P_{H}\phi_{1}\|_{L^{2}(\T^d)}^{2}\right)^{\frac 12}
\left(\sum_{C_{1}\in \CC_{L}}\|P_{C_{2}(C_{1})}P_{H'}\phi_{2}\|_{L^{2}(\T^d)}^{2}\right)^{\frac 12}
\end{split}
\]
and the proof is completed.

Now, we prove the estimate (\ref{be_C_decom_T}) for some $\delta>0$.
Set
$M= \max \big\{ \frac{L^{2}}H,1 \big\}$ and
\[
\begin{split}
R_{1,k}&=
\left\{ \xi_{1}\in C_{1} \middle| \ \frac{(\xi_1-\xi_0) \cdot \xi_0}{|\xi_0|} \in [Mk, M(k+1)] \right\},
\\
R_{2,l}&=
\left\{\xi_{2}\in C_{2}(C_{1}) \middle| \ \frac{(\xi_2+\xi_0) \cdot \xi_0}{|\xi_0|} \in [Ml, M(l+1)] \right\}.
\end{split}
\]
Since
$\xi_0$ is the center of $C_1 \in \CC_{L}$,
the strip $R_{1,k}$ is not empty set if $|k| \lec \frac LM$.
Similarly,
$R_{2,l}$ is not empty set if $|l| \lec \frac LM$.
We decompose $C_{1}=\bigcup_{|k| \lesssim \frac LM}R_{1,k}$ and $C_{2}(C_{1})=\bigcup_{|l|\lesssim \frac LM}R_{2,l}$.
Therefore, we have
\begin{equation}\label{C1C2_R1R2_decom}
\begin{split}
P_{C_{1}}P_{H}u_{1}\cdot P_{C_{2}(C_{1})}P_{H'}u_{2}
&= \sum_{|k|, |l|\lec \frac LM}
P_{R_{1,k}}P_{H}u_{1}\cdot P_{R_{2,l}}P_{H'}u_{2}. 
\end{split}
\end{equation}
It follows from $\xi_{1}\in R_{1,k}$ that
%Because $L^{2}\lesssim M^{2}k$ and $M^2|m|^2\lesssim L^2$, we have 
\[
|\xi_{1}|^{2} - |\xi_0|^2
= 2 (\xi_1-\xi_0) \cdot \xi_0 + |\xi_1-\xi_0|^2
= 2 M|\xi_0|k+O( HM ).
\]
Similarly,
for $\xi_{2}\in R_{2,l}$,
we have
\[
|\xi_2|^{2} - |\xi_0|^2
= -2 (\xi_2+\xi_0) \cdot \xi_0 + |\xi_2+\xi_0|^2
= -2 M|\xi_0|l+O(HM).
\]
Hence, there exists a constant $A>0$ which is independent of 
$k$ and $l$ such that 
\begin{equation}\label{loc_time_freq}
\big|
\sigma_1 |\xi_1|^2+\sigma_2|\xi_2|^2
-
2M |\xi_0| (\sigma_1 k - \sigma_2 l)
- (\sigma_1+\sigma_2) |\xi_0|^2
\big|
\le A HM
\end{equation}
for $\xi_{1}\in R_{1,k}$ and $\xi_2\in R_{2,l}$.

Set
\[
\begin{split}
F_{k,l}(\tau,\xi)&:=\F[\eta (t)P_{R_{1,k}}P_{H}u_{1}\cdot P_{R_{2,l}}P_{H'}u_{2}](\tau,\xi)\\
&\ =\sum_{\xi_1+\xi_2=\xi}
\widehat{\eta}(\tau +\sigma_1|\xi_1|^2+\sigma_2|\xi_2|^2)
\F_x[P_{R_{1,k}}P_{H}\phi_{1}](\xi_1)\F_x[P_{R_{2,l}}P_{H'}\phi_{2}](\xi_2). 
\end{split}
\]
A direct calculation with \eqref{cutoff} yields that
\begin{equation}
\widehat{\eta}(\tau)=(\ee_{[-\frac{\pi}{2},\frac{\pi}{2}]}*\ee_{[-\frac{\pi}{2},\frac{\pi}{2}]})(\tau).
\label{etaft1}
\end{equation}
It follows from \eqref{loc_time_freq} that
\begin{equation}
{\rm supp}F_{k,l}
\subset
\left\{
(\tau,\xi)
\in \R \times \Z^d
\, \middle| \,
\begin{aligned}
&\big| \tau + 2M |\xi_0| (\sigma_1k-\sigma_2l)+ (\sigma_1+\sigma_2) |\xi_0|^2\big|
\le
2 A HM,
\\
& \xi \cdot \frac{\xi_0}{|\xi_0|} \in [ M(k+l) , M(k+l+2)]
\end{aligned}
\right\}.
\label{suppFkl}
\end{equation}
Then,
there exists a constant $A' >0$ which is independent of 
$k$, $l$, $k'$, and $l'$ such that 
\begin{equation}
{\rm supp}F_{k,l}\cap {\rm supp}F_{k',l'}=\emptyset\ \
\label{alorF}
\end{equation}
holds if $|k-k'| + |l-l'| \ge A'$.
Indeed,
by \eqref{suppFkl},
we have \eqref{alorF} if
%\[
%|(\sigma_1k-\sigma_2l) - (\sigma_1k'-\sigma_2l')| \ge A \frac H{|\xi_0|}
%\ \text{ or } \
%|(k+l) - (k'+l')| \ge 4.
%\]
\[
|\sigma_1(k-k')-\sigma_2(l-l')| \ge 4 A \frac H{|\xi_0|}
\ \text{ or } \
|(k-k') + (l-l')| \ge 4.
\]
From $\sigma_1 + \sigma_2 \neq 0$ and $|\xi_0| \sim H$,
this condition is equivalent to
$|k-k'| + |l-l'| \ge A'$
for some $A'>0$.

It follows from \eqref{alorF} that
%This implies the almost orthogonality
\begin{equation}\label{almost_ortho_est}
\begin{split}
\bigg\|\sum_{|k|, |l|\lec \frac LM}F_{k,l}(\tau,\xi) \bigg\|_{L^2_{\tau}l^2_{\xi}}^2
\lesssim \sum_{|k|, |l|\lec \frac LM}\left\|F_{k,l}(\tau,\xi)\right\|_{L^2_{\tau}l^2_{\xi}}^2.
\end{split}
\end{equation}
By (\ref{C1C2_R1R2_decom}) and (\ref{almost_ortho_est}), we have
\[
\|\eta (t)P_{C_{1}}P_{H}u_{1}\cdot P_{C_{2}(C_{1})}P_{H'}u_{2}\|_{L^{2}(\R \times \T^{d})}
\lesssim
\bigg( \sum_{|k|, |l| \lesssim \frac LM} \left\|F_{k,l}(\tau,\xi)\right\|_{L^2_{\tau}l^2_{\xi}}^2 \bigg)^{\frac 12}.
\]
Recall that 
$M= \max \big\{ \frac{L^{2}}H,1 \big\}$,
$R_{1,k}\in \RR_M(L)$, and $R_{2,l}\in \RR_M(3L)$.
The H\"older inequality and Lemma~\ref{bilin_lemm_l4_eta_est} yield that
\[
\begin{split}
\left\|F_{k,l}(\tau,\xi)\right\|_{L^2_{\tau}l^2_{\xi}}
&\lesssim \|\eta (t)P_{R_{1,k}}P_{H}u_{1}\cdot P_{R_{2,l}}P_{H'}u_{2}\|_{L^2(\R\times \T^d)}\\
&\le \big\| \eta (t)^{\frac{1}{2}}P_{R_{1,k}}P_{H}u_{1} \big\|_{L^4(\R\times \T^d)}
\big\| \eta (t)^{\frac{1}{2}}P_{R_{2,l}}P_{H'}u_{2} \big\|_{L^4(\R\times \T^d)}\\
&\lesssim L^{s}\left(\frac{L}{H}+\frac{1}{L}\right)^{\delta}\|P_{R_{1,k}}P_H\phi_1\|_{L^2_x}\|P_{R_{2,l}}P_{H'}\phi_2\|_{L^2_x}.
\end{split}
\]
Therefore, we obtain (\ref{be_C_decom_T}). 
\end{proof}
\begin{rem}\label{bilin_Stri_gene_bilin_op}
By the same argument in the proof of Proposition~\ref{beStest_T}, 
we can obtain 
\[
\begin{split}
&\|\eta (t)R_{L}[P_{H}(e^{it\sigma_{1}}\phi_{1}),P_{H'}(e^{it\sigma_{2}}\phi_{2})]\|_{L^{2}(\R \times \T^{d})}\\
&\lesssim L^{s}\left(\frac{L}{H}+\frac{1}{L}\right)^{\delta}\|P_{H}\phi_{1}\|_{L^{2}(\T^{d})}
\|P_{H'}\phi_{2}\|_{L^{2}(\T^{d})}, 
\end{split}
\]
where $R_L$ is a bilinear operator defined by
\[
\F_x[R_L(u_1,u_2)](\xi)=\sum_{\substack{\xi_1,\xi_2\in \Z^d\\ \xi_1+\xi_2=\xi}}\psi_L(a\xi_1+b\xi_2)\widehat{u_1}(\xi_1)\widehat{u_2}(\xi_2)
\]
for $a,b\in \R \backslash \{0\}$. 
\end{rem}

\begin{rem}
\label{rem:irrm1}
If we consider the estimate on the irrational tori $\T^d_\theta$,
$R_{1,k}$ and $R_{2,l}$ are replaced with
\[
\begin{split}
R_{1,k}&=
\left\{ \xi_{1}\in C_{1} \middle| \ \sum_{j=1}^d \theta_j^2 \frac{(\xi_{1,j}-\xi_{0,j}) \cdot \xi_{0,j}}{|\xi_{0}|} \in [Mk, M(k+1)] \right\},
\\
R_{2,l}&=
\left\{\xi_{2}\in C_{2}(C_{1}) \middle| \ \sum_{j=1}^d \theta_j^2 \frac{(\xi_{2,j}+\xi_{0,j}) \cdot \xi_{0,j}}{|\xi_{0}|} \in [Ml, M(l+1)] \right\},
\end{split}
\]
where $\xi_{m,j}$ denotes the $j$-th component of $\xi_m$ for $m=0,1,2$.
Hence, \eqref{loc_time_freq} is replaced with
\[
\Big|
\sum_{j=1}^d
\theta_j^2
(\sigma_1 |\xi_{1,j}|^2+\sigma_2|\xi_{2,j}|^2)
-
2M |\xi_{0}| (\sigma_1 k - \sigma_2 l)
- (\sigma_1+\sigma_2)
\sum_{j=1}^d
\theta_j^2
|\xi_{0,j}|^2
\Big|
\le A HM.
\]
With straightforward modifications,
the same calculation as in the proof works well.
\end{rem}
From Proposition~\ref{intpol_T},
we have the following. 
\begin{prop}\label{Ybe_2_T}
Let $d\ge 1$ and $\sigma_{1}$, $\sigma_{2}\in \R \backslash \{0\}$ with $\sigma_1+\sigma_2 \neq 0$.
Assume 
\begin{alignat*}{3}
&s\ge s_c\ \Big( \! =\frac d2-1 \Big)& &\quad{\rm if}\  d\ge 3,\\
&s>0& &\quad{\rm if}\ d=1\ {\rm or}\ 2.  
\end{alignat*}
{\rm (i)}\ There exists $\delta>0$ such that for any dyadic numbers $H$ and $L$ with $H\geq L\geq 1$, we have
\begin{equation}
\label{Ybe_HL_T}
\|\eta (t)P_{H}u_{1}\cdot P_{L}u_{2}\|_{L^{2}(\R\times \T^{d})}\lesssim L^{s}\left(\frac{L}{H}+\frac{1}{L}\right)^{\delta}\|P_{H}u_{1}\|_{Y^{0}_{\sigma_{1}}}\|P_{L}u_{2}\|_{Y^{0}_{\sigma_{2}}}.
\end{equation}
{\rm (ii)}\ There exists $\delta>0$ such that for any dyadic numbers $L$, $H$, and $H'$ with $H\sim H'\gg L\geq 1$, we have
\begin{equation}\label{Ybe_HHL_T}
\|\eta (t)P_L(P_{H}u_{1}\cdot P_{H'}u_{2})\|_{L^{2}(\R \times \T^{d})}\lesssim L^{s}\left(\frac{L}{H}+\frac{1}{L}\right)^{\delta}\|P_{H}u_{1}\|_{Y^{0}_{\sigma_{1}}}\|P_{H'}u_{2}\|_{Y^{0}_{\sigma_{2}}}.
\end{equation}
\end{prop}
\begin{proof}
We only give the proof of \eqref{Ybe_HHL_T},
since a slight modification yields \eqref{Ybe_HL_T}.
Proposition~\ref{multiest_T}
with the bilinear Strichartz estimate (\ref{beStest_T_2}) 
(see, also Remark~\ref{bilin_eta_remark})
yields that
\begin{equation}
\label{be_U2_est_2}
\begin{split}
\|\eta (t)P_{L}(P_{H}u_{1}\cdot P_{H'}u_{2})\|_{L^{2}(\R \times \T^{d})}
&\lesssim L^{s}\left(\frac{L}{H}+\frac{1}{L}\right)^{\delta}\|P_{H}u_{1}\|_{U^{2}_{\sigma_{1}}L^{2}}\|P_{H'}u_{2}\|_{U^{2}_{\sigma_{2}}L^{2}}
\end{split}
\end{equation}
for any $0<T\le 1$. 
On the other hand,
by the H\"older inequality and \eqref{eta_lp_up_est_C}, we obtain
\begin{equation}
\label{be_U4_est_2}
\begin{split}
\|\eta (t)P_{L}(P_{H}u_{1}\cdot P_{H'}u_{2})\|_{L^{2}(\R \times \T^{d})}
&\lesssim L^{s}\|P_{H}u_{1}\|_{U^{4}_{\sigma_{1}}L^{2}}\|P_{H'}u_{2}\|_{U^{4}_{\sigma_{2}}L^{2}}
\end{split}
\end{equation}
for any $0<T\le 1$.
It follows from
Proposition~\ref{intpol_T} with \eqref{be_U2_est_2} and \eqref{be_U4_est_2} that
\[
\begin{split}
\|\eta (t)P_{L}(P_{H}u_{1}\cdot P_{H'}u_{2})\|_{L^{2}(\R \times \T^{d})}
&\lesssim L^{s}\left(\frac{L}{H}+\frac{1}{L}\right)^{\delta}\|P_{H}u_{1}\|_{V^{2}_{\sigma_{1}}L^{2}}\|P_{H'}\phi_{2}\|_{V^{2}_{\sigma_{2}}L^{2}}
\end{split}
\]
for some $\delta >0$.
Therefore, we get \eqref{Ybe_HHL_T}
by  the embedding $Y^0_{\sigma}\hookrightarrow V^2_{\sigma}L^2$\ (see, Proposition~\ref{UVembedpr}). 
\end{proof}
%
%

%
%\begin{rem}
%Since  the bilinear Stricartz estimates (\ref{beStest_T_1}) and 
%(\ref{beStest_T_2}) are not used in the proof of Proposition~\ref{Ybe_ex_loc_T}, 
%the condition $\sigma_{1}+\sigma_{2}\neq 0$ 
%is not necessary for (\ref{Ybe_ext_loc_T}). 
%\end{rem}
%
%%%%%%%%%%%%%%%%%%%%%%%%%%%%%%%%%%%%%%%%%%%%%%%%%%%%%%%%%%%%%%%%%%%%%%%%%%%%%%%%%
%%%%%%%%%%%%%%%%%%%%%%%%%%%%%%%%%%%%%%%%%%%%%%%%%%%%%%%%%%%%%%%%%%%%%%%%%%%%%%%%%
%%%%%%%%%%%%%%%%%%%%%%%%%%%%%%%%%%%%  Section 4   %%%%%%%%%%%%%%%%%%%%%%%%%%%%%%%%%%
%%%%%%%%%%%%%%%%%%%%%%%%%%%%%%%%%%%%%%%%%%%%%%%%%%%%%%%%%%%%%%%%%%%%%%%%%%%%%%%%%
%%%%%%%%%%%%%%%%%%%%%%%%%%%%%%%%%%%%%%%%%%%%%%%%%%%%%%%%%%%%%%%%%%%%%%%%%%%%%%%%%
%
\section{Trilinear estimates\label{tri_est_T}}
In this section, we give the trilinear estimates 
which will be used to prove the well-posedness. 
Set
\begin{align}
\mus
&:= \sigma_1 \sigma_2 \sigma_3 \Big( \frac1{\sigma_1} + \frac 1{\sigma_2} + \frac 1{\sigma_3} \Big)
\label{mus},
\\
\kks
&:= (\sigma_{1}+\sigma_{2})(\sigma_{2}+\sigma_{3})(\sigma_{3}+\sigma_{1})
.
\label{kks}
\end{align}
We first give a lemma related to a nonresonance condition. 
\begin{lemm}\label{modul_est_T}
Let $d\ge 1$ and $\sigma_{1}$, $\sigma_{2}$, $\sigma_{3} \in \R \backslash \{0\}$. 
We assume that  $\tau_0\in \R$
and $(\tau_{1},\xi_{1})$, $(\tau_{2}, \xi_{2})$, $(\tau_{3}, \xi_{3})\in \R\times \R^{d}$ satisfy $\tau_0+\tau_{1}+\tau_{2}+\tau_{3}=0$ and $\xi_{1}+\xi_{2}+\xi_{3}=0$.\\ 
{\rm (i)} Let $(i,j,k)$ be a permutation of $(1,2,3)$ and 
assume $\sigma_i+\sigma_j\ne 0$.
If $|\xi_i|\sim |\xi_j|\gg |\xi_k|$ holds, then there exists 
$C_0>0$, which is independent of $\{\tau_k\}_{k=0}^3$ and $\{\xi_k\}_{k=1}^3$, such that
\begin{equation}\label{modulation_est}
|\tau_0|+\max_{1\leq j\leq 3}|\tau_{j}+\sigma_{j}|\xi_{j}|^{2}|
\ge C_0\max_{1\leq j\leq 3}|\xi_{j}|^{2}. 
\end{equation}
{\rm (ii)} Assume $\mus> 0$. 
If $|\xi_{1}|\sim |\xi_{2}|\sim |\xi_{3}|$ holds, then we have {\rm (\ref{modulation_est})}. 
\end{lemm}
The proof of this lemma is same as Lemma\ 4.1 in \cite{Hi}. 
\begin{rem}\label{resonance_cond_remark41}
{\rm (i)}\ If $\mus\ge 0$, 
then $\kks \neq 0$ holds. 
In particular, (\ref{modulation_est}) always holds when $\mus> 0$. \\
{\rm (ii)}\ Under the condition $\kks \neq 0$, 
Lemma~\ref{modulation_est} {\rm (i)} says that 
{\rm (\ref{modulation_est})} holds unless $|\xi_1|\sim |\xi_2|\sim |\xi_3|$.
\end{rem}
To obtain the well-posedness, 
we need the estimates for the integral
\begin{equation}\label{integral_tri_need}
\left|\int_{0}^{T}\int_{\T^{d}}\left(\prod_{j=1}^{3}P_{N_{j}}u_{j}\right)dxdt\right|. 
\end{equation}
Because $\eta$ defined in \eqref{cutoff} satisfies $\eta (t)\ge \frac 1{\pi^2}$ on $[0,1]$,
for $0<T\le 1$,
there exists $\psi_T\in C_0^{\infty}(\R)$ such that 
\begin{equation}
\eta (t)\psi_T(t)^3=1
\label{defpsi}
\end{equation}
on $[0,T]$. 
Therefore, the integral (\ref{integral_tri_need}) is controlled by
\begin{equation}\label{integral_tri_need_T}
\left|\int_{\R}\int_{\T^{d}}\eta (t)\left(\prod_{j=1}^{3}\psi_T(t)\ee_{[0,T)}(t)P_{N_{j}}u_{j}\right)dxdt\right|. 
\end{equation}
We will give the estimate for the integral (\ref{integral_tri_need_T}) 
instead of (\ref{integral_tri_need}). 
\begin{lemm}
Let $0<T\le 1$, $\sigma \in \R\backslash \{0\}$, and $f\in Y^{0}_{\sigma}$. 
Then, we have
\begin{equation}\label{Y-norm_T_est}
\|\psi_T\ee_{[0,T)}f\|_{Y^{0}_{\sigma}}\lesssim \|f\|_{Y^{0}_{\sigma}}. 
\end{equation}
\end{lemm}
\begin{proof}
We note that 
\[
\|\psi_T\ee_{[0,T)}\|_{V^2(\R;\C)}
= \big\| \eta^{-\frac{1}{3}}\ee_{[0,T)} \big\|_{V^2(\R;\C)}
\lesssim \eta (0)^{-\frac{1}{3}}-\eta (T)^{-\frac{1}{3}}
\le  \eta (0)^{-\frac{1}{3}}-\eta (1)^{-\frac{1}{3}}\lesssim 1
\]
holds for any $T>0$ because $\eta$ is positive and decreasing on $[0,1]$. 
Therefore, we get
\begin{equation}\label{cut_func_V2_proper}
\|e^{it\sigma |\xi|^2}\psi_T(t)\ee_{[0,T)}(t)\widehat{f(t)}(\xi)\|_{V^2(\R;\C)}\lesssim 
\|e^{it\sigma |\xi|^2}\widehat{f(t)}(\xi)\|_{V^2(\R;\C)}
\end{equation}
by the algebra-type property (see, Lemma\ B.14 in \cite{KT18})
\[
\|FG\|_{V^2(\R;L^2)}\le \|F\|_{L^{\infty}(\R;L^2)}\|G\|_{V^{2}(\R;L^2)}
+\|F\|_{V^{2}(\R;L^2)}\|G\|_{L^{\infty}(\R;L^2)}
\]
and the embedding $V^2(\R;L^2) \hookrightarrow L^{\infty}(\R ;L^2)$.  
The desired estimate (\ref{Y-norm_T_est}) follows from (\ref{cut_func_V2_proper}). 
\end{proof}

Throughout of this section, 
we put
\begin{align*}
\displaystyle
&N_{\max}:=\max_{1\leq j\leq 3}N_{j}, \quad \displaystyle N_{\min}:=\min_{1\leq j\leq 3}N_{j}, \\ &u_{j,T}:=\psi_T\ee_{[0,T)}P_{N_j}u_j \quad \ (j=1,2,3).
\end{align*}
 
\begin{rem}\label{lowfreq_trilin_est}
If $N_{\max}\lesssim 1$, we obtain
\[
\begin{split}
&\left| N_{\max}\int_{\R}\int_{\T^{d}}\eta(t)\left(\prod_{j=1}^{3}P_{N_{j}}u_{j,T}\right)dxdt\right|\\
&\lesssim \|\ee_{[0,T)}\|_{L^2}\|P_{N_{1}}u_{1}\|_{L^{\infty}([0,T);L^2(\T^d))} \|P_{N_{2}}u_{2}\|_{L^4([0,T)\times \T^d)} \|P_{N_{3}}u_{3}\|_{L^4([0,T)\times \T^d)}\\
&\lesssim T^{\frac{1}{2}}\prod_{j=1}^3\|P_{N_j}u_{j}\|_{V^2_{\sigma_j}L^2}
\lesssim T^{\frac{1}{2}}\prod_{j=1}^3\|P_{N_j}u_{j}\|_{Y^0_{\sigma_j}}
\end{split}
\]
by the H\"older inequality, (\ref{U_Stri_T}) with $p=4$, and
%the embeddings 
$Y_{\sigma_j}^0\hookrightarrow V^{2}_{\sigma_j}L^{2}\hookrightarrow L^{\infty}(\R ;L^{2}(\T^d ))$. 
Therefore, we only consider $N_{\max}\gg 1$ in the following argument.
\end{rem}

We divide the integral (\ref{integral_tri_need_T})
into $8$ pieces of the form 
\begin{equation}\label{integral_divide_8piece}
\int_{\R}\int_{\T^d}\eta (t)\left(\prod_{j=1}^{3}Q_{j}^{\sigma_j}P_{N_{j}}u_{j,T}\right) dxdt
\end{equation}
with $Q_{j}^{\sigma_j }\in \{Q_{\geq M}^{\sigma_j }, Q_{<M}^{\sigma_j }\}$\ $(j=1, 2, 3)$. 
\begin{lemm}\label{nonreso_case_tri_est}
Let $d\ge 3$ and $\sigma_{1}$, $\sigma_{2}$ $\sigma_3 \in \R \backslash \{0\}$
satisfy 
$\kks \neq 0$. 
Assume $Q_j^{\sigma_j}=Q_{\ge M}^{\sigma_j}$ for some $j\in \{1,2,3\}$. 
Then, there exists $\delta >0$ such that 
for any $0<T\le 1$,
dyadic numbers $N_{1}$, $N_{2}$, $N_{3}$, $M\geq 1$ with $M\sim N_{\max}^2\gg 1$, 
and $P_{N_{j}}u_{j}\in V^{2}_{\sigma_{j}}L^{2}$\ $(j=1,2,3)$, we have
\begin{equation}\label{nonreso_est_trilin}
\begin{split}
&\left| N_{\max}\int_{\R}\int_{\T^{d}}\eta(t)\left(\prod_{j=1}^{3}Q_{j}^{\sigma_j}P_{N_{j}}u_{j,T}\right)dxdt\right|\\
&\lesssim 
N_{\min}^{s_c}\left(\frac{N_{\min}}{N_{\max}}+\frac{1}{N_{\min}}\right)^{\delta}\prod_{j=1}^{3}\|P_{N_{j}}u_{j}\|_{Y^{0}_{\sigma_{j}}}.
\end{split}
\end{equation}
\end{lemm}
\begin{proof}
We only consider the case $Q_{3}^{\sigma_3}=Q_{\ge M}^{\sigma_3}$ 
because the other cases can be treated in the same way. 
By the Cauchy-Schwarz inequality, we have
\[
\begin{split}
J&:=\left|\int_{\R}\int_{\T^d}
\eta (t)\left(Q_{1}^{\sigma_1}P_{N_{1}}u_{1,T}Q_{2}^{\sigma_2}P_{N_{2}}u_{2,T}Q_{\ge M}^{\sigma_3}P_{N_{3}}u_{3,T}\right) dxdt\right|\\
&\lesssim \|\eta (t)\widetilde{P}_{N_{3}}(Q_{1}^{\sigma_{1}}P_{N_{1}}u_{1,T}\cdot Q_{2}^{\sigma_{2}}P_{N_{2}}u_{2,T})\|_{L^{2}(\R\times \T^{d})}
\|Q_{\geq M}^{\sigma_{3}}P_{N_{3}}u_{3,T}\|_{L^{2}(\R \times \T^{d})},
\end{split}
\]
where $\widetilde{P}_{N_{3}} :=P_{\frac{N_{3}}2}+P_{N_{3}}+P_{2N_{3}}$. 
Furthermore, by (\ref{highMproj_T}), $M\sim N_{\max}^{2}$, the embedding $Y^{0}_{\sigma_{3}}\hookrightarrow V^{2}_{\sigma_{3}}L^{2}$, and (\ref{Y-norm_T_est}), we have
\begin{equation}\label{high_mod_deriv_gain}
\|Q_{\geq M}^{\sigma_{3}}P_{N_{3}}u_{3,T}\|_{L^{2}(\R\times \T^{d})}
\lesssim N_{\max}^{-1}\|P_{N_{3}}u_{3,T}\|_{V^{2}_{\sigma_{3}}L^{2}}\lesssim N_{\max}^{-1}\|P_{N_{3}}u_{3}\|_{Y^{0}_{\sigma_{3}}}.
\end{equation}
On the other hand, by Proposition~\ref{Ybe_2_T}, (\ref{Vproj_T}), and (\ref{Y-norm_T_est}), we have
\[
\begin{split}
&\|\eta (t)\widetilde{P}_{N_{3}}(P_{N_{1}}Q_{1}^{\sigma_{1}}u_{1,T}\cdot P_{N_{2}}Q_{2}^{\sigma_{2}}u_{2})\|_{L^{2}(\R\times \T^{d})}\\
&\lesssim N_{\min}^{s_c}\left(\frac{N_{\min}}{N_{\max}}+\frac{1}{N_{\min}}\right)^{\delta}\|P_{N_{1}}u_{1}\|_{Y^{0}_{\sigma_{1}}}\|P_{N_{2}}u_{2}\|_{Y^{0}_{\sigma_{2}}}.
\end{split}
\]
Therefore, we obtain
\[
\begin{split}
N_{\max}J
&\lesssim N_{\min}^{s_c}\left(\frac{N_{\min}}{N_{\max}}+\frac{1}{N_{\min}}\right)^{\delta}\prod_{j=1}^{3}\|P_{N_{j}}u_{j}\|_{Y^{0}_{\sigma_{j}}}. 
\end{split}
\]
%The estimates for $J_{32}$ and $J_{33}$ can be obtained in the same way as for $J_{31}$. 
\end{proof}
\begin{lemm}\label{nonreso_case_tri_est_loc}
Let $d\ge 1$, $\sigma_{1}$, $\sigma_{2}$, $\sigma_3\in \R \backslash \{0\}$, 
 and $s>\max\{s_c,0\}$. 
Assume $Q_j^{\sigma_j}=Q_{\ge M}^{\sigma_j}$ for some $j\in \{1,2,3\}$. 
Then there exist $\delta >0$ and $\eps >0$ such that 
for any $0<T\le 1$,
dyadic numbers $N_{1}$, $N_{2}$, $N_{3}$, $M\geq 1$ with $M\sim N_{\max}^2\gg 1$, 
and $P_{N_{j}}u_{j}\in V^{2}_{\sigma_{j}}L^{2}$\ $(j=1,2,3)$, we have
\begin{equation}\label{nonreso_est_trilin_loc}
\begin{split}
&\left| N_{\max}\int_{\R}\int_{\T^{d}}\eta(t)\left(\prod_{j=1}^{3}Q_{j}^{\sigma_j}P_{N_{j}}u_{j,T}\right)dxdt\right|\\
&\lesssim 
T^{\eps}N_{\min}^{s}\left(\frac{N_{\min}}{N_{\max}}+\frac{1}{N_{\min}}\right)^{\delta}\prod_{j=1}^{3}\|P_{N_{j}}u_{j}\|_{Y^{0}_{\sigma_{j}}}.
\end{split}
\end{equation}
\end{lemm}
\begin{proof}
We only consider the case $Q_{3}^{\sigma_3}=Q_{\ge M}^{\sigma_3}$ 
because the other cases can be treated in the same way. 
We decompose 
\[
Q_1^{\sigma_1}P_{N_1}u_{1,T}\cdot Q_2^{\sigma_2}P_{N_2}u_{2,T}
=\sum_{C_1\in \mathcal{C}_{N_{\min}}}
Q_1^{\sigma_1}P_{C_1}P_{N_1}u_{1,T}\cdot Q_2^{\sigma_2}P_{C_1(C_2)}P_{N_2}u_{2,T}
\]
as in the proof of  {\rm Proposition~\ref{beStest_T}}. 
By using the H\"older inequality, (\ref{eta_lp_up_est_C}) with $p=4$, 
the embedding $V^{2}_{\sigma_j}L^{2}\hookrightarrow U^{4}_{\sigma_j}L^{2}$, 
and (\ref{Vproj_T}), we have
\[
\begin{split}
&\|\eta (t)Q_1^{\sigma_1}P_{N_1}u_{1,T}\cdot Q_2^{\sigma_2}P_{N_2}u_{2,T}\|_{L^2(\R\times \T^d)}\\
&\lesssim N_{\min}^{\max\{s_c,a\}}
\sum_{C_1\in \mathcal{C}_{N_{\min}}}\|P_{C_1}P_{N_1}u_{1,T}\|_{V^2_{\sigma_1}L^2}
\|P_{C_1(C_2)}P_{N_2}u_{2,T}\|_{V^2_{\sigma_2}L^2}
\end{split}
\]
for any $a>0$. 
Therefore, by the Schwarz inequality, Proposition~\ref{UVembedpr}, 
and (\ref{Y-norm_T_est}), we obtain
\[
\|\eta (t)Q_1^{\sigma_1}P_{N_1}u_{1,T}\cdot Q_2^{\sigma_1}P_{N_2}u_{2,T}\|_{L^2(\R\times \T^d)}
\lesssim N_{\min}^{\max\{s_c,a\}}
\|P_{N_1}u_{1}\|_{Y_{\sigma_1}^0}
\|P_{N_2}u_{2}\|_{Y_{\sigma_2}^0}.
\]
This and (\ref{high_mod_deriv_gain}) imply that
\begin{equation}\label{bilin_est_by_l4stri}
\left| N_{\max}\int_{\R}\int_{\T^{d}}\eta(t)\left(\prod_{j=1}^{3}Q_{j}^{\sigma_j}P_{N_{j}}u_{j,T}\right)dxdt\right|
\lec
N_{\min}^{\max\{s_c,a\}}
\prod_{j=1}^{3}\|P_{N_{j}}u_{j}\|_{Y^{0}_{\sigma_{j}}}.
\end{equation}
On the other hand, the H\"older inequality, (\ref{high_mod_deriv_gain}), and
the Bernstein inequality yield that
\begin{equation}
\left| N_{\max}\int_{\R}\int_{\T^{d}}\eta(t)\left(\prod_{j=1}^{3}Q_{j}^{\sigma_j}P_{N_{j}}u_{j,T}\right)dxdt\right|
\lec
T^{\frac{1}{2}}
N_{\min}^{\frac d2}
\prod_{j=1}^{3}\|P_{N_{j}}u_{j}\|_{Y^{0}_{\sigma_{j}}}. 
\label{bilin_est_by_l4strib}
\end{equation}
When $d\ge 3$,
we have $s_c>0$.
By interpolating \eqref{bilin_est_by_l4strib} and \eqref{bilin_est_by_l4stri} with $a= \frac{s_c}2$,
we obtain 
\[
\left| N_{\max}\int_{\R}\int_{\T^{d}}\eta(t)\left(\prod_{j=1}^{3}Q_{j}^{\sigma_j}P_{N_{j}}u_{j,T}\right)dxdt\right|
\lec
T^{\eps}
N_{\min}^{s_c+2\eps}
\prod_{j=1}^{3}\|P_{N_{j}}u_{j}\|_{Y^{0}_{\sigma_{j}}}
\]
for any $0<\eps <\frac{1}{2}$. 
By choosing $0<\eps <\frac{1}{2}$ and $\delta >0$ 
such that $\eps<\frac{s-s_c}{2}$ and $\delta =s-s_c-2\eps$, we get
\eqref{nonreso_est_trilin_loc}. 
The cases $d=1,2$ can be treated in the same manner. 
\end{proof}
\subsection{Nonresonance case}

We give the trilinear estimates under the condition
\[
\mus>0.
\]
Note that this condition implies 
$\kks \neq 0$.
\begin{prop}\label{HL_est_T}
Let $d\geq 3$ and $\sigma_{1}$, $\sigma_{2}$, $\sigma_{3}\in \R \backslash \{ 0\}$ satisfy
$\mus>0$. 
There exists $\delta >0$ such that 
for any $0<T\le 1$,
dyadic numbers $N_{1}$, $N_{2}$, $N_{3}\geq 1$ with $N_{\max}\gg 1$, 
and $P_{N_{j}}u_{j}\in V^{2}_{\sigma_{j}}L^{2}$\ $(j=1,2,3)$, we have
\begin{equation}\label{hl_T}
\begin{split}
&\left| N_{\max}\int_{\R}\int_{\T^{d}}\eta(t)\left(\prod_{j=1}^{3}P_{N_{j}}u_{j,T}\right)dxdt\right|
\\
&\lesssim 
N_{\min}^{s_c}\left(\frac{N_{\min}}{N_{\max}}+\frac{1}{N_{\min}}\right)^{\delta}\prod_{j=1}^{3}\|P_{N_{j}}u_{j}\|_{Y^{0}_{\sigma_{j}}}.
\end{split}
\end{equation}
\end{prop}
\begin{proof} 
For sufficiently large constant $C$ 
(for example, $C= \frac{32}{C_0}$, 
where $C_0$ is given in Lemma~\ref{modul_est_T}), 
we put 
$M:=C^{-1}N_{\max}^{2}$ and 
%
%We divide the integral on the left-hand side of (\ref{hl_T}) into 
%\[
%\begin{split}
%\int_{0}^{T}\int_{\T^{d}}\left(\prod_{j=1}^{3}P_{N_{j}}u_{j}\right)dxdt
%=\int_{\R}\int_{\T^{d}}\left(\prod_{j=1}^{3}P_{N_{j}}u_{j,T}\right)dxdt
%=J_1 +J_2 + J_{31}+J_{32}+J_{33},
%\end{split}
%\]
%where
%\begin{equation}\label{division}
%\begin{split}
%&J_1 =\int_{\R}\int_{\T^{d}}\left(\prod_{j=1}^{3}Q_{<M}^{\sigma_j }P_{N_{j}}u_{j,T}\right)dxdt,\ 
%J_2=\int_{\R}\int_{\T^{d}}\left(\prod_{j=1}^{3}Q_{\ge M}^{\sigma_j }P_{N_{j}}u_{j,T}\right)dxdt,\\
%&J_{31}=\int_{\R}\int_{\T^{d}}\left(Q_{\ge M}^{\sigma_1 }P_{N_{1}}u_{1,T}\right)\left( Q_{<M}^{\sigma_2 }P_{N_{2}}u_{2,T}\right)P_{N_{3}}u_{3,T}dxdt,\\
%&J_{32}=\int_{\R}\int_{\T^{d}}P_{N_{1}}u_{1,T}\left(Q_{\ge M}^{\sigma_2 }P_{N_{2}}u_{2,T}\right)\left( Q_{<M}^{\sigma_3}P_{N_{3}}u_{3,T}\right)dxdt,\\
%&J_{33}=\int_{\R}\int_{\T^{d}}\left( Q_{<M}^{\sigma_1 }P_{N_{1}}u_{1,T}\right)P_{N_{2}}u_{2,T}\left(Q_{\ge M}^{\sigma_3 }P_{N_{3}}u_{3,T}\right)dxdt.
%\end{split}
%\end{equation}
%By Plancherel's theorem, Lemma~\ref{modul_est_T} {\rm (ii)} and $N_{\max}\ge 2$, we have
%\[
%J_1=0.
%\]
%
divide the integral (\ref{integral_tri_need_T})
into $8$ pieces of the form such as (\ref{integral_divide_8piece}). 
By Plancherel's theorem, Lemma~\ref{modul_est_T} (and Remark~\ref{resonance_cond_remark41}), and $N_{\max}\gg 1$, we have
\[
\int_{\R}\int_{\T^d}\eta (t)\left(\prod_{j=1}^{3}Q_{<M}^{\sigma_j}P_{N_{j}}u_{j,T}\right) dxdt=0
\]
because $\tau_0\in {\rm supp}\hspace{0.5ex}\widehat{\eta}$ satisfies $|\tau_0|\le \pi\ (\ll N_{\max}^2)$. 
Therefore, we can assume 
at least one of $Q_j^{\sigma_j}$ is equal to $Q_{\ge M}^{\sigma_j}$ 
and obtain (\ref{hl_T}) by Lemma~\ref{nonreso_case_tri_est}.
\end{proof}
We also obtain the following local estimate 
by using Lemma~\ref{nonreso_case_tri_est_loc}. 
\begin{prop}\label{HL_est_loc_T}
Let $d\geq 1$ and $\sigma_{1}$, $\sigma_{2}$, $\sigma_{3}\in \R \backslash \{ 0\}$ satisfy 
$\mus>0$. 
Assume $s>\max\{s_{c},0\}$. 
There exist $\delta>0$ and  $\eps >0$ such that for any $0<T\le 1$,
dyadic numbers $N_{1}$, $N_{2}$, $N_{3}\geq 1$ with $N_{\max}\gg 1$, 
and $P_{N_{j}}u_{j}\in V^{2}_{\sigma_{j}}L^{2}$\ $(j=1,2,3)$, we have
\begin{equation}\label{hl_loc_T}
\begin{split}
&\left| N_{\max}\int_{\R}\int_{\T^{d}}\eta(t)\left(\prod_{j=1}^{3}P_{N_{j}}u_{j,T}\right)dxdt\right|
\\
&\lesssim 
T^{\eps}N_{\min}^{s}\left(\frac{N_{\min}}{N_{\max}}+\frac{1}{N_{\min}}\right)^{\delta}\prod_{j=1}^{3}\|P_{N_{j}}u_{j}\|_{Y^{0}_{\sigma_{j}}}.
\end{split}
\end{equation}
\end{prop}

\subsection{Resonance case I}
In this subsection, we give the trilinear estimates under the condition
$\mus= 0$.
Note that this condition implies 
$\kks \neq 0$.
\begin{prop}\label{HL_est_T_mass_reso}
Let $d\geq 4$ and $\sigma_{1}$, $\sigma_{2}$, $\sigma_{3}\in \R \backslash \{ 0\}$ satisfy 
$\mus = 0$. 
There exists $\delta >0$ such that for any $0<T\le 1$, 
dyadic numbers $N_{1}$, $N_{2}$, $N_{3}\geq 1$ with $N_{\max}\gg 1$, and 
$P_{N_{j}}u_{j}\in V^{2}_{\sigma_{j}}L^{2}$\ $(j=1,2,3)$, we have {\rm (\ref{hl_T})}. 
\end{prop}
\begin{proof}
For sufficiently large constant $C$, 
we put 
$M:=C^{-1}N_{\max}^{2}$ and 
divide the integral (\ref{integral_tri_need_T})
into $8$ pieces of the form such as (\ref{integral_divide_8piece}). 
Thanks to Lemma~\ref{nonreso_case_tri_est}, 
it suffices to consider the case $Q_j^{\sigma_j}=Q_{<M}^{\sigma_j}$\ $(j=1,2,3)$. 
By Plancherel's theorem and Lemma~\ref{modul_est_T} {\rm (i)} 
(see also Remark~\ref{resonance_cond_remark41}), we have
\[
\int_{\R}\int_{\T^d}\eta (t)\left(\prod_{j=1}^{3}Q_{<M}^{\sigma_j}P_{N_{j}}u_{j,T}\right) dxdt=0
\]
unless $N_1\sim N_2\sim N_3$. 
Therefore, we only have to consider the case $N_1\sim N_2\sim N_3$. 

We decompose
\[
\prod_{j=1}^{3}Q_{<M}^{\sigma_j}P_{N_{j}}u_{j,T}
=\sum_{1\le M_1,M_2,M_3<M}\prod_{j=1}^{3}Q_{M_j}^{\sigma_j}P_{N_{j}}u_{j,T}. 
\]
By the symmetry, 
we can assume $\max_{1\le j\le 3}M_j=M_3$. 
Then, it suffices to prove the estimate for the integral
\[
\sum_{M_3<M}
\int_{\R}\int_{\T^d}\eta (t)(Q_{\le M_3}^{\sigma_1}P_{N_{1}}u_{1,T})(Q_{\le M_3}^{\sigma_2}P_{N_{2}}u_{2,T})(Q_{M_3}^{\sigma_3}P_{N_{3}}u_{3,T})dxdt. 
\]
By \eqref{mus},
the conditions $\mus = 0$ and $\xi_1+\xi_2+\xi_3=0$ imply that
\[
|\sigma_1 |\xi_1|^2+\sigma_2|\xi_2|^2+\sigma_3|\xi_3|^2|
= \Big| \frac{\sigma_3}{\sigma_1\sigma_2} \Big| |\sigma_1\xi_1-\sigma_2\xi_2|^2.
\]
On the other hand, $\tau_0\in {\rm supp}$\hspace{0.5ex}$\widehat{\eta}$, 
$(\tau_j,\xi_j)\in {\rm supp}$\hspace{0.5ex}$\F[Q_{\le M_3}^{\sigma_j}P_{N_{j}}u_{j,T}]$\ $(j=1,2)$, 
and $(\tau_3,\xi_3)\in {\rm supp}$\hspace{0.5ex}$\F[Q_{M_3}^{\sigma_3}P_{N_{3}}u_{3,T}]$
with $\tau_0+\tau_1+\tau_2+\tau_3=0$, $\xi_1+\xi_2+\xi_3=0$ satisfy
\[
|\sigma_1 |\xi_1|^2+\sigma_2|\xi_2|^2+\sigma_3|\xi_3|^2|
\le |\tau_0|+\sum_{j=1}^3|\tau_j+\sigma_j|\xi_j|^2|\lesssim M_3. 
\]
Therefore, we have
\[
|\sigma_1\xi_1-\sigma_2\xi_2|\lesssim  M_3^{\frac{1}{2}}.
\]
By the same argument in the proof of Proposition~\ref{beStest_T} 
(see, also Remark~\ref{bilin_Stri_gene_bilin_op}), 
we obtain 
\[
\begin{split}
&\|\eta (t)(Q_{\le M_3}^{\sigma_1}P_{N_{1}}u_{1,T})(Q_{\le M_3}^{\sigma_2}P_{N_{2}}u_{2,T})\|_{L^2(\R\times \T^d)}\\
&\lesssim 
M_{3}^{\frac{s_c}{2}}
\left(\frac{M_3^{\frac{1}{2}}}{N_{\max}}+\frac{1}{M_{3}^{\frac{1}{2}}}
\right)^{\delta}\prod_{j=1}^{2}\|P_{N_{j}}u_{j,T}\|_{Y^{0}_{\sigma_{j}}}.
\end{split}
\]
Furthermore, by (\ref{highMproj_Tequal}) and the embedding $Y^{0}_{\sigma_{1}}\hookrightarrow V^{2}_{\sigma_{1}}L^{2}$,  we have
\[
\|Q_{M_3}^{\sigma_{3}}P_{N_{3}}u_{3,T}\|_{L^{2}(\R\times \T^{d})}
\lesssim M_{3}^{-\frac{1}{2}}\|P_{N_{3}}u_{3,T}\|_{V^{2}_{\sigma_{3}}L^{2}}\lesssim 
M_{3}^{-\frac{1}{2}}\|P_{N_{3}}u_{3,T}\|_{Y^{0}_{\sigma_{1}}}.
\]
By these estimates with the H\"older inequality and (\ref{Y-norm_T_est}), we obtain
\begin{equation}
%\label{mass_reso_trilin_key_est}
\notag
\begin{split}
&\left|N_{\max}\sum_{M_3<M}
\int_{\R}\int_{\T^d}\eta (t)(Q_{\le M_3}^{\sigma_1}P_{N_{1}}u_{1,T})(Q_{\le M_3}^{\sigma_2}P_{N_{2}}u_{2,T})(Q_{M_3}^{\sigma_3}P_{N_{3}}u_{3,T})dxdt\right|\\
&\lesssim N_{\max}\sum_{M_3<M}
M_{3}^{\frac{s_c-1}{2}}
\left(\frac{M_3^{\frac{1}{2}}}{N_{\max}}+\frac{1}{M_{3}^{\frac{1}{2}}}
\right)^{\delta}\prod_{j=1}^{3}\|P_{N_{j}}u_{j,T}\|_{Y^{0}_{\sigma_{j}}}. 
\end{split}
\end{equation}
This estimate and $M\sim N_{\max}^2$ imply {\rm (\ref{hl_T})}
because $s_c\ge 1$ for $d\ge 4$, and it holds
\[
\sum_{M_3<M}M_3^{\frac{s_c-1}{2}}
\left(\frac{M_3^{\frac{1}{2}}}{N_{\max}}+\frac{1}{M_{3}^{\frac{1}{2}}}
\right)^{\delta}
\lesssim M^{\frac{s_c-1}{2}}
\left\{
\left(\frac{M^{\frac{1}{2}}}{N_{\max}}\right)^{\delta}+1
\right\}
\lesssim N_{\max}^{s_c-1}.
\]
%Since $N_1 \sim N_2 \sim N_3$,
%we obtain the desired bound.
\end{proof}
Note that
\[
\|Q_{M_3}^{\sigma_{3}}P_{N_{3}}u_{3,T}\|_{L^{2}(\R\times \T^{d})}
\lesssim T^{\frac 12} \|P_{N_{3}}u_{3,T}\|_{Y^{0}_{\sigma_{1}}}.
\]
By interpolating this estimate and \eqref{highMproj_T},
it holds that
\[
\|Q_{M_3}^{\sigma_{3}}P_{N_{3}}u_{3,T}\|_{L^{2}(\R\times \T^{d})}
\lesssim T^{\eps}M_{3}^{-\frac{1}{2}+\eps}\|P_{N_{3}}u_{3,T}\|_{Y^{0}_{\sigma_{1}}}
\]
for any $0<\eps<\frac 12$. 
By using this estimate in the proof of Proposition~\ref{HL_est_T_mass_reso}, 
we have
\[
\begin{split}
&\left|N_{\max}\sum_{M_3<M}
\int_{\R}\int_{\T^d}\eta (t)(Q_{\le M_3}^{\sigma_1}P_{N_{1}}u_{1,T})(Q_{\le M_3}^{\sigma_2}P_{N_{2}}u_{2,T})(Q_{M_3}^{\sigma_3}P_{N_{3}}u_{3,T})dxdt\right|\\
&\lesssim T^{\eps}N_{\max}\sum_{M_3<M}
M_{3}^{\frac{s_c-1}{2}+\eps}
\left(\frac{M_3^{\frac{1}{2}}}{N_{\max}}+\frac{1}{M_{3}^{\frac{1}{2}}}
\right)^{\delta}\prod_{j=1}^{3}\|P_{N_{j}}u_{j,T}\|_{Y^{0}_{\sigma_{j}}} 
\end{split}
\]
for $d\ge 1$. We note that $\frac{s_c-1}2 \le -\frac 14 <0$ if $1\le d\le 3$. 
Therefore, by choosing $\eps >0$
such that
$\eps = \min \{ \frac{s-s_c}2, \frac 18 \}$, 
we obtain the following. 
\begin{prop}\label{HL_est_loc_T_22}
Let $d\geq 1$ and $\sigma_{1}$, $\sigma_{2}$, $\sigma_{3}\in \R \backslash \{ 0\}$ satisfy 
$\mus = 0$. 
Assume $s>s_c$ and $s\ge 1$.  
There exist $\delta>0$ and $\eps>0$ such that
for any $0<T\le 1$, 
dyadic numbers $N_{1}$, $N_{2}$, $N_{3}\geq 1$ with $N_{\max}\gg 1$, 
and $P_{N_{j}}u_{j}\in V^{2}_{\sigma_{j}}L^{2}$\ $(j=1,2,3)$, we have {\rm (\ref{hl_loc_T})}. 
\end{prop}
\subsection{Resonance case II}\label{reso_1_subsec}
We give the trilinear estimates under the condition
\[
\mus < 0, \quad (\sigma_2+\sigma_3)(\sigma_3+\sigma_1)\ne 0,
\]
where $\mus$ is defined in \eqref{mus}.
In this subsection, we do not consider the case 
$d=1,2$ and $s=1$, and these cases will be treated in the next subsection.

First, we show the trilinear estimate under a stronger condition $\kks \neq 0$,
where $\kks$ is defined in \eqref{kks}.

\begin{prop}\label{HL_est_T_2}
Let $d\geq 5$ and $\sigma_{1}$, $\sigma_{2}$, $\sigma_{3}\in \R \backslash \{ 0\}$ satisfy 
$\mus< 0$ and 
$\kks \neq 0$. 
There exists $\delta >0$ such that for any $0<T\le 1$, 
dyadic numbers $N_{1}$, $N_{2}$, $N_{3}\geq 1$ with $N_{\max}\gg 1$, and 
$P_{N_{j}}u_{j}\in V^{2}_{\sigma_{j}}L^{2}$\ $(j=1,2,3)$, we have {\rm (\ref{hl_T})}. 
\end{prop}
\begin{proof}
We set $M = C^{-1} N_{\max}^2$ for some $C \gg 1$.
Because of $\kks \neq 0$, 
by a similar reason in the proof of Proposition~\ref{HL_est_T_mass_reso}, 
it suffices to show the estimate for the integral
\[
\int_{\R}\int_{\T^d}\eta (t)\left(\prod_{j=1}^{3}Q_{<M}^{\sigma_j}P_{N_{j}}u_{j,T}\right) dxdt
\]
with $N_1 \sim N_2\sim N_3$. 
By the H\"older inequality, we have
\begin{align*}
&\left| N_{\max}\int_{\R}\int_{\T^{d}}\eta (t)\left(\prod_{j=1}^{3}Q_{<M}^{\sigma_j}P_{N_{j}}u_{j,T}\right)dxdt\right|
\\
&\leq N_{\max}\prod_{j=1}^{3}\|\eta (t)^{\frac{1}{3}}Q_{<M}^{\sigma_j}P_{N_j}u_{j,T}\|_{L^3 (\R\times \T^{d})}. 
\end{align*}
Furthermore, by (\ref{eta_lp_up_est}) with $p=3$, 
the embeddings $Y^{0}_{\sigma_{j}}\hookrightarrow V^{2}_{\sigma_{j}}L^{2}\hookrightarrow U^{3}_{\sigma_{j}}L^{2}$,  (\ref{Vproj_T}), and (\ref{Y-norm_T_est}), 
we have
\[
\big\| \eta (t)^{\frac{1}{3}}Q_{<M}^{\sigma_j}P_{N_j}u_{j,T} \big\|_{L^3 (\R\times \T^{d})}\lesssim N_j^{\frac d6- \frac 23}\|P_{N_j}u_{j}\|_{Y^{0}_{\sigma_j}}
\]
since $3>\frac{2(d+2)}d$ holds for $d\ge 5$. Therefore, we obtain 
\[
\begin{split}
\left| N_{\max}\int_{\R}\int_{\T^{d}}\eta (t)\left(\prod_{j=1}^{3}Q_{<M}^{\sigma_j}P_{N_{j}}u_{j,T}\right)dxdt\right|
&\lesssim 
N_{\min}^{s_c}\prod_{j=1}^{3}\|P_{N_{j}}u_{j}\|_{Y^{0}_{\sigma_{j}}}
\end{split}
\]
because 
\[
N_{\max}\prod_{j=1}^3N_j^{\frac d6- \frac 23}\sim N_{\max}^{s_c}\sim N_{\min}^{s_c}.
\]
\end{proof}
\begin{prop}\label{HL_est_loc_T_reso}
Let $d\ge 1$, and $\sigma_{1}$, $\sigma_{2}$, $\sigma_{3}\in \R \backslash \{ 0\}$ satisfy 
$\mus<0$ 
and $(\sigma_2+\sigma_3)(\sigma_3+\sigma_1)\ne 0$. 
Assume $s>\max\{s_c,1\}$. 
There exist $\delta >0$ and $\eps>0$ such that
for any $0<T\le 1$,  
dyadic numbers $N_{1}$, $N_{2}$, $N_{3}\geq 1$ with $N_{\max}\gg 1$, 
and $P_{N_{j}}u_{j}\in V^{2}_{\sigma_{j}}L^{2}$\ $(j=1,2,3)$, we have
\begin{equation}\label{hl_loc_T_reso}
\begin{split}
&\left| N_{3}\int_{\R}\int_{\T^d}\eta (t)\left(\prod_{j=1}^{3}P_{N_{j}}u_{j,T}\right)dxdt\right|
\\
&\lesssim 
T^{\eps}N_{\min}^{s}\left(\frac{N_{\min}}{N_{\max}}+\frac{1}{N_{\min}}\right)^{\delta}\prod_{j=1}^{3}\|P_{N_{j}}u_{j}\|_{Y^{0}_{\sigma_{j}}}.
\end{split}
\end{equation} 
\end{prop}
\begin{proof}For sufficiently large constant $C$, 
we put 
$M:=C^{-1}N_{\max}^{2}$ and 
divide the integral (\ref{integral_tri_need_T})
into $8$ pieces of the form such as (\ref{integral_divide_8piece}). 
Thanks to Lemma~\ref{nonreso_case_tri_est_loc}, 
it suffices to consider the case $Q_j^{\sigma_j}=Q_{<M}^{\sigma_j}$\ $(j=1,2,3)$. 
Because $(\sigma_2+\sigma_3)(\sigma_3+\sigma_1)\ne 0$, 
by Plancherel's theorem and Lemma~\ref{modul_est_T} {\rm (i)}, we have
\[
\int_{\R}\int_{\T^d}\eta (t)\left(\prod_{j=1}^{3}Q_{<M}^{\sigma_j}P_{N_{j}}u_{j,T}\right) dxdt=0
\]
if $N_2\sim N_3\gg N_1$ or $N_3\sim N_1\gg N_2$ holds. 
Therefore, we only have to consider the case $N_1\sim N_2\gtrsim N_3$. 
In the same way as in the proof of Proposition~\ref{beStest_T}, 
we decompose $P_{N_1}u_{1}=\sum_{C_{1}\in \CC_{N_3}}P_{C_{1}}P_{N_1}u_{1}$. 
For fixed $C_{1}\in \CC_{N_3}$, let $\xi_{0}=\xi_0(C_1)$ be the center of $C_{1}$. 
Since $\xi_{1}\in C_{1}$ and $|\xi_{1}+\xi_{2}|\leq 2N_3$ imply $|\xi_{2}+\xi_{0}|\leq 3N_3$, we obtain
\begin{equation}
%\label{decom_lp_ident}
\notag
\begin{split}
&\big\|\eta (t)^{\frac{1}{q}}P_{N_3}(Q_{<M}^{\sigma_1}P_{C_{1}}P_{N_1}u_{1,T}\cdot Q_{<M}^{\sigma_2}P_{N_2}u_{2,T}) \big\|_{L^q(\R\times \T^d)}\\
&= \big\|\eta (t)^{\frac{1}{q}} P_{N_3}(Q_{<M}^{\sigma_1}P_{C_{1}}P_{N_1}u_{1,T}\cdot Q_{<M}^{\sigma_2}P_{C_{2}(C_{1})}P_{N_2}u_{2,T}) \big\|_{L^q(\R\times \T^d)}
\end{split}
\end{equation}
for $q\ge 1$, 
where $C_{2}(C_{1})$ is a cube contained in $\{\xi_{2}\in \Z^{d}|\ |\xi_{2}+\xi_{0}|\leq 3N_3\}$. 

We first assume $1\le d\le 4$ and $s>1$. 
In this case, we choose $q= \frac 32$. 
By the H\"older inequality, 
(\ref{eta_lp_up_est_C}) with $p=3$,  the embeddings
$Y^0_{\sigma}\hookrightarrow V^{2}_{\sigma}L^{2}\hookrightarrow U^{3}_{\sigma}L^{2}$, 
(\ref{Vproj_T}), and (\ref{Y-norm_T_est}), we have
\begin{equation}\label{23bilinest_int_1}
\begin{split}
&\big\|\eta (t)^{\frac{2}{3}} P_{N_3}(Q_{<M}^{\sigma_1}P_{C_{1}}P_{N_1}u_{1,T}\cdot Q_{<M}^{\sigma_2}P_{C_{2}(C_{1})}P_{N_2}u_{2,T}) \big\|_{L^{\frac{3}{2}}(\R\times \T^d)}\\
&\lesssim \big\|\eta (t)^{\frac{1}{3}}Q_{<M}^{\sigma_1}P_{C_{1}}P_{N_1}u_{1,T} \big\|_{L^{3}(\R\times \T^d)}
\big\|\eta (t)^{\frac{1}{3}}Q_{<M}^{\sigma_2}P_{C_{2}(C_1)}P_{N_2}u_{2,T} \big\|_{L^{3}(\R\times \T^d)}\\
&\lesssim N_3^{2a}\|P_{C_{1}}P_{N_1}u_{1}\|_{Y_{\sigma_1}^0}\|P_{C_{2}(C_1)}P_{N_2}u_{2}\|_{Y_{\sigma_2}^0}
\end{split}
\end{equation}
for any $a>0$ because $3\le \frac{2(d+2)}d$ holds for $1\le d\le 4$. 
On the other hand, by the boundedness of $\eta (t)$ and the embeddings
$Y^0_{\sigma}\hookrightarrow V^{2}_{\sigma}L^{2}\hookrightarrow L^{\infty}(\R;L^{2}(\T^d))$ ,  
we obtain
\[
\big\|\eta (t)^{\frac{1}{2}}Q_{<M}^{\sigma_1}P_{C_{1}}P_{N_1}u_{1,T}\big\|_{L^{2}(\R\times \T^d)}
 \lesssim \|P_{C_{1}}P_{N_1}u_{1,T}\|_{L^{2}(\R\times \T^d)}
 \lesssim T^{\frac{1}{2}}\|P_{C_{1}}P_{N_1}u_{1,T}\|_{Y_{\sigma_1}^0}. 
\]
Therefore, 
by the H\"older inequality, 
(\ref{eta_lp_up_est_C}) with $p=6$,  the embeddings
$Y^0_{\sigma}\hookrightarrow V^{2}_{\sigma}L^{2}\hookrightarrow U^{6}_{\sigma}L^{2}$, 
(\ref{Vproj_T}), and (\ref{Y-norm_T_est}), we have
\begin{equation}\label{23bilinest_int_2}
\begin{split}
&\big\|\eta (t)^{\frac{2}{3}} P_{N_3}(Q_{<M}^{\sigma_1}P_{C_{1}}P_{N_1}u_{1,T}\cdot Q_{<M}^{\sigma_2}P_{C_{2}(C_{1})}P_{N_2}u_{2,T}) \big\|_{L^{\frac{3}{2}}(\R\times \T^d)}\\
&\lesssim \big\|\eta (t)^{\frac{1}{2}}Q_{<M}^{\sigma_1}P_{C_{1}}P_{N_1}u_{1,T} \big\|_{L^{2}(\R\times \T^d)}
\big\|\eta (t)^{\frac{1}{6}}Q_{<M}^{\sigma_2}P_{C_{2}(C_1)}P_{N_2}u_{2,T}\big\|_{L^{6}(\R\times \T^d)}\\
&\lesssim T^{\frac{1}{2}}N_3^{\frac{d-1}{3}}\|P_{C_{1}}P_{N_1}u_{1}\|_{Y_{\sigma_1}^0}\|P_{C_{2}(C_1)}P_{N_2}u_{2}\|_{Y_{\sigma_2}^0}
\end{split}
\end{equation}
for $2\le d\le 4$
because $6>\frac{2(d+2)}d$ and $\frac d2- \frac{d+2}6= \frac{d-1}3$ hold. 
By the interpolation between (\ref{23bilinest_int_1}) and (\ref{23bilinest_int_2}), 
it holds that
\[
\begin{split}
&\big\|\eta (t)^{\frac{2}{3}} P_{N_3}(Q_{<M}^{\sigma_1}P_{C_{1}}P_{N_1}u_{1,T}\cdot Q_{<M}^{\sigma_2}P_{C_{2}(C_{1})}P_{N_2}u_{2,T}) \big\|_{L^{\frac{3}{2}}(\R\times \T^d)}\\
&\lesssim T^{\eps}N_3^{2a+\left(\frac{d-1}{3}-2a\right)\eps}\|P_{C_{1}}P_{N_1}u_{1}\|_{Y_{\sigma_1}^0}\|P_{C_{2}(C_1)}P_{N_2}u_{2}\|_{Y_{\sigma_2}^0}
\end{split}
\]
for any $a>0$ and $0<\eps <\frac 12$. 
By using this estimate, (\ref{eta_lp_up_est}) with $p=3$, the embeddings
$Y^0_{\sigma}\hookrightarrow V^{2}_{\sigma}L^{2}\hookrightarrow U^{3}_{\sigma}L^{2}$, 
(\ref{Vproj_T}), and (\ref{Y-norm_T_est}), we obtain
\[
\begin{split}
&\left| N_{3}\int_{\R}\int_{\T^d}\eta (t)\left(\prod_{j=1}^{3}Q_{<M}^{\sigma_j}P_{N_{j}}u_{j,T}\right)dxdt\right|\\
&\le N_3\sum_{C_1\in \CC_{N_3}}
\big\|\eta (t)^{\frac{2}{3}} P_{N_3}(Q_{<M}^{\sigma_1}P_{C_{1}}P_{N_1}u_{1,T}\cdot Q_{<M}^{\sigma_2}P_{C_{2}(C_{1})}P_{N_2}u_{2,T}) \big\|_{L^{\frac{3}{2}}(\R\times \T^d)}\\
&\hspace{34ex}\times \big\|\eta (t)^{\frac{1}{3}}Q_{<M}^{\sigma_3}P_{N_3}u_{3,T} \big\|_{L^3(\R\times \T^d)}\\
&\lesssim T^{\eps}N_3^{1+3a+\left(\frac{d-1}{3}-2a\right)\eps}\prod_{j=1}^3\|P_{N_{j}}u_{j}\|_{Y^0_{\sigma_j}}
\end{split}
\]
for any $a>0$ and $0<\eps <\frac 12$. 
For $2\le d\le 4$ and $s> \max \{ s_c,1 \}$, by choosing $a>0$ and $0<\eps<\frac 12$ such that
$0<a< \min\{ \frac{d-1}6, \frac{s-1}3 \}$ and $3a+(\frac{d-1}{3}-2a)\eps<s-1$,
we get (\ref{hl_loc_T_reso}) with $\delta =s-1-3a-(\frac{d-1}{3}-2a)\eps$. 
Note that for $d=1$, we have
\begin{equation}
%\label{23bilinest_int_2_d1}
\notag
\begin{split}
&\big\|\eta (t)^{\frac{2}{3}} P_{N_3}(Q_{<M}^{\sigma_1}P_{C_{1}}P_{N_1}u_{1,T}\cdot Q_{<M}^{\sigma_2}P_{C_{2}(C_{1})}P_{N_2}u_{2,T}) \big\|_{L^{\frac{3}{2}}(\R\times \T^d)}\\
&\lesssim T^{\frac{1}{2}}N_3^{a}\|P_{C_{1}}P_{N_1}u_{1}\|_{Y_{\sigma_1}^0}\|P_{C_{2}(C_1)}P_{N_2}u_{2}\|_{Y_{\sigma_2}^0}
\end{split}
\end{equation}
for any $a>0$
by the same calculation in (\ref{23bilinest_int_2}) because $6=\frac{2(d+2)}d$ holds for $d=1$. 
By using this estimate, (\ref{eta_lp_up_est}) with $p=3$, the embeddings
$Y^0_{\sigma}\hookrightarrow V^{2}_{\sigma}L^{2}\hookrightarrow U^{3}_{\sigma}L^{2}$, 
(\ref{Vproj_T}), and (\ref{Y-norm_T_est}), we obtain
\[
\begin{split}
&\left| N_{3}\int_{\R}\int_{\T}\eta (t)\left(\prod_{j=1}^{3}Q_{<M}^{\sigma_j}P_{N_{j}}u_{j,T}\right)dxdt\right|\\
&\le N_3\sum_{C_1\in \CC_{N_3}}
\big\|\eta (t)^{\frac{2}{3}} P_{N_3}(Q_{<M}^{\sigma_1}P_{C_{1}}P_{N_1}u_{1,T}\cdot Q_{<M}^{\sigma_2}P_{C_{2}(C_{1})}P_{N_2}u_{2,T}) \big\|_{L^{\frac{3}{2}}(\R\times \T)}\\
&\hspace{34ex}\times \big\|\eta (t)^{\frac{1}{3}}Q_{<M}^{\sigma_3}P_{N_3}u_{3,T} \big\|_{L^3(\R\times \T)}\\
&\lesssim T^{\frac{1}{2}}N_3^{1+2a}\prod_{j=1}^3\|P_{N_{j}}u_{j}\|_{Y^0_{\sigma_j}}
\end{split}
\]
for any $a>0$. 
By choosing $a>0$ as $2a<s-1$, 
we get (\ref{hl_loc_T_reso}) with $\delta =s-1-2a$. 

Next, we assume $d\ge 5$ and $s>s_c$. 
In this case, we choose $q= \frac{d+2}d$. 
By the same argument as above, we have
\begin{equation}\label{d2dbilinest_int_1}
\begin{split}
&\big\|\eta (t)^{\frac{d}{d+2}} P_{N_3}(Q_{<M}^{\sigma_1}P_{C_{1}}P_{N_1}u_{1,T}\cdot Q_{<M}^{\sigma_2}P_{C_{2}(C_{1})}P_{N_2}u_{2,T}) \big\|_{L^{\frac{d+2}{d}}(\R\times \T^d)}\\
&\lesssim \big\|\eta (t)^{\frac{d}{2(d+2)}}Q_{<M}^{\sigma_1}P_{C_{1}}P_{N_1}u_{1,T} \big\|_{L^{\frac{2(d+2)}{d}}(\R\times \T^d)}
\\
&\qquad \times
\big\|\eta (t)^{\frac{d}{2(d+2)}} Q_{<M}^{\sigma_2}P_{C_{2}(C_1)}P_{N_2}u_{2,T} \big\|_{L^{\frac{2(d+2)}{d}}(\R\times \T^d)}\\
&\lesssim N_3^{2a}\|P_{C_{1}}P_{N_1}u_{1}\|_{Y_{\sigma_1}^0}\|P_{C_{2}(C_1)}P_{N_2}u_{2}\|_{Y_{\sigma_2}^0}
\end{split}
\end{equation}
for any $a>0$ and
\begin{equation}\label{d2dbilinest_int_2}
\begin{split}
&\big\|\eta (t)^{\frac{d}{d+2}} P_{N_3}(Q_{<M}^{\sigma_1}P_{C_{1}}P_{N_1}u_{1,T}\cdot Q_{<M}^{\sigma_2}P_{C_{2}(C_{1})}P_{N_2}u_{2,T}) \big\|_{L^{\frac{d+2}{d}}(\R\times \T^d)}\\
&\lesssim \big\|\eta (t)^{\frac{1}{2}}Q_{<M}^{\sigma_1}P_{C_{1}}P_{N_1}u_{1,T} \big\|_{L^{2}(\R\times \T^d)}
\\
&\qquad \times
\big\|\eta (t)^{\frac{d-2}{2(d+2)}}Q_{<M}^{\sigma_2}P_{C_{2}(C_1)}P_{N_2}u_{2,T} \big\|_{L^{\frac{2(d+2)}{d-2}}(\R\times \T^d)}\\
&\lesssim T^{\frac{1}{2}}N_3\|P_{C_{1}}P_{N_1}u_{1}\|_{Y_{\sigma_1}^0}\|P_{C_{2}(C_1)}P_{N_2}u_{2}\|_{Y_{\sigma_2}^0}. 
\end{split}
\end{equation}
By the interpolation between (\ref{d2dbilinest_int_1}) and (\ref{d2dbilinest_int_2}), 
it holds that
\[
\begin{split}
&\big\|\eta (t)^{\frac{d}{d+2}} P_{N_3}(Q_{<M}^{\sigma_1}P_{C_{1}}P_{N_1}u_{1,T}\cdot Q_{<M}^{\sigma_2}P_{C_{2}(C_{1})}P_{N_2}u_{2,T}) \big\|_{L^{\frac{d+2}{d}}(\R\times \T^d)}\\
&\lesssim T^{\eps}N_3^{2a+2\left(1-a\right)\eps}\|P_{C_{1}}P_{N_1}u_{1}\|_{Y_{\sigma_1}^0}\|P_{C_{2}(C_1)}P_{N_2}u_{2}\|_{Y_{\sigma_2}^0}
\end{split}
\]
for any $a>0$ and $0<\eps <\frac 12$. 
By using this estimate, (\ref{eta_lp_up_est}) with $p=\frac{d+2}2$, the embeddings
$Y^0_{\sigma}\hookrightarrow V^{2}_{\sigma}L^{2}\hookrightarrow U^{3}_{\sigma}L^{2}$, 
(\ref{Vproj_T}), and (\ref{Y-norm_T_est}), we obtain
\[
\begin{split}
&\left| N_{3}\int_{\R}\int_{\T^d}\eta (t)\left(\prod_{j=1}^{3}Q_{<M}^{\sigma_j}P_{N_{j}}u_{j,T}\right)dxdt\right|\\
&\le N_3\sum_{C_1\in \CC_{N_3}}
\big\|\eta (t)^{\frac{d}{d+2}} P_{N_3}(Q_{<M}^{\sigma_1}P_{C_{1}}P_{N_1}u_{1,T}\cdot Q_{<M}^{\sigma_2}P_{C_{2}(C_{1})}P_{N_2}u_{2,T}) \big\|_{L^{\frac{d+2}{d}}(\R\times \T^d)}\\
&\hspace{34ex}\times \big\|\eta (t)^{\frac{2}{d+2}}Q_{<M}^{\sigma_3}P_{N_3}u_{3,T} \big\|_{L^{\frac{d+2}{2}}(\R\times \T^d)}\\
&\lesssim T^{\eps}N_3^{s_c+2a+2\left(1-a\right)\eps}\prod_{j=1}^3\|P_{N_{j}}u_{j}\|_{Y^0_{\sigma_j}}
\end{split}
\]
for any $a>0$ and $0<\eps <\frac 12$.
Now we have used the fact that 
$p>\frac{2(d+2)}d$ 
 and $s_c-1= \frac d2-\frac{d+2}p$ 
hold for $p=\frac{d+2}2$ and $d\ge 5$.
By choosing $a>0$ and $0<\eps <\frac 12$ such that 
$0<a< \min \{ 1 , \frac{s-s_c}2 \}$ and $2a+2(1-a)\eps<s-s_c$,
we get (\ref{hl_loc_T_reso}) with $\delta =s-s_c-2a-2(1-a)\eps$. 
\end{proof}

\subsection{Resonance case III}
\label{subsec:ResIII}
We give the trilinear estimate 
for $d=1,2$ under the condition $\mus< 0$.
We first consider the two dimensional case.
The following trilinear estimate plays a crucial role to handle resonant interactions.
Analogous trilinear estimates have been studied in \cite{Kishi13}, \cite{LO24}. 
\begin{thm}\label{thm:Triest_4.3}
Let $d=2$, and
$\sigma_1$, $\sigma_2$, $\sigma_3 \in \R\setminus \{0\}$ satisfy $\mus < 0$.
For any dyadic numbers $N$, $M_{1}$, $M_{2}$, $M_{3}$ with $M_{\max} \ll N$.
Then, we have
\begin{equation}\label{est:TriEst2d01}
\begin{split}
& \bigg| \int_\R \int_{\T^2} \eta(t) \bigl(Q^{\sigma_1}_{<M_1} P_{N} u_1\bigr) \bigl( Q^{\sigma_2}_{<M_2} P_{< N} u_2 \bigr) \bigl( Q^{\sigma_3}_{<M_3} P_{< N} u_3 \bigr)dxdt \bigg|\\
& \lesssim M_{\min}^{\frac12} M_{\max}^{\frac14} \| Q^{\sigma_1}_{<M_1} P_{N} u_1\|_{L_{t,x}^2} 
\| Q^{\sigma_2}_{<M_2} P_{< N} u_2\|_{L_{t,x}^2}  \| Q^{\sigma_3}_{<M_3} P_{< N} u_3\|_{L_{t,x}^2}.
\end{split}
\end{equation}
\end{thm}
\begin{rem}\label{rem:ConvEst}
Theorem \ref{thm:Triest_4.3} can be viewed as a refined nonlinear Loomis--Whitney inequality on $\R \times$(lattices) obtained in \cite{KS21}. See Proposition 4.8 in \cite{KS21}. The nonlinear Loomis--Whitney inequality can be applied to the study of general dispersive equations. However, the transversality condition, which is not assumed here, is a crucial for the nonlinear Loomis--Whitney inequality. 
Hence a simple application of Proposition 4.8 in \cite{KS21} would not yield Theorem \ref{thm:Triest_4.3}. 
We will adopt a similar but more direct approach to show Theorem \ref{thm:Triest_4.3} compared with the proof of Proposition 4.8 in \cite{KS21}.
\end{rem}

\begin{proof}[Proof of Theorem \ref{thm:Triest_4.3}]
By Plancherel' theorem, \eqref{est:TriEst2d01} is equivalent to 
\begin{equation}
\label{goal:ConvEst1}
| (\widehat \eta \ast_\tau f_1 \ast f_2 \ast f_3) (0)| \lesssim M_{\min}^{\frac12} M_{\max}^{\frac14} 
\prod_{j=1}^3 \|f_j\|_{L^2},
\end{equation}
where
\begin{align*}
& \supp f_1 \subset \{ (\tau,\xi) \in \R \times \Z^2 \, | \, 
|\tau + \sigma_1 |\xi|^2| \leq M_1, \ |\xi| \sim N\},\\
& \supp f_2 \subset \{ (\tau,\xi) \in \R \times \Z^2 \, | \, 
|\tau+\sigma_2 |\xi|^2| \leq M_2, \ |\xi| \lesssim N\},\\
& \supp f_3 \subset \{ (\tau,\xi) \in \R \times \Z^2 \, | \, 
|\tau+ \sigma_3 |\xi|^2| \leq M_3, \ |\xi| \lesssim N\}.
\end{align*}
By the harmless decomposition, we may assume that there exist $\widetilde{\xi_1}$, $\widetilde{\xi_2} \in \R^2$ such that 
\begin{align*}
& \supp f_1 \subset \{ (\tau,\xi) \in \R \times \Z^2 \, | \, 
|\tau+ \sigma_1 |\xi|^2| \leq M_1, \ |\xi-\widetilde{\xi_1}| \ll N \},\\
& \supp f_2 \subset \{ (\tau,\xi) \in \R \times \Z^2 \, | \, 
|\tau+ \sigma_2 |\xi|^2| \leq M_2,  \ |\xi-\widetilde{\xi_2}| \ll N\},\\
& \supp f_3 \subset \{ (\tau,\xi) \in \R \times \Z^2 \, | \, 
|\tau+ \sigma_3 |\xi|^2| \leq M_3,  \ |\xi+\widetilde{\xi_1} + \widetilde{\xi_2} | \ll N\}.
\end{align*}
It follows from \eqref{etaft1} that $\supp \widehat \eta \subset [-\pi,\pi]$.
Define $\supp_{\xi} f_j = \{\xi_j \in \Z^2 \, | \, \text{there exists} \ \tau_j \in \R \ \mathrm{such} \ \mathrm{that} \ (\tau_j,\xi_j) \in \supp f_j \}$ and 
\begin{align*}
&\Psi_2(\tau_0,\tau_1,\xi_1,\tau_2,\xi_2)
= |\tau_0| + |\tau_1+ \sigma_1|\xi_1|^2| + |\tau_2 + \sigma_2|\xi_2|^2|
\\
&\hspace*{130pt} + |\tau_0+\tau_1+\tau_2 - \sigma_3 |\xi_1+\xi_2|^2|,\\ 
& \Psi_3(\tau_0,\tau_1,\xi_1,\tau_3,\xi_3)
= |\tau_0| + |\tau_1+ \sigma_1|\xi_1|^2| + |\tau_0+\tau_1+\tau_3 - \sigma_2|\xi_1+\xi_3|^2|
\\
&\hspace*{130pt} + |\tau_3 + \sigma_3 |\xi_3|^2|,\\ 
& S_{\xi_1,M_{\max}}^2  = \left\{ \xi_2 \in \supp_{\xi} f_2 \, \middle| \, 
\begin{aligned} 
&-\xi_1-\xi_2 \in \supp_{\xi} f_3,
\\
&\text{there exist} \ |\tau_0| \le \pi \text{ and } \tau_1, \tau_2 \in \R \ \mathrm{ such} \ \mathrm{that}
\\
& \Psi_2(\tau_0, \tau_1-\tau_0,\xi_1,\tau_2,\xi_2) \leq 3 M_{\max}
\end{aligned} \right\},\\
& S_{\xi_1,M_{\max}}^3  = \left\{ \xi_3 \in \supp_{\xi} f_3 \, \middle| \, 
\begin{aligned} 
&-\xi_1-\xi_3 \in \supp_{\xi} f_2,
\\
& \text{there exist} \ |\tau_0| \le \pi \text{ and } \tau_1, \tau_3 \in \R \ \mathrm{such} \ \mathrm{that} \\
& \Psi_3(\tau_0, \tau_1-\tau_0,\xi_1,\tau_3,\xi_3) \leq 3 M_{\max}
\end{aligned} \right\}.
\end{align*}
To see \eqref{goal:ConvEst1}, it suffices to show
\begin{equation}
\label{est:ConvEst2}
\sup_{\substack{\xi_2 \in \supp_{\xi} f_2\\\xi_3 \in \supp_{\xi} f_3}} \sum_{\xi_1 \in \supp_{\xi} f_1} \1_{S_{\xi_1,M_{\max}}^2}(\xi_2) \times \1_{S_{\xi_1,M_{\max}}^3}(\xi_3) \lesssim M_{\max}^{\frac12}.
\end{equation}
Indeed, if \eqref{est:ConvEst2} holds, by the Cauchy-Schwarz inequality, we have
\begin{align*}
&| (\widehat \eta \ast_\tau f_1 \ast f_2 \ast f_3) (0)|
\\
&=
\bigg| \int_{\R} \widehat \eta (\tau_0) (f_1 \ast f_2 \ast f_3)(-\tau_0,0) d \tau_0 \bigg|\\
&
\leq \sum_{\xi_1 \in \supp_{\xi} f_1}
\bigg| \int_{\R} \widehat \eta(\tau_0)
\int_\R f_1 (\tau_1-\tau_0,\xi_1) ( f_2 \ast f_3)(-\tau_1,-\xi_1) d \tau_1 d \tau_0 \bigg|\\
& \lesssim M_{\min}^{\frac12} \sum_{\xi_1 \in \supp_{\xi} f_1}
\| \widehat \eta \|_{L^1}
\|f_1(\xi_1)\|_{L_{\tau}^2} \|f_2|_{S_{\xi_1,M_{\max}}^2}\|_{L^2}\|f_3|_{S_{\xi_1,M_{\max}}^3}\|_{L^2}\\
& \leq M_{\min}^{\frac12} 
\| \widehat \eta \|_{L^1}
\|f_1\|_{L^2}
\Bigl( \sum_{\xi_1 \in \supp_{\xi} f_1} \|f_2|_{S_{\xi_1,M_{\max}}^2}\|_{L^2}^2 \|f_3|_{S_{\xi_1,M_{\max}}^3}\|_{L^2}^2 \Bigr)^{\frac12} \\
& \lesssim  M_{\min}^{\frac12} M_{\max}^{\frac14} \prod_{j=1}^3 \|f_j\|_{L^2}.
\end{align*}
To show \eqref{est:ConvEst2}, let us observe the condition of $\xi_1$ such that $(\xi_2,\xi_3) \in S_{\xi_1,M_{\max}}^2 \times S_{\xi_1,M_{\max}}^3$. If $\xi_2 \in S_{\xi_1,M_{\max}}^2$, since $\Psi_2(\tau_0,\tau_1-\tau_0,\xi_1,\tau_2,\xi_2) \lesssim M_{\max}$ for some $|\tau_0| \le \pi$ and $\tau_1, \tau_2 \in \R$, we have
\begin{equation}\label{est:ConvEst2A}
\bigl|\sigma_1 |\xi_1|^2 +\sigma_2|\xi_2|^2 + \sigma_3|\xi_1+\xi_2|^2 \bigr| \lesssim M_{\max}.
\end{equation}
Similarly, if $\xi_3 \in S_{\xi_1,M_{\max}}^3$, there exist $|\tau_0| \le \pi$ and $\tau_1, \tau_3 \in \R$ such that $\Psi_3(\tau_0, \tau_1 - \tau_0,\xi_1,\tau_3,\xi_3) \lesssim M_{\max}$. Hence, we have
\begin{equation}\label{est:ConvEst2B}
\bigl|\sigma_1 |\xi_1|^2 +\sigma_2|\xi_1+\xi_3|^2 + \sigma_3|\xi_3|^2 \bigr| \lesssim M_{\max}.
\end{equation}
We divide the proof of \eqref{est:ConvEst2} into the following two cases:
\begin{enumerate}
\item[{\rm (i)}] \label{ConvEst2_i}$(\sigma_1+\sigma_2)(\sigma_3+\sigma_1)\not= 0$,
\item[{\rm (ii)}] \label{ConvEst2_ii}$(\sigma_1+\sigma_2)(\sigma_3+\sigma_1) = 0$.
\end{enumerate}

\noindent \underline{The case {\rm (i)}}:
By $(\sigma_1+\sigma_2)(\sigma_3+\sigma_1)\not= 0$ and \eqref{mus}, we may write
\begin{align}
& \eqref{est:ConvEst2A} 
\iff  \bigg| \bigg| \xi_1 + \frac{\sigma_3}{\sigma_1+\sigma_3} \xi_2\bigg|^2 + \frac{\mus}{(\sigma_1+ \sigma_3)^2} |\xi_2|^2 \bigg| \lesssim M_{\max},\label{est:ConvEst3} \\
& \eqref{est:ConvEst2B} 
\iff  \bigg| \bigg| \xi_1 + \frac{\sigma_2}{\sigma_1+\sigma_2} \xi_3\bigg|^2 + \frac{\mus}{(\sigma_1+ \sigma_2)^2} |\xi_3|^2 \bigg| \lesssim M_{\max}.\label{est:ConvEst4}
\end{align}
It is clear that \eqref{est:ConvEst3} and \eqref{est:ConvEst4} imply $|\xi_2| \sim |\xi_3| \sim N$.
With $M_{\max} \ll N$,
it follows from \eqref{est:ConvEst3} and \eqref{est:ConvEst4} that
\begin{equation}
\begin{aligned}
&\bigg| \Bigl(\frac{2 \sigma_3}{\sigma_1+\sigma_3} \xi_2 - \frac{2\sigma_2}{\sigma_1+\sigma_2} \xi_3 \Bigr) \cdot \xi_1 
\\
&\quad
+ \frac{\sigma_3^2+\mus}{\sigma_3^2} \Bigl(\frac{\sigma_3^2}{ (\sigma_1+\sigma_3)^2} |\xi_2|^2 - \frac{\sigma_2^2}{(\sigma_1+\sigma_2)^2}|\xi_3|^2 \Bigr) 
%\\
%&\quad
+ \frac{\mus (\sigma_2^2 -\sigma_3^2)}{\sigma_3^2(\sigma_1+\sigma_2)^2} |\xi_3|^2 \bigg| \\
& \lesssim M_{\max}.
\end{aligned}
\label{est:ConvEst5}
\end{equation}
Now let us see
\begin{equation}
\label{est:ConvEst6}
\bigg| \frac{\sigma_3}{\sigma_1+\sigma_3} \xi_2 - \frac{\sigma_2}{\sigma_1+\sigma_2} \xi_3 \bigg| \sim N.
\end{equation} 
First, suppose that $\sigma_2 \not = \sigma_3$. 
Then, \eqref{est:ConvEst5} and $|\frac{\sigma_3}{\sigma_1+\sigma_3} \xi_2 - \frac{\sigma_2}{\sigma_1+\sigma_2} \xi_3| \ll N$ imply
\[
\bigg| \frac{\mus (\sigma_2^2 -\sigma_3^2)}{\sigma_3^2(\sigma_1+\sigma_2)^2} \bigg| |\xi_3|^2  \ll N^2,
\]
which contradicts the assumptions on exponents. 
While, in the case $\sigma_2 = \sigma_3$, if $|\frac{\sigma_3}{\sigma_1+\sigma_3} \xi_2 - \frac{\sigma_2}{\sigma_1+\sigma_2} \xi_3| \ll N$, we have $|\xi_2 - \xi_3| \ll N$. 
Consequently, the support conditions
\[
|\xi_1 - \widetilde{\xi_1}| + |\xi_2 - \widetilde{\xi_2}| + |\xi_3 + \widetilde{\xi_1} + \widetilde{\xi_2}| \ll N
\]
imply $|\xi_1 + 2 \xi_2| \ll N$.
This condition yields that $\widetilde{\xi_1} = -2 \widetilde{\xi_2}$.
Then, \eqref{est:ConvEst3} with $M_{\max} \ll N$ implies that
$\big| -2 + \frac{\sigma_3}{\sigma_1+\sigma_2} \big|^2 + \frac{\mus}{(\sigma_1+\sigma_2)^2} =0$,
which contradicts $\mus<0$.
We complete the proof of \eqref{est:ConvEst6}.

As a consequence, the conditions \eqref{est:ConvEst3}, \eqref{est:ConvEst5}, and \eqref{est:ConvEst6} imply that $\xi_1$ is contained in the intersection of the annulus of radius $\sim N$, width $\sim \frac{M_{\max}}N$ and the strip of width $\sim \frac{M_{\max}}N$. It is easy to see that the number of such $\xi_1 \in \Z^2$ is at most $\sim M_{\max}^{\frac12}$. See Figures \ref{fig:AnnuluStrip1} and \ref{fig:AnnuluStrip2}. This completes the proof of \eqref{est:ConvEst2} in the case {\rm (i)}. \\

\begin{figure}[t]
\begin{tabular}{cc}\hspace{-0.5cm}
\begin{minipage}[t]{7.7cm}
\begin{center}
\begin{tikzpicture}[]
\useasboundingbox (0,0) rectangle (4.5,5);
\path[draw,thick,fill=gray]
plot[domain=0.935:1.05,variable=\t, samples=50] ({2*cos(\t r) + 2},{2*sin(\t r)+2.5})--plot[domain=1.1:0.995,variable=\t, samples=50] ({2.2*cos(\t r) + 2},{2.2*sin(\t r)+2.5});
\path[draw,thick,fill=gray]
plot[domain=-0.935:-1.05,variable=\t, samples=50] ({2*cos(\t r) + 2},{2*sin(\t r)+2.5})--plot[domain=-1.1:-0.995,variable=\t, samples=50] ({2.2*cos(\t r) + 2},{2.2*sin(\t r)+2.5});
\draw (2,2.5) circle [radius=2];
\draw (2,2.5) circle [radius=2.2];
\draw (3,0.2)--(3,4.8);
\draw (3.2,0.2)--(3.2,4.8);
\end{tikzpicture}
\caption{An annulus and a strip intersect transversely.}
\label{fig:AnnuluStrip1}
\end{center}
\end{minipage} \hspace{-0.5cm}
\begin{minipage}[t]{7.7cm}
\centering
\begin{tikzpicture}[]
\useasboundingbox (0,0) rectangle (5,5);
\path[draw,thick,fill=gray]
plot[domain=-0.42:0.42,variable=\t, samples=50] ({2.2*cos(\t r) + 2},{2.2*sin(\t r)+2.5});
\draw (2,2.5) circle [radius=2];
\draw (2,2.5) circle [radius=2.2];
\draw (4,0.2)--(4,4.8);
\draw (4.2,0.2)--(4.2,4.8);
\draw[densely dotted] ({2.2*cos(0.42 r) + 2},{2.2*sin(0.42 r)+2.5})--(4.7,{2.2*sin(0.42 r)+2.5});
\draw[densely dotted] ({2.2*cos(-0.42 r) + 2},{2.2*sin(-0.42 r)+2.5})--(4.7,{2.2*sin(-0.42 r)+2.5});
\draw[<->] (4.5,3.35)--(4.5,1.65);
\node[anchor=west] at(4.5,2.5){$\sim M_{\max}^{\frac12}$};
\draw[<->] ({2.2*cos(0.42 r)+1.99},{2.2*sin(0.42 r)+4})--({2.2*cos(0.42 r)+2.2},{2.2*sin(0.42 r)+4});
\node[anchor=west] at(4.1,5){$\sim \frac{M_{\max}}N$};
\end{tikzpicture}
\caption{An annulus and a strip intersect tangentially.}
\label{fig:AnnuluStrip2}
\end{minipage}
\end{tabular}
\end{figure}

\noindent \underline{The case {\rm (ii)}}:
By symmetry, it is enough to consider the case $\sigma_1+ \sigma_2=0$. We consider the two cases: (iiA) $\sigma_1 + \sigma_3 \not=0$ and (iiB) $\sigma_1 + \sigma_3=0$. In the first case, similarly to \eqref{est:ConvEst3} and \eqref{est:ConvEst4}, we have
\begin{align}
& \eqref{est:ConvEst2A} 
\iff  \bigg| \bigg| \xi_1 + \frac{\sigma_3}{\sigma_1+\sigma_3} \xi_2\bigg|^2 - \frac{\sigma_1^2}{(\sigma_1+ \sigma_3)^2} |\xi_2|^2 \bigg| \lesssim M_{\max},\label{est:ConvEst7} \\
& \eqref{est:ConvEst2B} 
\iff  \bigg|\xi_3 \cdot \Bigl(\xi_1+ \frac{\sigma_1-\sigma_3}{2 \sigma_1} \xi_3\Bigr) \bigg| \lesssim M_{\max}.
%\label{est:ConvEst8}
\notag
\end{align}
It is clear that \eqref{est:ConvEst7} implies $|\xi_2| \sim N$.
Moreover,
these conditions imply that $\xi_1$ is confined in the intersection of the strip of width 
$\frac{M_{\max}}{|\xi_3|}$ and the annulus of width $\frac{M_{\max}}N$. 
The case $|\xi_3| \sim N$ can be dealt with in the same way as in the case {\rm (i)}. 
Hence, we suppose that $|\xi_3| \ll N$. Let us consider $|\xi_3| \lesssim M_{\max}^{\frac12}$ first. 
In this case, since we may assume $|\xi_1+ \xi_2| \lesssim M_{\max}^{\frac12}$,
%\footnote{Strictly speaking, we need to divide the proof into the two cases $|\xi_3| \lesssim M_{\max}^{\frac12}$ and $|\xi_3| \gtrsim M_{\max}^{\frac12}$ at the beginning of the proof.}
\eqref{est:ConvEst7} implies the claim \eqref{est:ConvEst2}.
While, if $|\xi_3| \gtrsim M_{\max}^{\frac12}$, the width of the strip is $\frac{M_{\max}}{|\xi_3|} \lesssim M_{\max}^{\frac12}$. Since $|\xi_1| \sim N$ and $|\xi_3| \ll N$, the strip and the annulus intersect transversely. Therefore, we get the bound \eqref{est:ConvEst2} for the case $\sigma_1 + \sigma_2=0$ and (iiA). 

Next, we consider the case (iiB). In this case, the two conditions are
\begin{align}
& \eqref{est:ConvEst2A} 
\iff  |\xi_2 \cdot (\xi_1 + \xi_2) | \lesssim M_{\max},\label{est:ConvEst9} \\
& \eqref{est:ConvEst2B} 
\iff  |\xi_3 \cdot (\xi_1+\xi_3) | \lesssim M_{\max}.\label{est:ConvEst10}
\end{align}
Notice that these conditions imply that $\xi_1$ is contained in the intersection of two strips of widths $\sim \frac{M_{\max}}{|\xi_2|}$ and $\frac{M_{\max}}{|\xi_3|}$, respectively. 
Without loss of generality, we may assume $|\xi_2| \sim N$. 
If $|\xi_3| \sim N$, it follows from \eqref{est:ConvEst9}, \eqref{est:ConvEst10}, and the support condition $|\xi_1 +\xi_2 + \xi_3| \ll N$ that $\xi_2$ and $\xi_3$ are, as vectors, almost perpendicular. 
Hence, the two strips intersect transversely and $\xi_1$ is contained in a square cube of side length $\sim \frac{M_{\max}}N$. Next, we assume that $|\xi_3| \ll N$. In this case, however, we may show \eqref{est:ConvEst2} in the same way as in the proof of the case (iiA) under $|\xi_3| \ll N$. Thus we omit the proof.
\end{proof}

\begin{rem}\label{rem:ConvEstTori}
By replacing $\Z^2$ in the proof of Theorem \ref{thm:Triest_4.3} with $\theta_1 \Z \times \theta_2 \Z$ where $0 < \theta_1,\theta_2< \infty$, it is easy to check that we may replace $\T^2$ of \eqref{est:TriEst2d01} with $(\R / 2\pi \theta _1 \Z) \times (\R / 2\pi \theta _2 \Z)$. 
\end{rem}

A similar, but simpler, calculation yields the trilinear estimate in the one dimensional case.

\begin{cor}\label{cor:Triest_4.4}
Let $d=1$, and
$\sigma_1$, $\sigma_2$, $\sigma_3 \in \R\setminus \{0\}$ satisfy $\mus < 0$.
For any dyadic numbers $N$, $M_{1}$, $M_{2}$, $M_{3}$ with $M_{\max} \ll N$.
Then, we have
\[
\begin{split}
& \bigg| \int_\R \int_{\T} \eta (t) \bigl(Q^{\sigma_1}_{<M_1} P_{N} u_1\bigr) \bigl( Q^{\sigma_2}_{<M_2} P_{< N} u_2 \bigr) \bigl( Q^{\sigma_3}_{<M_3} P_{< N} u_3 \bigr)dxdt \bigg|\\
& \lesssim M_{\min}^{\frac12} \| Q^{\sigma_1}_{<M_1} P_{N} u_1\|_{L_{t,x}^2} 
\| Q^{\sigma_2}_{<M_2} P_{< N} u_2\|_{L_{t,x}^2}  \| Q^{\sigma_3}_{<M_3} P_{< N} u_3\|_{L_{t,x}^2}.
\end{split}
\]
\end{cor}

\begin{proof}
We use the same notation as in the proof of Theorem \ref{thm:Triest_4.3}.
When $d=1$,
\eqref{est:ConvEst2A} and \eqref{est:ConvEst2B} yield that $\xi_1$ is contained in an interval of side length $\lec \frac{M_{\max}}N$.
Since $\xi_1 \in \Z$,
we obtain
\[
\sup_{\substack{\xi_2 \in \supp_{\xi} f_2\\\xi_3 \in \supp_{\xi} f_3}} \sum_{\xi_1 \in \supp_{\xi} f_1} \1_{S_{\xi_1,M_{\max}}^2}(\xi_2) \times \1_{S_{\xi_1,M_{\max}}^3}(\xi_3) \lesssim 1
\]
instead of \eqref{est:ConvEst2},
which shows the desired bound.
\end{proof}

\subsection{Time local estimates for critical case}

To prove the local well-posedness for large initial data in $\mathcal{H}^{s_c}(\T^d )$,
we use the following proposition.

\begin{prop}\label{nonreso_case_tri_est_loc_critical} 
Let $d\ge 3$ and $\sigma_{1}$, $\sigma_{2}$, $\sigma_3\in \R \backslash \{0\}$ satisfy 
$\kks \ne 0$. 
Then,
there exist $\eps, \delta, \theta>0$ such that
for any $0<T\le 1$, dyadic numbers $N_{1}$, $N_{2}$, $N_{3}$, $K \geq 1$,
and $P_{N_{j}}u_{j}\in V^{2}_{\sigma_{j}}L^{2}$\ $(j=1,2,3)$, we have 
\begin{align*}
&\left| N_{\max}\int_{\R}\int_{\T^{d}}\eta(t)
(P_{N_{1}} P_{<K} u_{1,T})
\left(\prod_{j=2}^{3} P_{N_{j}}u_{j,T}\right)
dxdt\right|
\\
&\lesssim  T^{\eps}K^{\theta}
N_{\min}^{s_c}
\left(\frac{N_{\min}}{N_{\max}}+\frac{1}{N_{\min}}\right)^{\delta}
\prod_{j=1}^{3}\|P_{N_{j}}u_{j}\|_{Y^{0}_{\sigma_{j}}}.
\end{align*}
\end{prop}

\begin{proof}
%As in the proof of Proposition \ref{HL_est_T_mass_reso},
We set $M = C^{-1} N_{\max}^2$
for some $C \gg 1$.
We divide the integral into $8$ pieces of the form such as \eqref{integral_divide_8piece}.

If $Q_j^{\sigma_j}=Q_{\ge M}^{\sigma_j}$ for some $j\in \{1,2,3\}$,
the H\"older inequality, (\ref{high_mod_deriv_gain}), and
the Bernstein inequality yield that
\begin{align*}
&\left| N_{\max}\int_{\R}\int_{\T^{d}}\eta(t)
(Q_{1}^{\sigma_1} P_{N_{1}} P_{<K} u_{1,T})
\left(\prod_{j=2}^{3}Q_{j}^{\sigma_j} P_{N_{j}}u_{j,T}\right)
dxdt\right|
\\
&\lesssim 
T^{\frac{1}{2}}K^{\frac{d}{2}}\prod_{j=1}^{3}\|P_{N_{j}}u_{j}\|_{Y^{0}_{\sigma_{j}}}.
\end{align*}
By interpolating this estimate and \eqref{nonreso_est_trilin},
we obtain the desired bound.

If $Q_j^{\sigma_j}=Q_{< M}^{\sigma_j}$ for any $j\in \{1,2,3\}$,
Lemma \ref{modul_est_T} with $\kks \ne 0$
yields that $N_1\sim N_2\sim N_3$. 
Then,
it follows from
the H\"older inequality and the Bernstein inequality with $N_{\max}\lesssim K$
that
\begin{align*}
&\left| N_{\max}\int_{\R}\int_{\T^{d}}\eta(t)
(Q_{1}^{\sigma_1} P_{N_{1}} P_{<K} u_{1,T})
\left(\prod_{j=2}^{3}Q_{j}^{\sigma_j} P_{N_{j}}u_{j,T}\right)
dxdt\right|
\\
&\lesssim 
TK^{\frac{d}{2}+1}\prod_{j=1}^{3}\|P_{N_{j}}u_{j}\|_{Y^{0}_{\sigma_{j}}},
\end{align*}
Since $d \ge 3$ implies $s_c>0$,
this shows the desired bound.
\end{proof}

\section{Proof of the well-posedness\label{WP2}}
In this section, we prove the well-posedness of (\ref{NLS_sys_torus}).
We define the map 
\[
\Phi(u,v,w)=(\Phi_{\alpha, u_{0}}^{(1)}(w, v), \Phi_{\beta, v_{0}}^{(1)}(\overline{w}, v), \Phi_{\gamma, w_{0}}^{(2)}(u, \overline{v}))
\]
as
\[
\begin{split}
\Phi_{\sigma, \varphi}^{(1)}(f,g)(t)&:=e^{it\sigma \Delta}\varphi +i I^{(1)}_{\sigma}(f,g)(t),\\
\Phi_{\sigma, \varphi}^{(2)}(f,g)(t)&:=e^{it\sigma \Delta}\varphi -i I^{(2)}_{\sigma}(f,g)(t),
\end{split}
\] 
where
\[
\begin{split}
I^{(1)}_{\sigma}(f,g)(t)&:=\int_{0}^{t}\ee_{[0,\infty )}(t')e^{i(t-t')\sigma \Delta}(\nabla \cdot f(t'))g(t')dt',\\
I^{(2)}_{\sigma}(f,g)(t)&:=\int_{0}^{t}\ee_{[0,\infty )}(t')e^{i(t-t')\sigma \Delta}\nabla (f(t')\cdot g(t'))dt'.
\end{split}
\]
\subsection{Except the case $\mu<0$ and $s=1$}
\label{subsec:WP11}

In this subsection, 
we prove Theorems~\ref{wellposed_T} and ~\ref{wellposed_sub_T} 
except the case $\mu<0$ and $s=1$. 
Key estimates are the followings.
\begin{prop}\label{Duam_est_T}
Assume that
$\alpha$, $\beta$, $\gamma \in \R\backslash \{0\}$ satisfy
\[
\begin{cases}
{\rm (a)}\ \mu>0&{\rm if}\  d=3,\\ 
{\rm (b)}\ \mu \ge0&{\rm if}\  d=4,\\
{\rm (c)}\ \kappa \neq 0&{\rm if}\ d\ge 5,
\end{cases}
\] 
where $\mu$ and $\kappa$ are defined in \eqref{coeff_condition}.
Then, for $0<T\le 1$, we have
\begin{align}
&\|I_{\alpha}^{(1)}(w,v)\|_{Z^{s_c}_{\alpha}([0,T))}\lesssim \|w\|_{Y^{s_c}_{\gamma}([0,T))}\|v\|_{Y^{s_c}_{\beta}([0,T))},\label{Duam_al_T}\\
&\|I_{\beta}^{(1)}(\overline{w},u)\|_{Z^{s_c}_{\beta}([0,T))}\lesssim \|w\|_{Y^{s_c}_{\gamma}([0,T))}\|u\|_{Y^{s_c}_{\alpha}([0,T))},\label{Duam_be_T}\\
&\|I_{\gamma}^{(2)}(u,\overline{v})\|_{Z^{s_c}_{\gamma}([0,T))}\lesssim \|u\|_{Y^{s_c}_{\alpha}([0,T))}\|v\|_{Y^{s_c}_{\beta}([0,T))}\label{Duam_ga_T}.
\end{align} 
\end{prop}
\begin{proof}
We prove only (\ref{Duam_ga_T}) for the case {\rm (a)}  
since the other cases and the estimates (\ref{Duam_al_T}), (\ref{Duam_be_T}) can be proved in the same way (we use Proposition~\ref{HL_est_T_mass_reso} for the case {\rm (b)} 
and Proposition~\ref{HL_est_T_2} for the case {\rm (c)} instead of Proposition~\ref{HL_est_T}). 
Let
\[
(u_{1},u_{2}):=(u,\overline{v}), \quad
(\sigma_{1},\sigma_{2},\sigma_{3}):=(\alpha ,-\beta ,-\gamma ).
\]
We define
\[
S_{j}:=\{ (N_{1}, N_{2}, N_{3})| \ N_{\max}\sim N_{{\rm med}}\gtrsim N_{\min}\geq 1,\ N_{\min}=N_{j}\}\ (j=1,2,3)
\]
and $S:=\bigcup_{j=1}^{3}S_{j}$, where $(N_{\max},N_{\rm med},N_{\min})$ is one of the permutation of $(N_{1},N_{2},N_{3})$ such that $N_{\max}\geq N_{\rm med}\geq N_{\min}$. 
%Since
%\[
%\widehat{u_1}(t, 0)\widehat{u_2}(t, 0)\widehat{\nabla \cdot u_3}(t, 0)=0
%\]
%for any $t>0$, we can assume $N_{\max}\ge 2$. 
Then we have
\[
\begin{split}
&\left\| I_{-\sigma_{3}}^{(2)}(u_{1}, u_{2})\right\|_{Z^{s_c}_{-\sigma_{3}}([0,T))}
\\
&\lesssim \sup_{\|u_{3}\|_{Y^{-s_c}_{\sigma_{3}}=1}}\left|\int_{0}^{T}\int_{\T^{d}}u_{1}u_{2}(\nabla \cdot u_{3})dxdt\right|\\
&\leq \sup_{\|u_{3}\|_{Y^{-s_c}_{\sigma_{3}}=1}}\sum_{(N_1,N_2,N_3)\in S}\left|\int_{0}^{T}\int_{\T^{d}}P_{N_{1}}u_{1}P_{N{2}}u_{2}P_{N_{3}} (\nabla \cdot u_{3}) dxdt\right|\\
&\leq \sup_{\|u_{3}\|_{Y^{-s_c}_{\sigma_{3}}=1}}\sum_{(N_1,N_2,N_3)\in S}N_{\min}^{s_c}\left(\frac{N_{\min}}{N_{\max}}+\frac{1}{N_{\min}}\right)^{\delta}\prod_{j=1}^{3}\|P_{N_{j}}u_{j}\|_{Y^{0}_{\sigma_{j}}}
\end{split}
\]
by Proposition~\ref{duality_T} and Proposition~\ref{HL_est_T} (see, also Remark~\ref{lowfreq_trilin_est}).  
Furthermore, we have
\[
\begin{split}
&\sum_{(N_1,N_2,N_3)\in S_{1}}N_{\min}^{s_c}\left(\frac{N_{\min}}{N_{\max}}+\frac{1}{N_{\min}}\right)^{\delta}\prod_{j=1}^{3}\|P_{N_{j}}u_{j}\|_{Y^{0}_{\sigma_{j}}}\\
&\sim\sum_{N_{2}}\sum_{N_{3}\sim N_{2}}\sum_{N_{1}\lesssim N_{2}}N_{3}^{s_c}N_{1}^{s_c}\left(\frac{N_{1}}{N_{2}}+\frac{1}{N_{1}}\right)^{\delta}\|P_{N_{1}}u_{1}\|_{Y^{0}_{\sigma_{1}}}\|P_{N_{2}}u_{2}\|_{Y^{0}_{\sigma_{2}}}\|P_{N_{3}}u_{3}\|_{Y^{-s_c}_{\sigma_{3}}}\\
&\leq \|u_{1}\|_{Y^{s_c}_{\sigma_{1}}}\|u_{2}\|_{Y^{s_c}_{\sigma_{2}}}\|u_{3}\|_{Y^{-s_c}_{\sigma_{3}}}
\end{split}
\]
and
\[
\begin{split}
&\sum_{(N_1,N_2,N_3)\in S_{3}}N_{\min}^{s_c}\left(\frac{N_{\min}}{N_{\max}}+\frac{1}{N_{\min}}\right)^{\delta}\prod_{j=1}^{3}\|P_{N_{j}}u_{j}\|_{Y^{0}_{\sigma_{j}}}\\
&\sim \sum_{N_{1}}\sum_{N_{2}\sim N_{1}}\sum_{N_{3}\lesssim N_{2}}N_{3}^{2s_c}\left(\frac{N_{3}}{N_{2}}+\frac{1}{N_{3}}\right)^{\delta}\|P_{N_{1}}u_{1}\|_{Y^{0}_{\sigma_{1}}}\|P_{N_{2}}u_{2}\|_{Y^{0}_{\sigma_{2}}}\|P_{N_{3}}u_{3}\|_{Y^{-s_c}_{\sigma_{3}}}\\
&\leq \|u_{1}\|_{Y^{s_c}_{\sigma_{1}}}\|u_{2}\|_{Y^{s_c}_{\sigma_{2}}}\|u_{3}\|_{Y^{-s_c}_{\sigma_{3}}}
\end{split}
\]
by the Cauchy-Schwarz inequality for the dyadic sum. In the same way as the estimate for the summation of $S_{1}$, we have 
\[
\sum_{(N_1,N_2,N_3)\in S_{2}}N_{\min}^{s}\left(\frac{N_{\min}}{N_{\max}}+\frac{1}{N_{\min}}\right)^{\delta}\prod_{j=1}^{3}\|P_{N_{j}}u_{j}\|_{Y^{0}_{\sigma_{j}}}
\lesssim \|u_{1}\|_{Y^{s}_{\sigma_{1}}}\|u_{2}\|_{Y^{s}_{\sigma_{2}}}\|u_{3}\|_{Y^{-s}_{\sigma_{3}}}. 
\]
Therefore, we obtain (\ref{Duam_ga_T}) since $\|u_{1}\|_{Y^{s}_{\sigma_{1}}}=\|u\|_{Y^{s}_{\alpha}}$ and $\|u_{2}\|_{Y^{s}_{\sigma_{2}}}=\|v\|_{Y^{s}_{\beta}}$. 
\end{proof}

The same argument with Proposition \ref{nonreso_case_tri_est_loc_critical} yields
the following time local estimate.

\begin{prop}\label{Duam_est_T_loc_critical}
Let $d \ge 3$ and 
$\alpha$, $\beta$, $\gamma \in \R\backslash \{0\}$ satisfy $\kappa \neq 0$.
Then, there exists $\eps, \theta >0$ such that 
for any $0<T\le 1$ and dyadic number $K\ge 1$,  we have
\begin{align*}
&\|I_{\alpha}^{(1)}(w,v)-I_{\alpha}^{(1)}(P_{\ge K}w,P_{\ge K}v)\|_{Z^{s_c}_{\alpha}([0,T))}\lesssim 
T^{\eps}K^{\theta}\|w\|_{Y^{s_c}_{\gamma}([0,T))}\|v\|_{Y^{s_c}_{\beta}([0,T))},\\
&\|I_{\beta}^{(1)}(\overline{w},u)-I_{\beta}^{(1)}(\overline{P_{\ge K}w},P_{\ge K}u)\|_{Z^{s_c}_{\beta}([0,T))}\lesssim T^{\eps}K^{\theta}\|w\|_{Y^{s_c}_{\gamma}([0,T))}\|u\|_{Y^{s_c}_{\alpha}([0,T))},\\
&\|I_{\gamma}^{(2)}(u,\overline{v})-I_{\gamma}^{(2)}(P_{\ge K}u,\overline{P_{\ge K}v})\|_{Z^{s_c}_{\gamma}([0,T))}\lesssim T^{\eps}K^{\theta}\|u\|_{Y^{s_c}_{\alpha}([0,T))}\|v\|_{Y^{s_c}_{\beta}([0,T))}.
\end{align*} 
\end{prop}

Combining the estimates above,
we obtain Theorem~\ref{wellposed_T}.
While the argument is the same as that in \cite{HTT11},
we give the proof for completeness.

\begin{proof}[\rm{\bf{Proof of Theorem~\ref{wellposed_T}.}}]
For an interval $I\subset \R$, we define
\begin{align}\label{Xs_norm_T}
X^{s_c}(I)
&:=Z^{s_c}_{\alpha}(I)\times Z^{s_c}_{\beta}(I)\times Z^{s_c}_{\gamma}(I), \\
\| (u,v,w) \|_{X^{s_c}(I)}
&:= \max \big\{ \|u\|_{Z_{\alpha}^{s_c}(I)},\ \|v\|_{Z_{\beta}^{s_c}(I)},\ \|w\|_{Z_{\gamma}^{s_c}(I)} \big\}.
\notag
\end{align}
Moreover,
we set
\[
X^{s_c}_{r}(I)
:=\left\{(u,v,w)\in X^{s_c}(I)\left|\
%\|u\|_{Z_{\alpha}^{s_c}(I)},\ \|v\|_{Z_{\beta}^{s_c}(I)},\ \|w\|_{Z_{\gamma}^{s_c}(I)}\leq 
\| (u,v,w) \|_{X^{s_c}(I)}
\le r \right.\right\}
\]
for $r>0$.
Note that $X^{s_c}_{r}(I)$ is a closed subset of the Banach space $X^{s_c}(I)$. 
Let
$C$ be the maximum of the implicit constants in the estimates in Propositions \ref{Duam_est_T} and \ref{Duam_est_T_loc_critical}.

\noindent \underline{\rm Case (a)} (Small initial data):
Let $r>0$ satisfy
\[
r< \frac 1{8C}.
\]
Let $(u_{0}$, $v_{0}$, $w_{0})\in \H^{s_c}(\T^d )$ satisfy
\[
\max\{ \| u_0 \|_{H^{s_c}}, \| v_0 \|_{H^{s_c}}, \| w_0 \|_{H^{s_c}} \} \le r.
\]
Note that 
\[
\|e^{i\sigma t\Delta}\varphi \|_{Z^{s_c}_{\sigma}([0,1))}
\leq \| e^{i\sigma t\Delta}\varphi \|_{Z^{s_c}_{\sigma}}
\leq \|\varphi \|_{H^{s_c}}.
\]
For $(u,v,w)\in X^{s_c}_{2r}([0,1))$, 
Proposition~\ref{Duam_est_T} yields that
\[
\begin{split}
\|\Phi^{(1)}_{\alpha ,u_{0}}(w, v)\|_{Z_{\alpha}^{s_c}([0,1))}&\leq \|u_{0}\|_{H^{s_c}} +C\|w\|_{Z_{\gamma}^{s_c}([0,1))}\|v\|_{Z_{\beta}^{s_c}([0,1))}\leq r(1+4Cr) < 2r,\\
\|\Phi^{(1)}_{\beta ,v_{0}}(\overline{w}, u)\|_{Z_{\beta}^{s_c}([0,1))}&\leq \|v_{0}\|_{H^{s_c}} +C\|w\|_{Z_{\gamma}^{s_c}([0,1))}\|u\|_{Z_{\alpha}^{s_c}([0,1))}\leq r(1+4Cr)<2r,\\
\|\Phi^{(2)}_{\gamma ,w_{0}}(u, \overline{v})\|_{Z_{\gamma}^{s_c}([0,1))}&\leq \|w_{0}\|_{H^{s_c}} +C\|u\|_{Z_{\alpha}^{s_c}([0,1))}\|v\|_{Z_{\beta}^{s_c}([0,1))}\leq r(1+4Cr)<2r.
\end{split}
\]
Similarly,
for $(u_1,v_1,w_1), (u_2,v_2,w_2)\in X^{s_c}_{2r}([0,1))$, 
we have
\[
\begin{split}
&\|\Phi^{(1)}_{\alpha ,u_{0}}(w_{1}, v_{1})-\Phi^{(1)}_{\alpha ,u_{0}}(w_{2}, v_{2})\|_{Z_{\alpha}^{s_c}([0,1))}
\\
&\leq 4Cr
\left( \|w_{1}-w_{2}\|_{Z_{\gamma}^{s_c}([0,1))}+\|v_{1}-v_{2}\|_{Z_{\beta}^{s_c}([0,1))}\right)
,\\
&\|\Phi^{(1)}_{\beta ,v_{0}}(\overline{w_{1}}, u_{1})-\Phi^{(1)}_{\beta ,v_{0}}(\overline{w_{2}}, u_{2})\|_{Z_{\beta}^{s_c}([0,1))}
\\
&\leq 4Cr\left( \|w_{1}-w_{2}\|_{Z_{\gamma}^{s_c}([0,1))}+\|u_{1}-u_{2}\|_{Z_{\alpha}^{s_c}([0,1))}\right),\\
&\|\Phi^{(2)}_{\gamma ,w_{0}}(u_{1}, \overline{v_{1}})-\Phi^{(2)}_{\gamma ,w_{0}}(u_{2}, \overline{v_{2}})\|_{Z_{\gamma}^{s_c}([0,1))}
\\
&\leq 4Cr\left( \|u_{1}-u_{2}\|_{Z_{\alpha}^{s_c}([0,1))}+\|v_{1}-v_{2}\|_{Z_{\beta}^{s_c}([0,1))}\right).
\end{split}
\]
Therefore,
$\Phi$ is a contraction map on $X^{s_c}_{2r}([0,1))$. 
This implies the existence of the solution to the system (\ref{NLS_sys_torus}) and the uniqueness in the ball $X^{s_c}_{2r}([0,1))$. 
The uniqueness in $X^{s_c}([0,1))$ and
the Lipschitz continuity of the flow map can be obtained by the standard argument. \\

\noindent \underline{\rm Case (b)} (Large initial data):
Let $R>0$ be given and assume $(u_0,v_0,w_0)\in \mathcal{H}^{s_c}(\T^d)$ satisfy
\[
\max\{ \| u_0 \|_{H^{s_c}}, \| v_0 \|_{H^{s_c}}, \| w_0 \|_{H^{s_c}} \} \le R.
\]
Let $r\in (0,R)$ be a small constant to be chosen later.
Then,
there exists a dyadic number $K_0=K_0(u_0,v_0,w_0,r)$ such that
\[
\max\{ \| P_{\ge K_0}u_0 \|_{H^{s_c}}, \| P_{\ge K_0}v_0 \|_{H^{s_c}}, \| P_{\ge K_0}w_0 \|_{H^{s_c}} \} \le r. 
\]
We define 
\begin{equation}
%\label{Xrs_norm_T}
\notag
\widetilde{X}^{s_c}_{2R,2r}([0,T))
:=\left\{(u,v,w)\in X^{s_c}_{2R} ([0,T))
\left|\ 
(P_{\ge K_0}u,P_{\ge K_0}v,P_{\ge K_0}w)\in X^{s_c}_{2r} ([0,T))
 \right.\right\}.
\end{equation}
For $(u,v,w)\in \widetilde{X}^{s_c}_{2R,2r}([0,T))$, 
Propositions~\ref{Duam_est_T} and ~\ref{Duam_est_T_loc_critical} yield that
\[
\begin{split}
&\|\Phi^{(1)}_{\alpha ,u_{0}}(w, v)\|_{Z_{\alpha}^{s_c}([0,T))}\\
&\leq \|e^{i\alpha t\Delta}u_0\|_{Z^{s_c}_{\alpha}([0,T))}+\|I_{\alpha}^{(1)}(P_{\ge {K_0}}w,P_{\ge {K_0}}v)\|_{Z^{s}_{\alpha}([0,T))}\\
&\hspace{20ex}+\|I_{\alpha}^{(1)}(w,v)-I_{\alpha}^{(1)}(P_{\ge K_0}w,P_{\ge K_0}v)\|_{Z^{s_c}_{\alpha}([0,T))}\\
&\le \|u_0\|_{H^{s_c}}+C\|P_{\ge K_0}w\|_{Z^{s_c}_{\gamma}([0,T))}\|P_{\ge K_0}v\|_{Z^{s_c}_{\beta}([0,T))}
\\
&\quad +CT^{\eps}K_0^{\theta}\|w\|_{Z^{s_c}_{\gamma}([0,T))}\|v\|_{Z^{s_c}_{\beta}([0,T))}\\
&\le R+ 4Cr^2+ 4CT^{\eps}K_0^{\theta}R^2.
\end{split}
\]
Moreover,
we have
\[
\begin{split}
&\|P_{\ge K_0}\Phi^{(1)}_{\alpha ,u_{0}}(w, v)\|_{Z_{\alpha}^{s_c}([0,T))}\\
&\leq \|e^{i\alpha t\Delta}P_{\ge K_0}u_0\|_{Z^{s_c}_{\alpha}([0,T))}+\|I_{\alpha}^{(1)}(P_{\ge {K_0}}w,P_{\ge {K_0}}v)\|_{Z^{s}_{\alpha}([0,T))}\\
&\hspace{20ex}+\|I_{\alpha}^{(1)}(w,v)-I_{\alpha}^{(1)}(P_{\ge K_0}w,P_{\ge K_0}v)\|_{Z^{s_c}_{\alpha}([0,T))}\\
&\le \|P_{\ge K_0}u_0\|_{H^{s_c}}+C\|P_{\ge K_0}w\|_{Z^{s_c}_{\gamma}([0,T))}\|P_{\ge K_0}v\|_{Z^{s_c}_{\beta}([0,T))}
\\
&\quad +CT^{\eps}K_0^{\theta}\|w\|_{Z^{s_c}_{\gamma}([0,T))}\|v\|_{Z^{s_c}_{\beta}([0,T))}\\
&\le r+4Cr^2+4CT^{\eps}K_0^{\theta}R^2. 
\end{split}
\]
Here,
we choose $r\in (0,R)$ and $T\in (0,1]$ satisfying
\[
r\le \frac{1}{32C}, \quad
T^{\eps}\le \frac{r}{32CK_0^{\theta}R^2}
.
\]
Then, we obtain $\|\Phi^{(1)}_{\alpha ,u_{0}}(w, v)\|_{Z_{\alpha}^{s_c}([0,T))}\le 2R$, 
$\|P_{\ge K_0}\Phi^{(1)}_{\alpha ,u_{0}}(w, v)\|_{Z_{\alpha}^{s_c}([0,T))}\le 2r$, and
\[
\begin{split}
&\|\Phi^{(1)}_{\alpha ,u_{0}}(w_{1}, v_{1})-\Phi^{(1)}_{\alpha ,u_{0}}(w_{2}, v_{2})\|_{Z_{\alpha}^{s_c}([0,T))}
\\
&\leq (4Cr+6CT^{\eps}K_0^{\theta}R)\left( \|w_{1}-w_{2}\|_{Z_{\gamma}^{s_c}([0,T))}+\|v_{1}-v_{2}\|_{Z_{\beta}^{s_c}([0,T))}\right)\\
&\le \frac 5{16} \left( \|w_{1}-w_{2}\|_{Z_{\gamma}^{s_c}([0,T))}+\|v_{1}-v_{2}\|_{Z_{\beta}^{s_c}([0,T))}\right). 
\end{split}
\]
The similar estimates for $\|\Phi^{(1)}_{\beta ,v_{0}}(\overline{w}, u)\|_{Z_{\beta}^{s_c}([0,T))}$ 
and $\|\Phi^{(2)}_{\gamma ,w_{0}}(u, \overline{v})\|_{Z_{\gamma}^{s_c}([0,T))}$ 
can be obtained. 
Therefore, $\Phi$ is a contraction map on $\widetilde{X}^{s_c}_{2R,2r}([0,T))$. 
\end{proof}

By using Proposition~\ref{HL_est_loc_T}, ~\ref{HL_est_loc_T_22}, or ~\ref{HL_est_loc_T_reso},
instead of Proposition~\ref{HL_est_T} in the proof of Proposition~\ref{Duam_est_T},
we get the following.

\begin{prop}\label{Duam_est_T_loc}
Let $d\ge 1$ and $0<T\le 1$.
If one of 
\begin{enumerate}
\item[{\rm (i)}] $\mu>0$ and $s>\max \{ s_c, 0\};$
\item[{\rm (ii)}] $\mu = 0$, $s>s_c$, and $s\ge 1;$
\item[{\rm (iii)}] $\mu < 0$, $\widetilde{\kappa} \neq 0$, $s> \max \{ s_c, 1 \}$
\end{enumerate}
is satisfied, 
then there exists $\eps >0$, such that we have
\begin{align*}
&\|I_{\alpha}^{(1)}(w,v)\|_{Z^{s}_{\alpha}([0,T))}\lesssim T^{\eps}\|w\|_{Y^{s}_{\gamma}([0,T))}\|v\|_{Y^{s}_{\beta}([0,T))},\\
&\|I_{\beta}^{(1)}(\overline{w},u)\|_{Z^{s}_{\beta}([0,T))}\lesssim T^{\eps}\|w\|_{Y^{s}_{\gamma}([0,T))}\|u\|_{Y^{s}_{\alpha}([0,T))},\\
&\|I_{\gamma}^{(2)}(u,\overline{v})\|_{Z^{s}_{\gamma}([0,T))}\lesssim T^{\eps}\|u\|_{Y^{s}_{\alpha}([0,T))}\|v\|_{Y^{s}_{\beta}([0,T))}.
\end{align*} 
\end{prop}

Theorem \ref{wellposed_sub_T} except for $\mu<0$ and $s=1$ follows from Proposition \ref{Duam_est_T_loc}.
Since this is a standard contraction argument,
we omit the details here.

\subsection{The case $\mu<0$, $\widetilde{\kappa}\ne 0$, $d=1,2$, and $s=1$}
\label{subsec:ds21}
In this subsection, 
we prove Theorem~\ref{wellposed_sub_T} 
for the case $\mu<0$, $\widetilde{\kappa}\ne 0$, $d=1,2$, and $s=1$. 
We first give the definition of the solution space. 
\begin{defn}
We define $X$ as the space of all vector valued functions $F:\R \times \T^d \to \C^d$ such that $F(\cdot, x) \in \mathcal{S}(\R)$ for all $x \in \T^d$ and the map $x \mapsto F(\cdot, x)$ is $C^{\infty}$.

Let $s, b\in \R$, $\sigma \in \R\backslash \{0\}$. \\
(i)\ For $1 \leq p < \infty$, we define the function space $X^{s,b,p}_{\sigma}$ as the completion of $X$ with the norm
\[
\|u\|_{X^{s,b,p}_{\sigma}}=\biggl\{\sum_{N\geq 1}N^{2s}
\Bigl(\sum_{M\geq 1}M^{pb}\|Q_{M}^{\sigma}P_{N}u\|_{L^2}^p\Bigr)^\frac{2}{p}\biggr\}^{\frac{1}{2}}.
\]
Similarly, we define the function space $X^{s,b,\infty}_{\sigma}$ as the completion of $X$ with the norm
\[
\|u\|_{X^{s,b,\infty}_{\sigma}}=\biggl\{\sum_{N\geq 1}N^{2s}
\Bigl(\sup_{M\geq 1}M^{b}\|Q_{M}^{\sigma}P_{N}u\|_{L^2}\Bigr)^2\biggr\}^{\frac{1}{2}}. 
\]
(ii)\ For $T>0$, we define the time localized space $X^{1,\frac12,1}_{\sigma, T}$ (see Remark \ref{rem:ristI}) as
\[
X^{1,\frac12,1}_{\sigma, T} = X^{1,\frac12,1}_{\sigma}([0,T)).
\]
\end{defn}

Recall that $\chi \in C_0^{\infty}((-2,2))$ is non-negative with $\chi(t) = 1$ for $|t| \leq 1$. We define $\chi_T(t) = \chi( \frac tT)$.
The following linear estimates hold.
See Propositions 5.2 and 5.3 in \cite{BHHT09} for the proof.

\begin{prop}\label{Xsb_linear_est}
Let $\sigma\in \R\backslash \{0\}$, $b \in (0,\frac{1}{2})$, and 
$0<T\le 1$. 
\begin{enumerate}
\item[(1)] For any $\varphi \in H^1 (\T^d)$, we have
\[
\|e^{it\sigma \Delta}\varphi\|_{X^{1,\frac{1}{2},1}_{\sigma,T}}
\lesssim \|\varphi\|_{H^1}.
\]
\item[(2)] For any $F \in X^{1,-\frac12,1}_{\sigma}$, we have
\[
\biggl\|\chi (t) \int_{0}^{t}e^{i(t-t')\sigma \Delta}F(t')dt'\biggr\|_{X^{1,\frac{1}{2},1}_{\sigma}}
\lesssim \|F\|_{X^{1,-\frac12,1}_{\sigma}}.
\]
\item[(3)] For $u \in X^{1,\frac{1}{2},1}_{\sigma}$, we have
\[
\|\chi_T(t) u\|_{X^{1,b,1}_{\sigma}}\lesssim T^{\frac{1}{2}-b}\|u\|_{X^{1,\frac{1}{2},1}_{\sigma}}. 
\]
\end{enumerate}
\end{prop}

The following nonlinear estimates play a crucial role in the proof of the well-posedness. 
\begin{prop}\label{prop:Xsb_trilinear_est1}
Let $d \in \{ 1,2 \}$ and $\alpha$, $\beta$, $\gamma \in \R\setminus \{0\}$ satisfy $\mu < 0$ and $\widetilde \kappa \not=0$. Then, there exists $\varepsilon >0 $ such that
\begin{align*}
&\|I_{\alpha}^{(1)}(w,v)\|_{X^{1,\frac{1}{2},1}_{\alpha, T}}  \lesssim T^{\varepsilon} \|w\|_{X^{1,\frac{1}{2},1}_{\gamma, T}}\|v\|_{X^{1,\frac{1}{2},1}_{\beta, T}},\\
&\|I_{\beta}^{(1)}(\overline{w},u)\|_{X^{1,\frac{1}{2},1}_{\beta, T}} \lesssim T^{\varepsilon}\|w\|_{X^{1,\frac{1}{2},1}_{\gamma, T}}\|u\|_{X^{1,\frac{1}{2},1}_{\alpha, T}},\\
&\|I_{\gamma}^{(2)}(u,\overline{v})\|_{X^{1,\frac{1}{2},1}_{\gamma, T}} \lesssim T^{\varepsilon}\|u\|_{X^{1,\frac{1}{2},1}_{\alpha, T}}\|v\|_{X^{1,\frac{1}{2},1}_{\beta, T}}
\end{align*}
for $0 < T \leq 1$.
\end{prop}
It follows from Proposition~\ref{prop:Xsb_trilinear_est1} and Proposition~\ref{Xsb_linear_est} (1) that the standard contraction mapping argument implies the well-posedness. 
Thus, we  focus on Proposition~\ref{prop:Xsb_trilinear_est1}. 
%We refer to the proof of Theorem 1.1 in \cite{HK} for more details.

To prove Proposition~\ref{prop:Xsb_trilinear_est1}, it is enough to show the following proposition.
\begin{prop}\label{prop:Xsb_trilinear_est2}
Let $d \in \{ 1,2 \}$ and $\alpha$, $\beta$, $\gamma \in \R\setminus \{0\}$ satisfy $\mu < 0$ and $\widetilde \kappa \not=0$. Then, there exists $\varepsilon >0 $ such that
\begin{align*}
&\|\chi_T(t) (\nabla \cdot w) v \|_{X^{1,-\frac{1}{2},1}_{\alpha}}  \lesssim T^{\varepsilon} \|w\|_{X^{1,\frac{1}{2},1}_{\gamma}}\|v\|_{X^{1,\frac{1}{2},1}_{\beta}},\\
&\|\chi_T(t) (\nabla \cdot \overline{w}) u\|_{X^{1,-\frac{1}{2},1}_{\beta}} \lesssim T^{\varepsilon}\|w\|_{X^{1,\frac{1}{2},1}_{\gamma}}\|u\|_{X^{1,\frac{1}{2},1}_{\alpha}},\\
&\|\chi_T(t) \nabla (u \cdot \overline{v})\|_{X^{1,-\frac{1}{2},1}_{\gamma}} \lesssim T^{\varepsilon}\|u\|_{X^{1,\frac{1}{2},1}_{\alpha}}\|v\|_{X^{1,\frac{1}{2},1}_{\beta}}
\end{align*}
for $0 <T \le 1$.
\end{prop}

Let us see that Proposition~\ref{prop:Xsb_trilinear_est2} implies Proposition~\ref{prop:Xsb_trilinear_est1}.
\begin{proof}[Proof of Proposition~\ref{prop:Xsb_trilinear_est1}]
We only consider the first estimate:
\[
\|I_{\alpha}^{(1)}(w,v)\|_{X^{1,\frac{1}{2},1}_{\alpha, T}}  \lesssim T^{\varepsilon} \|w\|_{X^{1,\frac{1}{2},1}_{\gamma, T}}\|v\|_{X^{1,\frac{1}{2},1}_{\beta, T}}.
\]
To see this, we take the functions $W \in X^{1,\frac12,1}_{\gamma}$ and $V \in X^{1,\frac12,1}_{\beta}$ so that 
\begin{align*}
 W(t) & = w(t) \quad \mathrm{if} \ t \in [0,T), \qquad \|W \|_{X^{1,\frac12,1}_{\gamma}} \leq 2 \|w\|_{X^{1,\frac{1}{2},1}_{\gamma, T}},\\
V(t) & = v(t) \quad \, \mathrm{if} \ t \in [0,T), \qquad \  \|V \|_{X^{1,\frac12,1}_{\beta}} \leq 2 \|v\|_{X^{1,\frac{1}{2},1}_{\beta, T}}.
\end{align*}
It is known that if $(\nabla \cdot W) V \in X^{1,-\frac{1}{2},1}_{\alpha}$, then we have $I_{\alpha}^{(1)}(W,V) \in C(\R;H^1(\T^d))$. See Lemma~2.2 in \cite{GTV97}. Since $I_{\alpha}^{(1)}(W,V)(t) = I_{\alpha}^{(1)}(w,v)(t)$ if $t \in [0,T)$, we have
\[
\|I_{\alpha}^{(1)}(w,v)\|_{X^{1,\frac{1}{2},1}_{\alpha, T}} \leq \|\chi_T(t) I_{\alpha}^{(1)}(W,V)\|_{X^{1,\frac{1}{2},1}_{\alpha}}.
\]
Therefore, we deduce from Propositions~\ref{Xsb_linear_est} and \ref{prop:Xsb_trilinear_est2} that
\begin{align*}
\|I_{\alpha}^{(1)}(w,v)\|_{X^{1,\frac{1}{2},1}_{\alpha, T}} & \leq \|\chi_T(t) I_{\alpha}^{(1)}(W,V)\|_{X^{1,\frac{1}{2},1}_{\alpha}}\\
& \lesssim \|\chi_T(t) (\nabla \cdot W) V \|_{X^{1,-\frac{1}{2},1}_{\alpha}}\\
&  \lesssim T^{\varepsilon} \|W\|_{X^{1,\frac{1}{2},1}_{\gamma}}\|V\|_{X^{1,\frac{1}{2},1}_{\beta}}\\
& \lesssim T^{\varepsilon} \|w\|_{X^{1,\frac{1}{2},1}_{\gamma,T}}\|v\|_{X^{1,\frac{1}{2},1}_{\beta,T}}.
\end{align*}
This completes the proof.
\end{proof}
\begin{proof}[Proof of Proposition~\ref{prop:Xsb_trilinear_est2}]
We focus on the case $d=2$,
since the case $d=1$ is easily treated. 
We only consider the first estimate:
\[
\|\chi_T(t) (\nabla \cdot w) v \|_{X^{1,-\frac{1}{2},1}_{\alpha}}  \lesssim T^{\varepsilon} \|w\|_{X^{1,\frac{1}{2},1}_{\gamma}}\|v\|_{X^{1,\frac{1}{2},1}_{\beta}}.
\]
Set
\[
(u_1, u_2,u_3)=(u, \overline v,\overline  w),
\quad
(\sigma_1,\sigma_2,\sigma_3) = (\alpha,-\beta,-\gamma),
\]
for simplicity.

\noindent \underline{Case $\alpha \not= \beta$}:
Let us consider the case $\alpha \not= \beta$. 
By duality, we have
\[
\| \chi_T(t) (\nabla \cdot u_3) u_2 \|_{X^{1,-\frac{1}{2},1}_{-\sigma_1}}
= \sup_{\|u_1\|_{X^{-1,\frac12, \infty}_{\sigma_1}}=1} \bigg| \int_\R \int_{\T^2} \chi_T(t) u_1 u_2 (\nabla \cdot u_3) dxdt \bigg|.
\]
From the same argument as in \eqref{defpsi},
there exists $\psi_T\in C_0^{\infty}(\R)$ such that 
\[
\eta (t)\psi_T(t)^2=1
\]
on $[-2T,2T]$ for $0<T\le 1$.
We can write as follows:
\[
\int_\R \int_{\T^2} \chi_T(t) u_1 u_2 (\nabla \cdot u_3) dxdt
=
\int_\R \int_{\T^2} \eta(t) u_1 u_{2,T} (\nabla \cdot u_{3,T}) dxdt,
\]
where
$u_{2,T} := \psi_T \chi_T u_2$ and
$u_{3,T} := \psi_T u_3$.
Hence, by the dyadic decompositions and Propositions~\ref{Xsb_linear_est} (3),
it suffices to show that
there exits $\eps>0$ such that,
for $N_{\min} \ll N_{\max}$,
\begin{equation}\label{est:01_4.3}
\begin{aligned}
&\sum_{M_1,M_2,M_3}
\bigg| \int_\R \int_{\T^2}\eta(t) \Bigl( \prod_{j=1}^3  P_{N_j} Q_{M_j}^{\sigma_j}u_j \Bigr) dxdt\bigg|
\\
&\lesssim
N_{\min}^{\frac14} N_{\max}^{-1}
\|u_1 \|_{X^{0,\frac12,\infty}_{\sigma_1}}
\Big(
\|u_2\|_{X^{0,\frac12,\infty}_{\sigma_2}}
\|u_3\|_{X^{0,\frac12-\eps,\infty}_{\sigma_3}}
+
\|u_2\|_{X^{0,\frac12-\eps,\infty}_{\sigma_2}}
\|u_3\|_{X^{0,\frac12,\infty}_{\sigma_3}}
\Big)
,
\end{aligned}
\end{equation}
and that for $N_1 \sim N_2 \sim N_3$,
\begin{equation}
\label{est:02_4.3}
\begin{aligned}
&\sum_{M_1,M_2,M_3}\bigg| \int_\R \int_{\T^2} \eta(t)  \Bigl( \prod_{j=1}^3 P_{N_j} Q_{M_j}^{\sigma_j} u_j \Bigr) dx dt\bigg|
\\
&\lesssim
\|u_1 \|_{X^{0,\frac12,\infty}_{\sigma_1}}
\Big(
\|u_2\|_{X^{0,\frac12,\infty}_{\sigma_2}}
\|u_3\|_{X^{0,\frac12-\eps,\infty}_{\sigma_3}}
+
\|u_2\|_{X^{0,\frac12-\eps,\infty}_{\sigma_2}}
\|u_3\|_{X^{0,\frac12,\infty}_{\sigma_3}}
\Big)
.
\end{aligned}
\end{equation}
We consider \eqref{est:01_4.3}. Since $N_{\min} \ll N_{\max}$ and $(\sigma_1+\sigma_2)(\sigma_2 + \sigma_3)(\sigma_3 + \sigma_1) \not= 0$, as in the proof of Proposition~\ref{HL_est_T}, we may assume that $M_{\max} \gtrsim N_{\max}^2$. In addition, for each $j=1,2,3$, we may assume that the spatial frequency of $P_{N_j} Q_{M_j}^{\sigma_j}u_j$ is contained in a ball of radius $\sim N_{\min}$. 
Then, for $j=1,2,3$, by
%using the Strichartz estimate (Proposition \ref{Stri_est_T}),
\eqref{eta_lp_up_est_C},
we have
\[
\big\| \eta(t)^{\frac 14}  P_{N_j} Q_{M_j}^{\sigma_j}u_j \big\|_{L_{t,x}^{\frac{14}{3}}} \lesssim M_j^{\frac12}N_{\min}^{\frac17} \|P_{N_j} Q_{M_j}^{\sigma_j}u_j\|_{L_{t,x}^2}.
\]
Moreover, a trivial bound holds:
\[
\big\| \eta(t)^{\frac 14}  P_{N_j} Q_{M_j}^{\sigma_j}u_j \big\|_{L_{t,x}^2}
\lesssim \|P_{N_j} Q_{M_j}^{\sigma_j}u_j\|_{L_{t,x}^2}.
\]
By interpolating two estimates above,
we have
\[
\big\| \eta(t)^{\frac 14} P_{N_j} Q_{M_j}^{\sigma_j} u_j \big\|_{L_{t,x}^4}
\lesssim M_j^{\frac{7}{16}}N_{\min}^{\frac18} \|P_{N_j} Q_{M_j}^{\sigma_j}u_j\|_{L_{t,x}^2}.
\]

Suppose that $M_1=M_{\max}$.
We have
\begin{align*}
& \sum_{M_1 \gtrsim N_{\max}^2} \sum_{M_2,M_3} \bigg| \int_\R \int_{\T^2} \eta (t) \Bigl( \prod_{j=1}^3  P_{N_j} Q_{M_j}^{\sigma_j}u_j \Bigr) dxdt\bigg|\\
\lesssim & \sum_{M_1 \gtrsim N_{\max}^2} \sum_{M_2,M_3}
\big\| \eta(t)^{\frac 12} P_{N_1} Q_{M_1}^{\sigma_1}u_1 \big\|_{L_{t,x}^2}
\\
&\qquad \times
\big\| \eta(t)^{\frac 14} P_{N_2} Q_{M_2}^{\sigma_2}u_2 \big\|_{L_{t,x}^4}
\big\| \eta(t)^{\frac 14} P_{N_3} Q_{M_3}^{\sigma_3}u_3 \big\|_{L_{t,x}^4}\\
\lesssim & \sum_{M_1 \gtrsim N_1^2}
\Big( \frac{N_{\max}^2}{M_1} \Big)^{\frac12}N_{\max}^{-1} \Bigl( \sup_{M_1} M_1^{\frac12}
\|P_{N_1} Q_{M_1}^{\sigma_1}u_1\|_{L_{t,x}^2} \Bigr)
\\
&\qquad \times
N_{\min}^{\frac14} \| u_2\|_{X^{0,\frac{7}{16}+ \varepsilon,\infty}_{\sigma_2}} \| u_3\|_{X^{0,\frac{7}{16}+\eps,\infty}_{\sigma_3}}
\\
\lesssim
&
N_{\min}^{\frac14} N_{\max}^{-1}
 \|u_1\|_{X^{0,\frac12,\infty}_{\sigma_1}}
\prod_{j=2}^3 \|u_j\|_{X^{0,\frac 7{16}+\eps,\infty}_{\sigma_j}}
\end{align*}
for any $\eps>0$.
The other cases can be handled in a similar way. 

We turn to the proof of \eqref{est:02_4.3}. In the case $M_{\max} \gec N_{1}$, similarly to the proof of \eqref{est:01_4.3}, the Strichartz estimates imply \eqref{est:02_4.3}. 
For the case $M_{\max} \ll N_{1}$, Theorem~\ref{thm:Triest_4.3} readily yields \eqref{est:02_4.3}. 

\noindent \underline{Case $\alpha = \beta$}:
%As in the case $\alpha \not= \beta$, we put $(u_2,u_3)=(v,w)$, $(\sigma_1,\sigma_2,\sigma_3) = (-\alpha,\beta,\gamma)$ and consider the dual estimate. 
Let $N_j$ be the size of spatial frequency of $u_j$.
By \eqref{modul_est_T},
for the cases $N_1 \ll N_2 \sim N_3$, $N_2 \ll N_1 \sim N_3$, and $N_1 \sim N_2 \sim N_3$,
we can prove the desired bound in the same manner as in the case $\alpha \not= \beta$.
Thus, we only need to consider the case $N_3 \ll N_1 \sim N_2$. In this case, the difference from the case $\alpha \not= \beta$ is that $M_{\max} \gtrsim N_{1}^2$ does not necessarily hold true. While, the derivative hits $u_3$ whose frequency is smaller than that of $u_1$, $u_2$. 

By duality, it is enough to show
\begin{equation}
\label{est:03_4.3}
\begin{split}
\sum_{M_1,M_2,M_3} \bigg| & \int_\R \int_{\T^2} \eta(t) \bigl(Q^{\sigma_1}_{M_1} P_{N_1} u_1\bigr) \bigl( Q^{\sigma_2}_{M_2} P_{N_2} u_2 \bigr) \bigl( Q^{\sigma_3}_{M_3} P_{< N_1} \langle \nabla \rangle u_3 \bigr)dxdt\bigg|\\
& \lesssim
\|u_1\|_{X^{0,\frac12,\infty}_{\sigma_1}}
\|u_2\|_{X^{0,\frac12-\eps,\infty}_{\sigma_2}}
\|u_3\|_{X^{1,\frac12-\eps,\infty}_{\sigma_3}}
\end{split}
\end{equation}
for some $\eps>0$.
If $M_{\max} \gec N_1$, in the same way as in the case $\alpha \not= \beta$, the Strichartz estimate yields \eqref{est:03_4.3}. In the case $M_{\max} \ll N_1$, by using Theorem~\ref{thm:Triest_4.3} and
the Littlewood-Paley theorem,
we have
\[
\begin{split}
\sum_{M_1,M_2,M_3} \bigg| & \int_\R \int_{\T^2} \eta(t) \bigl(Q^{\sigma_1}_{M_1} P_{N_1} u_1\bigr) \bigl( Q^{\sigma_2}_{M_2} P_{N_2} u_2 \bigr) \bigl( Q^{\sigma_3}_{M_3} P_{< N_1} \langle \nabla \rangle u_3 \bigr)dxdt\bigg|\\
& \lesssim  \| u_1\|_{X^{0,\frac12- \varepsilon,\infty}_{\sigma_1}} \| u_2\|_{X^{0,\frac12- \varepsilon,\infty}_{\sigma_2}} \|\langle \nabla \rangle u_3\|_{X^{0,\frac12- \varepsilon,\infty}_{\sigma_3}}\\
& \lec \|u_1\|_{X^{0,\frac12-\eps,\infty}_{\sigma_1}} \|u_2\|_{X^{0,\frac12-\eps,\infty}_{\sigma_2}} \|u_3\|_{X^{1,\frac12- \varepsilon,\infty}_{\sigma_3}},
\end{split}
\]
as desired.
\end{proof}

\section{Proof of ill-posedness}
\label{sec:IP}

In this section, we prove Theorems~\ref{thm:IP} and ~\ref{thm:IP2}. 
Let $k$ be a rational number and $N \gg 1$ such that
$k N$ is an integer.
We consider a solution of the form
\begin{equation}
\begin{aligned}
u(t,x) &= (f(t) e^{-i t \al k^2 N^2} e^{ikNx_1}, 0, \dots, 0) \\
v(t,x) &= (g(t) e^{-i t \be (k-1)^2 N^2} e^{i (k-1) N x_1}, 0, \dots, 0) \\
w(t,x) &= (h(t) e^{-i t \ga N^2} e^{iNx_1}, 0, \dots, 0)
\end{aligned}
\label{uvwaa}
\end{equation}
for $t \ge 0$ and $x=(x_1,\cdots,x_d)\in \R^d$.
From this choice,
Theorems~\ref{thm:IP} and ~\ref{thm:IP2} for the multidimensional cases follow that for the one dimensional case.
In what follows,
we only consider the case $d=1$.

%If $k$ satisfies \eqref{keq1},
By \eqref{NLS_sys_torus},
$f,g,h$ satisfy the following system of ordinary differential equations:
\begin{equation}
\left\{
\begin{aligned}
&f'(t) = - N g(t) h(t) e^{it (\al k^2 - \be (k-1)^2-\ga) N^2} , && t>0, \\
&g'(t) = N f(t) \cj{h(t)} e^{-it (\al k^2 - \be (k-1)^2-\ga) N^2}, && t>0, \\
&h'(t) = N f(t) \cj{g(t)} e^{-it (\al k^2 - \be (k-1)^2-\ga) N^2} , && t>0. \\
\end{aligned}
\right.
\label{ODE1e}
\end{equation}
When $k$ is a solution to
\begin{equation}
\label{keq1}
\al k^2 - \be (k-1)^2-\ga=0,
\end{equation}
the oscillation part in \eqref{ODE1e} vanishes.
Note that
\eqref{keq1} is equivalent to 
\[
(\al-\be) k^2 + 2 \be k - (\be+\ga)=0.
\]
Namely,
\begin{equation}
k
= \begin{cases}
\frac{-\be \pm \sqrt{\be^2+ (\al-\be)(\be+\ga)}}{\al-\be}
= \frac{-\be \pm \sqrt{\al \be - \be \ga + \al \ga}}{\al-\be}
& \text{if } \al-\be \neq 0,
\\
\frac{\be+\ga}{2\be}
& \text{if } \al-\be = 0.
\end{cases}
\label{conk}
\end{equation}
A direct calculation shows that
\begin{equation}
\frac{d}{dt}
\big( |f(t)|^2 + |g(t)|^2 \big)
=
\frac{d}{dt}
\big( |f(t)|^2 + |h(t)|^2 \big)
=0.
\label{cons1}
\end{equation}
This is a reflection of the $L^2$-conservation law of \eqref{NLS_sys_torus}.

If $k$ satisfies \eqref{keq1} and
the initial data $f(0)$, $g(0)$, and $h(0)$ are real,
then $f,g,h$ are real-valued.
In particular, they satisfy
\begin{equation}
\left\{
\begin{aligned}
&f'(t) = - N g(t) h(t), && t>0, \\
&g'(t) = N f(t) h(t), && t>0, \\
&h'(t) = N f(t) g(t), && t>0. \\
\end{aligned}
\right.
\label{ODE1f}
\end{equation}

%We first prove Theorem \ref{thm:IP}.

\subsection{The case $\be+\ga=0$ and $s>0$}
\label{Subsec:a21}

For $N \gg 1$ and $0<\de \ll 1$,
we set
\begin{equation}
u_0(x)= \de, \quad
v_0(x)= 0, \quad
w_0(x)= \de N^{-s} e^{iNx}.
\label{iniaa}
\end{equation}
Then,
we have
\begin{equation}
\| u_0 \|_{H^s}
+
\| v_0 \|_{H^s}
+
\| w_0 \|_{H^s}
\le 2 \de.
\label{iniss}
\end{equation}

It follows from
$f(0)=\de$, $g(0)=0$, $h(0) = \de N^{-s}$,
and \eqref{cons1}
that
\[
f(t)^2 + 2 g(t)^2 - h(t)^2
%= 2(f(t)^2 + g(t)^2) - (f(t)^2+ h(t)^2)
= \de^2 (1-N^{-2s}).
\]
With
\eqref{ODE1f},
we have the following Cauchy problem:
\[
\left\{
\begin{aligned}
&g''(t) = -N^2 g(t) \big( 2 g(t)^2 - \de^2 (1-N^{-2s}) \big), && t>0, \\
&(g(0),g'(0)) = (0, \de^2 N^{1-s}).
\end{aligned}
\right.
\]
It follows from $s>0$ that $N^{-2s} \ll 1$.
We set 
\begin{equation}
\ka (t) = \frac{\sqrt 2}{\de \sqrt{1-N^{-2s}}} g \Big( \frac t{\de N \sqrt{1-N^{-2s}}} \Big).
\label{alf0}
\end{equation}
Then, $\ka$ satisfies
\[
\left\{
\begin{aligned}
&\ka''(t) = - \ka (t) \big( \ka (t)^2 - 1 \big), && t>0, \\
&(\ka(0),\ka'(0)) = \Big( 0, \frac{\sqrt 2 N^{-s}}{1-N^{-2s}} \Big).
\end{aligned}
\right.
\]

For simplicity, we set $x(t)=\ka(t)$ and $y(t)=\ka'(t)$.
Then, $x(t)$ and $y(t)$ satisfy
\[
\left\{
\begin{aligned}
&x'(t)=y(t), & t>0, \\
&y'(t)=-x(t)^3+x(t), & t>0, \\
&(x(0),y(0)) = \Big( 0, \frac{\sqrt 2 N^{-s}}{1-N^{-2s}} \Big).
\end{aligned}
\right.
\]
The solution
$(x(t), y(t))$
is on the curve
\[
\frac{y^2}2 + \frac{x^4}4 - \frac{x^2}2
= \frac{N^{-2s}}{(1-N^{-2s})^2} =: E_0.
\]

\begin{figure}[H]
\begin{tikzpicture}[scale=1.7, samples=100]
\draw[thick, ->] (-1.8,0) -- (1.8,0) node [right] {$x$};%x軸
\draw[thick, ->] (0,-1) -- (0,1) node [above] {$y$};%y軸
\coordinate (O) at (0,0) node at (O) [below left] {$O$};
%%%
%%% E_0 =1/16
%%%
\draw[thick, domain={0}:{1.4562}] plot (\x,{sqrt(1/8+\x*\x-\x*\x*\x*\x/2)});
\draw[thick, domain={0}:{0.7}] plot ({sqrt(1+sqrt(1+(1/4)-2*\x*\x))},\x);
%%%
\draw[thick, domain={0}:{1.4562}] plot (-\x,{sqrt(1/8+\x*\x-\x*\x*\x*\x/2)});
\draw[thick, domain={0}:{0.7}] plot ({-sqrt(1+sqrt(1+(1/4)-2*\x*\x))},\x);
%%%
\draw[thick, domain={0}:{1.4562}] plot (\x,{-sqrt(1/8+\x*\x-\x*\x*\x*\x/2)});
\draw[thick, domain={0}:{0.7}] plot ({sqrt(1+sqrt(1+(1/4)-2*\x*\x))},-\x);
%%%
\draw[thick, domain={0}:{1.4562}] plot (-\x,{-sqrt(1/8+\x*\x-\x*\x*\x*\x/2)});
\draw[thick, domain={0}:{0.7}] plot ({-sqrt(1+sqrt(1+(1/4)-2*\x*\x))},-\x);
%%%
\draw[dashed] plot (-1,{-sqrt(5/8)}) -- (-1,{sqrt(5/8)});
%%%
\draw[dashed] plot (1,{-sqrt(5/8)}) -- (1,{sqrt(5/8)});
\fill (1,0) circle(0.04);
\fill (-1,0) circle(0.04);
\node at (-1.2,-0.18) {$-1$};
\node at (1.1,-0.18) {$1$};
\end{tikzpicture}
\caption{$\frac{y^2}2+\frac{x^4}4-\frac{x^2}2=E_0 \, (>0)$}
\label{fig:3}
\end{figure}

Set
\begin{equation}
t_\ast = \inf \Big\{ t >0 \mid x(t)= 1 \Big\}.
\label{time1}
\end{equation}

\begin{lemm}
\label{lem:tast0}
For $s>0$ and $N \gg 1$,
we have
$t_\ast \lec \log N$.
\end{lemm}

\begin{proof}
Note that $x(t)$ and $y(t)$ are increasing for $0<t<t_\ast$,
since $y(0)>0$.
By
$\frac{dt}{dx} = \frac{1}{\sqrt{2E_0 - \frac{x^4}2+x^2}}$,
we have
\begin{align*}
t_\ast
&= \int_0^1 \frac{d x}{\sqrt{2E_0 - \frac{x^4}2 + x^2}}
= \sqrt 2 \int_0^1 \frac{d x}{\sqrt{-(x^2-1-\sqrt{1+4E_0})(x^2-1+\sqrt{1+4E_0})}} \\
&\lec
\int_0^1
\frac{dx}{\sqrt{x^2-1+\sqrt{1+4E_0}}}
= \bigg[ \log \bigg| x+\sqrt{x^2-1+\sqrt{1+4E_0}} \bigg| \bigg]_0^1
\\
&=
\log \Big( 1 + \sqrt[4]{1+4E_0} \Big)
-
\log
\underbrace{\sqrt{-1+\sqrt{1+4E_0}}}
_{= \sqrt{\frac{4E_0}{1+\sqrt{1+4E_0}}}}
\\
&\lec
\log E_0^{-\frac 12}
\sim \log N.
\qedhere
\end{align*}
\end{proof}

Set
\begin{equation}
T = \frac{t_\ast}{\de N \sqrt{1-N^{-2s}}},
\quad \de = (\log N)^{-1}.
\label{Ttime1}
\end{equation}
It follows from \eqref{uvwaa} with $d=1$ and $k=0$, \eqref{alf0}, and \eqref{time1} that
\[
\| v(T) \|_{H^s}
= N^s |g(T)|
= N^s \frac{\de \sqrt{1-N^{-2s}}}{\sqrt 2} |\ka ( t_\ast)|
\sim
N^s (\log N)^{-1} \gg 1,
\]
provided that $s>0$ and $N \gg 1$.
From Lemma \ref{lem:tast0} and \eqref{Ttime1},
we also have
\[
T \lec N^{-1} (\log N)^2 \ll 1.
\]
With \eqref{iniss} and $\de = (\log N)^{-1}$,
we obtain the norm inflation in $H^s(\T)$ for $\be+\ga=0$ and $s>0$.

\subsection{The case $\be+\ga=0$ and $s<0$}

Next,
we consider the case
\[
\be+\ga=0, \quad s<0.
\]
For $N \gg 1$ and $0<\de \ll 1$,
we set
\begin{equation}
u_0(x)= 0, \quad
v_0(x)= \de N^{-s} e^{-iNx}, \quad
w_0(x)= \de N^{-s} e^{iNx}.
\notag
\end{equation}
Then,
we have
\begin{equation}
\| u_0 \|_{H^s}
+
\| v_0 \|_{H^s}
+
\| w_0 \|_{H^s}
\le 2 \de.
\label{inissa}
\end{equation}

It follows from
$f(0)=0$, $g(0)=\de N^{-s}$, $h(0) = \de N^{-s}$,
and \eqref{cons1}
that
\[
2f(t)^2 + g(t)^2 + h(t)^2
= 2\de^2 N^{-2s}.
\]
With
\eqref{ODE1f},
we have the following Cauchy problem:
\[
\left\{
\begin{aligned}
&f''(t) = 2N^2 f(t) \big( f(t)^2 - \de^2 N^{-2s} \big), && t>0, \\
&(f(0),f'(0)) = (0, -\de^2 N^{1-2s}).
\end{aligned}
\right.
\]
We set 
\begin{equation}
\ka (t) = \frac{1}{\de N^{-s}} f \Big( \frac t{\sqrt 2 \de N^{1-s}} \Big).
\label{alf0a}
\end{equation}
Then, $\ka$ satisfies
\[
\left\{
\begin{aligned}
&\ka''(t) = \ka (t) \big( \ka (t)^2 - 1 \big), && t>0, \\
&(\ka(0),\ka'(0)) = \Big( 0, -\frac{1}{\sqrt 2} \Big).
\end{aligned}
\right.
\]

For simplicity, we set $x(t)=\ka(t)$ and $y(t)=\ka'(t)$.
Then, $x(t)$ and $y(t)$ satisfy
\[
\left\{
\begin{aligned}
&x'(t)=y(t), & t>0, \\
&y'(t)=x(t)^3-x(t),  & t>0, \\
&(x(0),y(0)) = \Big( 0, -\frac{1}{\sqrt 2} \Big).
\end{aligned}
\right.
\]
Note that this Cauchy problem is independent of $N$ and $\de$.
The solution
$(x(t), y(t))$
is on the curve
\[
\frac{y^2}2 - \frac{x^4}4 + \frac{x^2}2
= \frac 14.
\]

\begin{figure}[H]
\begin{tikzpicture}[scale=1.7, samples=100]
\draw[thick, ->] (-1.8,0) -- (1.8,0) node [right] {$x$};%x軸
\draw[thick, ->] (0,-1) -- (0,1) node [above] {$y$};%y軸
\coordinate (O) at (0,0) node at (O) [below left] {$O$};
%%%
\draw[thick, domain={-1.5}:{1.5}] plot (\x,{(\x*\x-1)/sqrt(2)});
\draw[thick, domain={-1.5}:{1.5}] plot (\x,{-(\x*\x-1)/sqrt(2)});
\fill (1,0) circle(0.04);
\fill (-1,0) circle(0.04);
%%%
\node at (-1,-0.3) {$-1$};
\node at (1,-0.3) {$1$};
\end{tikzpicture}
\caption{$\frac{y^2}2-\frac{x^4}4+\frac{x^2}2=\frac 14$}
\end{figure}
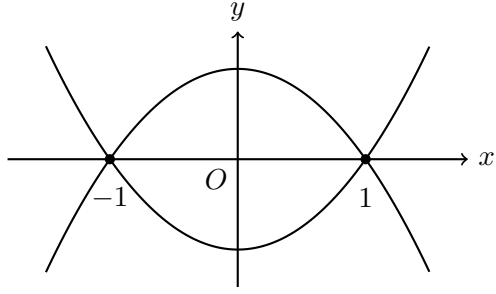

By setting
\begin{equation}
t_\ast = \inf \Big\{ t >0 \mid x(t)= - \frac 12 \Big\},
\label{time1a}
\end{equation}
we have $t_\ast \lec 1$.
Set
\begin{equation}
T = \frac{t_\ast}{\sqrt 2 \de N^{1-s}},
\quad
\de = (\log N)^{-1}.
\label{Ttime1a}
\end{equation}
It follows from \eqref{uvwaa} with $d=1$ and $k=0$, \eqref{alf0a}, and \eqref{time1a} that
\[
\| u(T) \|_{H^s}
= |f(T)|
= \de N^{-s} |\ka ( t_\ast)|
\sim
N^{-s} (\log N)^{-1}
\gg 1,
\]
provided that $s<0$ and $N \gg 1$.
From \eqref{Ttime1a},
we also have
\[
T \lec N^{s-1} \log N \ll 1.
\]
With \eqref{inissa} and $\de = (\log N)^{-1}$,
we obtain the norm inflation in $H^s(\T)$ for $\be+\ga=0$ and $s<0$.

\subsection{The case $\be+\ga=0$ and $s=0$}

We consider the case
\begin{equation}
%\label{cod1}
\notag
\be + \ga =0,
\quad s=0.
\end{equation}
We take $k=0$  in \eqref{uvwaa}.

For $N \gg 1$ and $0<\de \ll 1$,
we set
\begin{equation}
u_0(x)= 1+\de, \quad
v_0(x)= 0, \quad
w_0(x)= \de e^{iNx}.
\label{iniaa1}
\end{equation}
It follows from
$f(0)=1+\de$, $g(0)=0$, $h(0) = \de$,
and \eqref{cons1}
that
\[
f(t)^2 + 2 g(t)^2 - h(t)^2
= 1+2\de.
\]
With
\eqref{ODE1f},
we have the following Cauchy problem:
\[
\left\{
\begin{aligned}
&g''(t) = -N^2 g(t) \big( 2 g(t)^2 - (1+2\de) \big), && t>0, \\
&(g(0),g'(0)) = (0, \de (1+\de)N).
\end{aligned}
\right.
\]
We set 
\begin{equation}
\ka (t) = \sqrt{\frac 2{1+2\de}} g \Big( \frac t{N \sqrt{1+2\de}} \Big).
\label{alf01}
\end{equation}
Then, $\ka$ satisfies
\[
\left\{
\begin{aligned}
&\ka''(t) = - \ka (t) \big( \ka (t)^2 - 1 \big), && t>0, \\
&(\ka(0),\ka'(0)) = \Big( 0, \frac{\sqrt 2 \de (1+\de)}{1+2\de} \Big).
\end{aligned}
\right.
\]

For simplicity, we set $x(t)=\ka(t)$ and $y(t)=\ka'(t)$.
Then, $x(t)$ and $y(t)$ satisfy
\[
\left\{
\begin{aligned}
&x'(t)=y(t), & t>0, \\
&y'(t)=-x(t)^3+x(t), & t>0, \\
&(x(0),y(0)) = \Big( 0, \frac{\sqrt 2 \de (1+\de)}{1+2\de} \Big).
\end{aligned}
\right.
\]
The solution
$(x(t), y(t))$
is on the curve
\[
\frac{y^2}2 + \frac{x^4}4 - \frac{x^2}2
= \Big( \frac{\de (1+\de)}{1+2\de} \Big)^2 =: E_0.
\]
See figure \ref{fig:3}.

Set
\begin{equation}
t_\ast = \inf \Big\{ t >0 \mid x(t)= 1 \Big\}.
\label{time11}
\end{equation}
The same argument as in Lemma \ref{lem:tast0} yields the following:

\begin{lemm}
\label{lem:tast01}
For $s=0$ and $0 <\de \ll 1$,
we have
$
t_\ast
\lec |\log \de|.
$
\end{lemm}

Set
\begin{equation}
T = \frac{t_\ast}{N \sqrt{1+2\de}},
\quad \de = \frac 1N.
\label{Ttime11}
\end{equation}
It follows from \eqref{uvwaa} with $d=1$ and $k=0$, \eqref{alf01}, and \eqref{time11} that
\[
\| v(T) \|_{L^2}
= |g(T)|
= \sqrt{\frac{1+2\de}2} |\ka ( t_\ast)|
\sim
1,
\]
provided that $N \gg 1$.
From Lemma \ref{lem:tast01} and \eqref{Ttime11},
we also have
\[
T \lec N^{-1} \log N \ll 1
\]
for $N \gg 1$.
Note that
\[
\ti u(t,x) =1,
\quad
\ti v(t,x) = \ti w(t,x)=0
\]
is a solution to \eqref{NLS_sys_torus}.
Here,
\eqref{iniaa1} yields that
\[
\| u(0) - \ti u(0) \|_{L^2}
+ \| v(0) - \ti v(0) \|_{L^2}
+ \| w(0) - \ti w(0) \|_{L^2}
= \frac 2N \ll 1
\]
for $N \gg 1$ and $\de = \frac 1N$.
Moreover,
we obtain that
\[
\| v(T) - \ti v(T) \|_{L^2}
= \| v(T) \|_{L^2}
\sim 1,
\]
which shows the discontinuity of the flow map for $\be+\ga=0$ and $s=0$.

\subsection{The case $\al-\ga=0$ and $s<0$}

We consider the case
\[
\al - \ga =0,
\quad
s<0.
\]
In this case,
we take $k=1$ in \eqref{uvwaa}.

Let $N \gg 1$ and $0<\de \ll 1$.
Set
\begin{equation}
u_0(x) = \de N^{-s} e^{iNx}, \quad
v_0(x) = \de, \quad
w_0(x) = 0.
\notag
\end{equation}
Then,
we have
\begin{equation}
\| u_0 \|_{H^s}
+
\| v_0 \|_{H^s}
+
\| w_0 \|_{H^s}
\le 2 \de.
\label{inissab}
\end{equation}

It follows from
$f(0)=\de N^{-s}$, $g(0)=\de$, $h(0) = 0$,
and \eqref{cons1}
that
\[
f(t)^2 + 2g(t)^2 - h(t)^2
= \de^2(2+N^{-2s}).
\]
With
\eqref{ODE1f},
we have the following Cauchy problem:
\[
\left\{
\begin{aligned}
&g''(t) = -N^2 g(t) \big( 2g(t)^2 - \de^2(2+N^{-2s}) \big), && t>0, \\
&(g(0),g'(0)) = (\de, 0).
\end{aligned}
\right.
\]
We set 
\begin{equation}
\ka (t) = \frac{\sqrt 2}{\de \sqrt{2+N^{-2s}}} g \Big( \frac t{\de N \sqrt{2+N^{-2s}} } \Big).
\label{alf0ab}
\end{equation}
Then, $\ka$ satisfies
\[
\left\{
\begin{aligned}
&\ka''(t) = -\ka (t) \big( \ka (t)^2 - 1 \big), && t>0, \\
&(\ka(0),\ka'(0)) = \Big( \sqrt{\frac 2{2+N^{-2s}}}, 0 \Big).
\end{aligned}
\right.
\]

For simplicity, we set $x(t)=\ka(t)$ and $y(t)=\ka'(t)$.
Then, $x(t)$ and $y(t)$ satisfy
\[
\left\{
\begin{aligned}
&x'(t)=y(t), & t>0, \\
&y'(t)=-x(t)^3+x(t),  & t>0, \\
&(x(0),y(0)) = \Big( \sqrt{\frac 2{2+N^{-2s}}}, 0 \Big).
\end{aligned}
\right.
\]
The solution
$(x(t), y(t))$
is on the curve
\[
\frac{y^2}2 + \frac{x^4}4 - \frac{x^2}2
= -\frac{1+N^{-2s}}{(2+N^{-2s})^2}
=: E_0.
\]
Note that $E_0>-\frac 14$.

\begin{figure}[H]
\begin{tikzpicture}[scale=1.7, samples=100]
\draw[thick, ->] (-1.8,0) -- (1.8,0) node [right] {$x$};%x軸
\draw[thick, ->] (0,-1) -- (0,1) node [above] {$y$};%y軸
\coordinate (O) at (0,0) node at (O) [below left] {$O$};
%%%
%%% E_0 =1/16
%%%
\draw[thick, domain={0.4}:{1.3}] plot (\x,{sqrt(-1/8+\x*\x-\x*\x*\x*\x/2)});
\draw[thick, domain={0.4}:{1.3}] plot (\x,{-sqrt(-1/8+\x*\x-\x*\x*\x*\x/2)});
\draw[thick, domain={-0.4}:{0.4}] plot ({sqrt(1+sqrt(1-(1/4)-2*\x*\x))},\x);
\draw[thick, domain={-0.4}:{0.4}] plot ({sqrt(1-sqrt(1-(1/4)-2*\x*\x))},\x);
%%%
\draw[thick, domain={0.4}:{1.3}] plot (-\x,{sqrt(-1/8+\x*\x-\x*\x*\x*\x/2)});
\draw[thick, domain={0.4}:{1.3}] plot (-\x,{-sqrt(-1/8+\x*\x-\x*\x*\x*\x/2)});
\draw[thick, domain={-0.4}:{0.4}] plot ({-sqrt(1+sqrt(1-(1/4)-2*\x*\x))},\x);
\draw[thick, domain={-0.4}:{0.4}] plot ({-sqrt(1-sqrt(1-(1/4)-2*\x*\x))},\x);
%%%
\draw[dashed] plot (1,{-sqrt(3/8)}) -- (1,{sqrt(3/8)});
\draw[dashed] plot (-1,{-sqrt(3/8)}) -- (-1,{sqrt(3/8)});
\fill (1,0) circle(0.04);
\fill (-1,0) circle(0.04);
\node at (-1.2,-0.18) {$-1$};
\node at (1.1,-0.18) {$1$};
\end{tikzpicture}
\caption{$\frac{y^2}2+\frac{x^4}4-\frac{x^2}2=E_0 \in (-\frac 14,0)$}
\end{figure}
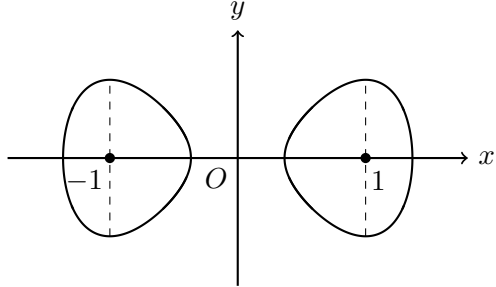

Set
\begin{equation}
t_\ast = \inf \Big\{ t >0 \mid x(t)= 1 \Big\}.
\label{time1ab}
\end{equation}

\begin{lemm}
\label{lem:tast0ab}
For $s<0$ and $N \gg 1$,
we have
$t_\ast \lec \log N$.
\end{lemm}

While the proof follows from the same as in Lemma \ref{lem:tast0},
we give a proof of Lemma \ref{lem:tast0ab} here for completeness.

\begin{proof}
Note that $x(t)$ and $y(t)$ are increasing for $0<t<t_\ast$,
since $0<x(0)<1$ and $y(0)=0$.

By
$\frac{dt}{dx} = \frac{1}{\sqrt{2E_0 -\frac{x^4}2+x^2}}$,
we have
\begin{align*}
t_\ast
&= \int_{\sqrt{\frac 2{2+N^{-2s}}}}^1 \frac{d x}{\sqrt{2E_0 - \frac{x^4}2 + x^2}}
\\
&= \sqrt 2 \int_{\sqrt{\frac 2{2+N^{-2s}}}}^1 \frac{d x}{\sqrt{-(x^2-1-\sqrt{1+4E_0})(x^2-1+\sqrt{1+4E_0})}} \\
&\lec
\int_{\sqrt{\frac 2{2+N^{-2s}}}}^1
\frac{dx}{\sqrt{x^2-1+\sqrt{1+4E_0}}}
= \bigg[ \log \Big| x+\sqrt{x^2-1+\sqrt{1+4E_0}} \Big| \bigg]_{\sqrt{\frac 2{2+N^{-2s}}}}^1
\\
&=
\log \Big( 1 + \sqrt[4]{1+4E_0} \Big)
\underbrace{
-
\log
\bigg(
\sqrt{\frac 2{2+N^{-2s}}}
+ \sqrt{\frac 2{2+N^{-2s}} -1 + \sqrt{1+4E_0}}
\bigg)}
_{= - \log \sqrt{\frac 2{2+N^{-2s}}}}
\\
&\lec
\log N.
\end{align*}
\end{proof}

Set
\begin{equation}
T = \frac{t_\ast}{\de N \sqrt{2+N^{-2s}}},
\quad
\de = (\log N)^{-1}.
\label{Ttime1ab}
\end{equation}
It follows from \eqref{uvwaa} with $d=1$ and $k=1$, \eqref{alf0ab}, and \eqref{time1ab} that
\[
\| v(T) \|_{H^s}
= |g(T)|
= \frac{\de \sqrt{2+N^{-2s}}}{\sqrt 2} |\ka ( t_\ast)|
\sim
N^{-s} ( \log N)^{-1}
\gg 1,
\]
provided that $s<0$ and $N \gg 1$.
From Lemma \ref{lem:tast0ab} and \eqref{Ttime1ab},
we also have
\[
T \lec N^{-1} (\log N)^2 \ll 1.
\]
With \eqref{inissab} and $\de = (\log N)^{-1}$,
we obtain the norm inflation in $H^s(\T)$ for $\al-\ga=0$ and $s<0$.

\begin{rem}
\label{rem:notip}
When $\al - \ga =0$,
even if we take
\[
u_0(x)= 0, \quad
v_0(x)= \de, \quad
w_0(x)= \de N^{-s} e^{-iNx}
\]
as in \eqref{iniaa},
the ill-posedness in $H^s(\T)$ for $s>0$ does not follow.
Indeed,
\eqref{uvwaa} with $d=1$ and $k=1$ and \eqref{cons1} yields that
\begin{align*}
&\| u(t) \|_{H^s}
+ \| w(t) \|_{H^s}
\sim
\sqrt{
\| u(t) \|_{H^s}^2
+ \| w(t) \|_{H^s}^2}
=
\sqrt{f(t)^2+ N^{2s} h(t)^2}
= \de,
\\
&\| v(t) \|_{H^s}
= |g(t)|
\le
\sqrt{f(t)^2+g(t)^2}
= \de.
\end{align*}
\end{rem}

\section{Not locally uniformly continuous}
\label{Sec:7}

In this section, we prove Theorem~\ref{thm:uni1}. 
By the same reason in Section~\ref{sec:IP}, 
we only consider $d=1$ in this section. 

\subsection{The case $\al-\ga=0$ and $s>0$}
\label{subsec:al-ga}

We consider the case
\[
\al - \ga =0.
\]
Let $0<\de \ll 1$ and $N \gg 1$.%
\footnote{We take small $\de$ and then large $N$ in this section.
Note that the order was reversed in Section \ref{sec:IP}.}
Set
\begin{equation}
u_{\pm,0}(x) = 0, \quad
v_{\pm,0}(x) = \pm \de, \quad
w_{\pm,0}(x) = N^{-s} e^{iNx}.
\label{ini00}
\end{equation}
Then,
we have
\begin{equation}
\begin{aligned}
&
\| u_{\pm,0} \|_{H^s}=0,
\quad
\| v_{\pm,0} \|_{H^s}
= \de \ll 1, \quad
\| w_{\pm,0} \|_{H^s}
= 1,
\\
&
\| u_{+,0} - u_{-,0} \|_{H^s}
+
\| v_{+,0} - v_{-,0} \|_{H^s}
+
\| w_{+,0} - w_{-,0} \|_{H^s}
= 2\de \ll 1.
\end{aligned}
\label{inis0}
\end{equation}
Let $(u_\pm, v_\pm, w_\pm)$ be the solution to \eqref{NLS_sys_torus} of the form \eqref{uvwaa} with $d=1$ and $k=1$ and the initial data \eqref{ini00}.
Moreover, $f_\pm, g_\pm, h_\pm$ are defined as in \eqref{uvwaa} with $d=1$.

It follows from
$f_\pm(0)=0$, $g_\pm(0)=\pm \de$, $h_\pm(0) = N^{-s}$,
and \eqref{cons1}
that
\[
2f_\pm(t)^2 + g_\pm(t)^2 + h_\pm(t)^2
= \de^2 +N^{-2s}.
\]
With
\eqref{ODE1f},
we have the following Cauchy problem:
\[
\left\{
\begin{aligned}
&f_\pm''(t) = N^2 f_\pm(t) \big( 2 f_\pm(t)^2 - (\de^2+N^{-2s}) \big), && t>0, \\
&(f_\pm(0),f_\pm'(0)) = (0, \mp \de N^{1-s}).
\end{aligned}
\right.
\]
We set
\begin{equation}
\ka_\pm (t) = \sqrt{\frac 2{\de^2 +N^{-2s}}} f_\pm \Big( \frac t{N \sqrt{\de^2+N^{-2s}}} \Big).
\label{alf}
\end{equation}
Then, $\ka_\pm$ satisfies
\[
\left\{
\begin{aligned}
&\ka_\pm''(t) = \ka_\pm (t) \big( \ka_\pm (t)^2 - 1 \big), && t>0, \\
&(\ka_\pm(0),\ka_\pm'(0)) = \Big( 0, \mp \frac{\sqrt 2 \de N^{-s}}{\de^2 +N^{-2s}} \Big).
\end{aligned}
\right.
\]

For simplicity, we set
$x_\pm (t) = \ka_\pm (t)$ and $y_\pm (t) = \ka_\pm'(t)$.
Then,
$x_\pm (t)$ and $y_\pm (t)$ satisfy
\[
\left\{
\begin{aligned}
&x_\pm'(t)=y_\pm(t), & t>0, \\
&y_\pm'(t)=x_\pm(t)^3-x_\pm(t),  & t>0, \\
&(x_\pm(0),y_\pm(0)) = \Big( 0, \mp \frac{\sqrt 2 \de N^{-s}}{\de^2 +N^{-2s}} \Big).
\end{aligned}
\right.
\]
The solution
$(x_\pm(t), y_\pm(t))$
is on the curve
\[
\frac{y^2}2 - \frac{x^4}4 + \frac{x^2}2 = E_{0} := \frac{\de^2 N^{-2s}}{(\de^2+N^{-2s})^2}.
\]
Note that $\de \neq N^{-s}$ implies that $E_0<\frac 14$.

\begin{figure}[H]
\begin{tikzpicture}[scale=1.7, samples=100]
\draw[thick, ->] (-1.8,0) -- (1.8,0) node [right] {$x$};%x軸
\draw[thick, ->] (0,-1) -- (0,1) node [above] {$y$};%y軸
\coordinate (O) at (0,0) node at (O) [below left] {$O$};
%%%
%%% E_0 =1/8
%%%
\draw[thick, domain={-0.5}:{0.5}] plot (\x,{sqrt(1/4-\x*\x+\x*\x*\x*\x/2)});
\draw[thick, domain={-0.5}:{0.5}] plot ({sqrt(1-sqrt(1-(1/2)+2*\x*\x))},\x);
%%%
\draw[thick, domain={-0.5}:{0.5}] plot (\x,{-sqrt(1/4-\x*\x+\x*\x*\x*\x/2)});
\draw[thick, domain={-0.5}:{0.5}] plot ({-sqrt(1-sqrt(1-(1/2)+2*\x*\x))},\x);
%%%
\draw[thick, domain={-1}:{1}] plot ({sqrt(1+sqrt(1-(1/4)+2*\x*\x))},\x);
\draw[thick, domain={-1}:{1}] plot ({-sqrt(1+sqrt(1-(1/4)+2*\x*\x))},\x);
%%%
\fill (1,0) circle(0.04)node[below]{$1$};
\fill (-1,0) circle(0.04)node[below]{$-1$};
%\node[below] at (1,0) {$1$};
%\node[below] at (-1,0) {$-1$};
\end{tikzpicture}
\caption{$\frac{y^2}2-\frac{x^4}4+\frac{x^2}2=E_0 \in (0, \frac 14)$}
\end{figure}
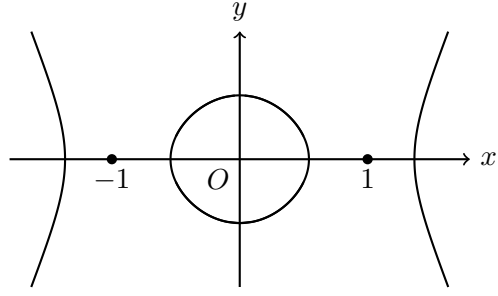

Set
\begin{equation}
t_{\ast, \pm}
= \inf \left\{ t >0 \, \middle| \ x_\pm(t)= \mp  \sqrt{1- \sqrt{1-4E_{0}}} \right\}.
\label{timed}
\end{equation}

\begin{lemm}
\label{lem:tast2}
For $s>0$, $0<\de \ll 1$, and $N> \big( \frac \de 2 \big)^{-\frac 1s}$,
we have
$t_{\ast,+} = t_{\ast,-}$ and
$t_{\ast, \pm} \lec 1$.
\end{lemm}

\begin{proof}
We first consider the case $\pm =+$.
Then,
$x_+(t)$ is decreasing and $y_-(t)$ is increasing for $0<t<t_{\ast,+}$,
since $y(0)<0$.
The condition
$N> \big( \frac \de 2 \big)^{-\frac 1s}$ yields that
$0<E_0<\frac 4{25}$.
It follows from
$\frac{dt}{dx} = - \frac{1}{\sqrt{2E_{0} + \frac{x^4}2-x^2}}$
that
\begin{align*}
t_{\ast,+}
&= \int_{-\sqrt{1- \sqrt{1-4E_{0}}}}^0 \frac{d x}{\sqrt{2E_{0} + \frac{x^4}2 - x^2}}
\\
&= \sqrt 2 \int_{-\sqrt{1- \sqrt{1-4E_{0}}}}^{0}
\frac{d x}{\sqrt{(x^2-1-\sqrt{1-4E_{0}})(x^2-1+\sqrt{1-4E_{0}})}} \\
&= \sqrt 2
\int_{-1}^{0}
\frac{d x}{\sqrt{(1+\sqrt{1-4E_{0}}-(1-\sqrt{1-4E_{0}})x^2)(1-x^2)}}
%\\
%&\to
%\int_{-1}^0 \frac{dx}{\sqrt{1-x^2}}
%= \frac \pi2.
%
\\
&\le
\frac 1{\sqrt[4]{1-4E_{0}}}
\int_{-1}^0 \frac{dx}{\sqrt{1-x^2}}
\lec 1.
\end{align*}

When $\pm=-$,
$x_-(t)$ is increasing and $y_-(t)$ is decreasing for $0<t<t_{\ast,-}$,
since $y(0)>0$.
It follows from
$\frac{dt}{dx} = \frac{1}{\sqrt{2E_{0} + \frac{x^4}2-x^2}}$
that
\[
t_{\ast,-}
= \int_0^{\sqrt{1- \sqrt{1-4E_{0}}}} \frac{d x}{\sqrt{2E_{0} + \frac{x^4}2 - x^2}}.
\]
Since the integrand is an even function,
we have $t_{\ast, +} = t_{\ast,-}$.
\end{proof}

Let $T$ satisfy
\begin{equation}
T = \frac{t_{\ast,\pm}}{N \sqrt{\de^2 + N^{-2s}}}.
\label{time2}
\end{equation}
By
\eqref{uvwaa} with $d=1$ and $k=1$,
\eqref{alf},
and $t_{\ast, +} = t_{\ast, -}$,
we have
\[
\| u_+ (T) - u_-(T) \|_{H^s}
= N^s | f_+ (T) - f_- (T) |
= N^s \sqrt{\frac{\de^2+N^{-2s}}2} |\ka_+ (t_{\ast, +}) - \ka_- (t_{\ast,-})|.
\]
Here,
\eqref{timed} yields that
\[
\ka_\pm (t_{\ast, \pm})
= \mp \sqrt{1- \sqrt{1-4E_0}}
= \mp \sqrt{\frac{4E_0}{\sqrt{1+ \sqrt{1-4E_0}}}}.
\]
It follows from $s>0$ that
\[
\lim_{N \to \infty}
N^s \sqrt{E_0}
= \frac{1}{\de}.
\]
We thus obtain
\begin{equation}
\lim_{N \to \infty}
\| u_+ (T) - u_-(T) \|_{H^s}
=
\frac \de{\sqrt 2}
\Big| \frac{\sqrt 2}{\de}
+ \frac{\sqrt 2}{\de} \Big|
= 2.
\label{low2}
\end{equation}
It follows from Lemma \ref{lem:tast2} and \eqref{time2} that
$\lim_{N \to \infty} T =0$.
Hence,
\eqref{inis0} and \eqref{low2} yield
that
the flow map for \eqref{NLS_sys_torus} fails to be 
locally uniformly continuous in $H^s(\T)$ for $\al-\ga=0$ and $s>0$.

\subsection{The case $\mu \le 0$ and $s<1$}
\label{Subsec:s<1}

Assume that $\mu \le 0$,
where $\mu$ is defined in \eqref{coeff_condition}.
Let $k$ be a real number given in \eqref{conk}.
We do not assume that $k$ is rational here,
namely, $k$ may be irrational.

Let $\{ p_n \}$ and $\{ q_n \}$ be sequences of integers
satisfying
\begin{equation}
%\Big| \frac{\beta+\sqrt{|\mu|}}{\alpha-\beta} + \frac{p_n}{q_n} \Big| < \frac 1{q_n^2}.
\Big| k - \frac{p_n}{q_n} \Big| < \frac 1{q_n^2}
\label{Diop1}
\end{equation}
for any $n \in \N$ and
\[
\lim_{n \to \infty} q_n =\infty.
\]
If $k$ is rational,
there exist integers $p,q$ such that $k= \frac pq$.
Then, we can take $p_n = np$ and $q_n = nq$.
If $k$ is irrational,
from Dirichlet's theorem on Diophantine approximation,
we can choose such sequences.

Note that
\[
\lim_{n \to \infty}
\frac{p_n}{q_n} = k.
\]
When $k \neq 0,1$,
we have
\begin{equation}
|p_n| \sim |q_n| \sim |p_n-q_n|
\label{pqd}
\end{equation}
for $n \gg 1$.
In what follows,
we assume $k \neq 0,1$.
See Remark \ref{rem:nuni} below for the case $k=0,1$.

For $0<\delta \ll 1$,
we set
\begin{equation}
u_{\pm,0}(x) = |p_n|^{-s} e^{ip_n x}, \quad
v_{\pm,0}(x) = \pm \delta |p_n-q_n|^{-s} e^{i(p_n-q_n) x}, \quad
w_{\pm,0}(x) = \pm \delta |q_n|^{-s} e^{iq_n x}.
\label{ini3ab}
\end{equation}
A direct calculation shows that
\begin{equation}
\begin{aligned}
&\| u_{\pm,0} \|_{H^s} = 1,
\quad
\| v_{\pm,0} \|_{H^s}
= \de \ll 1, \quad
\| w_{\pm,0} \|_{H^s}
= \de \ll 1,
\\
&
\begin{aligned}
\| u_{+,0} - u_{-,0} \|_{H^s}
+
\| v_{+,0} - v_{-,0} \|_{H^s}
&+
\| w_{+,0} - w_{-,0} \|_{H^s}
%\\
%&\quad
= 2 \de \ll 1.
\end{aligned}
\end{aligned}
\label{inidd4b}
\end{equation}

We use the same notation as in Subsection \ref{subsec:al-ga}.
Namely,
$(u_\pm, v_\pm, w_\pm)$ denotes the solution to \eqref{NLS_sys_torus} of the form \eqref{uvwaa} and the initial data \eqref{ini3ab}.
Moreover, $f_\pm, g_\pm, h_\pm$ are defined as in \eqref{uvwaa} with $d=1$.

It follows from
$f_\pm (0)= |p_n|^{-s}$, $g_\pm (0)= \pm \de |p_n-q_n|^{-s}$, $h_\pm(0) = \pm \de |q_n|^{-s}$
and \eqref{cons1}
that
\begin{equation}
\label{omn}
|f_\pm(t)|^2 - |g_\pm(t)|^2 + 2|h_\pm(t)|^2
= |p_n|^{-2s} + \de^2 |p_n-q_n|^{-2s} + \de^2 |q_n|^{-2s}
=: \om_n
.
\end{equation}
With
\eqref{ODE1e} and \eqref{omn},
we have the following Cauchy problem:
\begin{equation}
\left\{
\begin{aligned}
&h_\pm''(t) = -q_n^2 h_\pm(t) \big( 2|h_\pm(t)|^2 - \om_n \big) - i \big( \alpha p_n^2 - \beta (p_n-q_n)^2 - \gamma q_n^2 \big) h_\pm'(t), && t>0, \\
&(h_\pm(0),h_\pm'(0)) = (\pm \de |q_n|^{-s}, \pm \de q_n |p_n|^{-s} |p_n-q_n|^{-s}).
\end{aligned}
\right.
\label{hpm3}
\end{equation}
Moreover,
by \eqref{cons1} and \eqref{ODE1e},
we have
\begin{equation}
\begin{aligned}
|h_\pm(t)|
&\le
\sqrt{|f_\pm(t)|^2 + | h_\pm (t)|^2}
\sim |q_n|^{-s},
\\
|h_\pm'(t)|
&\le
|q_n|
(|f_\pm(t)|^2 + |g_\pm (t)|^2)
\sim |q_n|^{1-2s}.
\end{aligned}
\label{bdqq}
\end{equation}

Let $\ti h_\pm$ be the solution to the Cauchy problem:
\begin{equation}
\left\{
\begin{aligned}
&\ti h_\pm''(t) = -q_n^2 \ti h_\pm(t) \big( 2|\ti h_\pm(t)|^2 - \om_n \big)
, && t>0, \\
&(\ti h_\pm(0),\ti h_\pm'(0)) = (\pm \de |q_n|^{-s}, \pm \de q_n |p_n|^{-s} |p_n-q_n|^{-s}).
\end{aligned}
\right.
\label{hpm3a}
\end{equation}
Note that $\ti h_\pm$ is real-valued, since the initial data are real numbers.
We set 
\begin{equation}
\ka_\pm (t) = \sqrt{\frac 2{\om_n}} \ti h_\pm \Big( \frac t{ \sqrt{\om_n} q_n} \Big).
\label{al33b}
\end{equation}
Then, $\ka_\pm$ satisfies
\[
\left\{
\begin{aligned}
&\ka_\pm''(t) = -\ka_\pm (t)(\ka_\pm(t)^2-1), && t>0, \\
&(\ka_\pm(0),\ka_\pm'(0)) = \Big( \pm \sqrt{\frac 2{\om_n}} |q_n|^{-s} \de,
\pm \frac{\sqrt 2}{\om_n} |p_n|^{-s} |p_n-q_n|^{-s}
\de
\Big).
\end{aligned}
\right.
\]

For simplicity, we set $x_\pm(t)=\ka_\pm(t)$ and $y_\pm(t)=\ka_\pm'(t)$.
Then, $x_\pm(t)$ and $y_\pm(t)$ satisfy
\[
\left\{
\begin{aligned}
&x_\pm'(t)=y_\pm(t), & t>0, \\
&y_\pm'(t)=-x_\pm(t)^3+x_\pm(t),  & t>0, \\
&(x_\pm(0),y_\pm(0)) =
\Big( \pm \sqrt{\frac 2{\om_n}} |q_n|^{-s} \de,
\pm \frac{\sqrt 2}{\om_n} |p_n|^{-s} |p_n-q_n|^{-s}
\de
\Big).
\end{aligned}
\right.
\]
The solution
$(x_\pm(t), y_\pm(t))$
is on the curve
\[
\frac{y^2}2 + \frac{x^4}4 - \frac{x^2}2
= 0.
\]
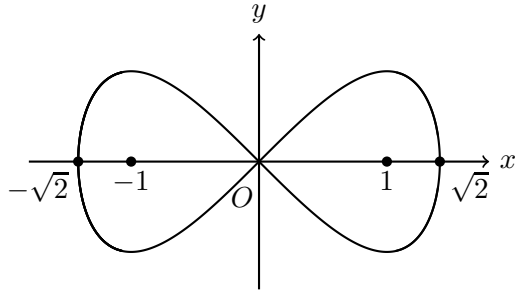
\begin{figure}[H]
\begin{tikzpicture}[scale=1.7, samples=100]
\draw[thick, ->] (-1.8,0) -- (1.8,0) node [right] {$x$};%x軸
\draw[thick, ->] (0,-1) -- (0,1) node [above] {$y$};%y軸
\coordinate (O) at (0,0) node at (-0.13,-0.28) {$O$};
%%%
\draw[thick, domain={0}:{1.4}] plot (\x,{sqrt(0\x*\x-\x*\x*\x*\x/2)});
\draw[thick, domain={-0.7}:{0.7}] plot ({sqrt(1+sqrt(1-2*\x*\x))},\x);
%%%
\draw[thick, domain={0}:{1.4}] plot (\x,{-sqrt(0\x*\x-\x*\x*\x*\x/2)});
\draw[thick, domain={-0.7}:{0.7}] plot ({-sqrt(1+sqrt(1-2*\x*\x))},\x);
%%%
\draw[thick, domain={0}:{1.4}] plot (-\x,{sqrt(0\x*\x-\x*\x*\x*\x/2)});
\draw[thick, domain={-0.7}:{0.7}] plot ({-sqrt(1+sqrt(1-2*\x*\x))},\x);
%%%
\draw[thick, domain={0}:{1.4}] plot (-\x,{-sqrt(0\x*\x-\x*\x*\x*\x/2)});
\draw[thick, domain={-0.7}:{0.7}] plot ({-sqrt(1+sqrt(1-2*\x*\x))},-\x);
%%%
\fill (1,0) circle(0.04)node[below]{$1$};
\fill (-1,0) circle(0.04)node[below]{$-1$};
%\node[below] at (1,0) {$1$};
%\node[below] at (-1,0) {$-1$};
\fill ({sqrt 2},0) circle(0.04)node[below right]{$\sqrt 2$};
\fill ({-sqrt 2},0) circle(0.04)node[below left]{$-\sqrt 2$};
%\node at (1.6,-0.2) {$\sqrt 2$};
%\node at (-1.7,-0.2) {-$\sqrt 2$};
\end{tikzpicture}
\caption{$\frac{y^2}2+\frac{x^4}4-\frac{x^2}2=0$}
\end{figure}
In particular,
it follows from \eqref{omn} and \eqref{pqd} that
\begin{equation}
|\ti h_\pm (t)|
= \sqrt{\frac{\om_n}2} | \ka_\pm (\sqrt{\om_n} q_n t)|
\le
\sqrt{\om_n}
\sim |q_n|^{-s}.
\label{bdqq2}
\end{equation}

\begin{lemm}
\label{lem:diff5a}
Assume that $\mu \le 0$.
Let $h_\pm$ and $\ti h_\pm$ be solutions to \eqref{hpm3} and \eqref{hpm3a}, respectively.
Then, we have
\[
| h_\pm(t) - \ti h_\pm(t)|
\lec t^2 |q_n|^{1-2s} \exp \big( t^2 |q_n|^{2(1-s)} \big)
\]
for $t>0$ and $n \gg 1$.
\end{lemm}

\begin{proof}
Set
\[
d_\pm := h_\pm - \ti h_\pm.
\]
By \eqref{hpm3} and \eqref{hpm3a},
$d_\pm$ satisfies
\[
\left\{
\begin{aligned}
&d_\pm''(t)
=
-q_n^2 d_\pm(t) \big( 2|h_\pm(t)|^2 - \om_n \big)
-q_n^2 \ti h_\pm(t) \big( 2 d_\pm(t) \cj{h_\pm(t)} - \om_n \big)
\\
&\hspace*{50pt}
-q_n^2 \ti h_\pm(t) \big( 2 \ti h_\pm(t) \cj{d_\pm(t)} - \om_n \big)
- i \big( \alpha p_n^2 - \beta (p_n-q_n)^2 - \gamma q_n^2 \big) h_\pm'(t), && t>0, \\
&(d_\pm(0),d_\pm'(0)) = (0,0).
\end{aligned}
\right.
\]
Recall that $k$ is defined in \eqref{conk}.
When $\alpha - \beta \neq 0$,
it follows from $\mu \le 0$, \eqref{Diop1}, and \eqref{pqd} that
\begin{align*}
| \alpha p_n^2 - \beta (p_n-q_n)^2 - \gamma q_n^2 |
&=
| (\alpha -\beta) p_n^2 + 2 \beta p_n q_n - (\beta+\gamma) q_n^2 |
\\
&=
\Big| (\alpha-\beta)
\Big( p_n + \frac{\beta+\sqrt{|\mu|}}{\alpha-\beta} q_n \Big)
\Big( p_n + \frac{\beta-\sqrt{|\mu|}}{\alpha-\beta} q_n \Big)
\Big|
\\
&\lec 1.
\end{align*}
When $\alpha - \beta = 0$,
a similar calculation yields that
\begin{align*}
| \alpha p_n^2 - \beta (p_n-q_n)^2 - \gamma q_n^2 |
&=
| 2 \beta p_n q_n - (\beta+\gamma) q_n^2 |
\\
&=
\Big| 2 \beta \Big( \frac{p_n}{q_n}  - \frac{\beta+\gamma}{2\beta} \Big) q_n^2 \Big|
\lec 1.
\end{align*}

Set
\[
D_\pm (t) := \sup_{0<t'<t} | d_\pm(t')|.
\]
From the corresponding integral equation with \eqref{bdqq} and \eqref{bdqq2},
we obtain
\[
D_\pm(t)
\lec
|q_n|^{2-2s}
\int_0^t t' D_\pm(t') dt'
+ t^2 |q_n|^{1-2s}.
\]
Gronwall's inequality yields that
\[
D_\pm(t)
\lec
t^2 |q_n|^{1-2s}
\exp \big( c t^2 |q_n|^{2-2s} \big),
\]
which shows the desired bound.
\end{proof}

Set
\begin{equation}
t_{\ast, \pm} = \inf \Big\{ t >0 \mid x_\pm (t)= \pm 1 \Big\}.
%\label{time1}
\notag
\end{equation}
By symmetry,
we have $t_{\ast,+} = t_{\ast,-}$.
The same argument as in Lemma \ref{lem:tast2} yields that
\[
t_{\ast, \pm} \lec |\log \de|.
\]
Set
\begin{equation}
T = \frac{t_{\ast, \pm}}{\sqrt{\om_n} q_n}.
\label{timT3b}
\end{equation}
Then, Lemma \ref{lem:diff5a} with \eqref{omn} and \eqref{pqd}
imply that
\begin{equation}
|h_\pm(T) - \ti h_\pm(T)|
%\lec
%T^2 |q_n|^{1-2s}
%\exp \big( C  (\log \de)^2 \big)
\lec |q_n|^{-1}
\exp \big( \theta (\log \de)^2 \big)
\label{diff6a}
\end{equation}
for some constant $\theta >0$.
By
\eqref{uvwaa}, \eqref{al33b}, $t_{\ast,+}= t_{\ast,-}$, and \eqref{diff6a},
we have
\[
\begin{aligned}
&\| w_+(T) - w_-(T) \|_{H^s}
\\
&= |q_n|^s |h_+(T) - h_-(T)|
\\
&\ge
|q_n|^s |\ti h_+(T) - \ti h_-(T)|
- |q_n|^s
\big( |h_+(T) - \ti h_+(T)| + |h_-(T) - \ti h_-(T)| \big)
\\
&\ge
|q_n|^s \sqrt{\frac{\om_n}2}|\ka_+ (t_{\ast,+}) - \ka_- (t_{\ast,-})|
- C |q_n|^{s-1} \exp \big( \theta (\log \de)^2 \big)
\\
&
=
\sqrt 2 |q_n|^s \sqrt{\om_n}
- C |q_n|^{s-1} \exp \big( \theta (\log \de)^2 \big).
\end{aligned}
\]
From \eqref{omn}, \eqref{pqd}, and $s<1$,
we obtain that
\begin{equation}
\| w_+(T) - w_-(T) \|_{H^s}
\sim 1
\label{funif3b}
\end{equation}
for $0<\de \ll 1$ and $n \gg 1$.

It follow from \eqref{timT3b} and \eqref{pqd} that 
$\lim_{n \to \infty} T=0$ for $s<1$.
With \eqref{inidd4b} and \eqref{funif3b},
the flow map for \eqref{NLS_sys_torus} fails to be 
locally uniformly continuous in $H^s(\T)$ for $s<1$.

\begin{rem}
\label{rem:nuni}

By \eqref{conk},
the conditions $k=0$ and $k=1$ correspond to $\be+\ga=0$ and $\al-\ga=0$, respectively.
The argument above also works for $k=1$ and $0 \le s <1$.
Indeed,
when $k=1$,
we replace \eqref{ini3ab} by
\[
u_{\pm,0}(x) = N^{-s} e^{i N x}, \quad
v_{\pm,0}(x) = \pm \delta N^{-s} , \quad
w_{\pm,0}(x) = \pm \delta N^{-s} e^{i N x}.
\]
Then, we have
\[
\begin{aligned}
&\| u_{\pm,0} \|_{H^s} =1,
\quad
\| v_{\pm,0} \|_{H^s}
= \de N^{-s} \ll 1, \quad
\| w_{\pm,0} \|_{H^s}
= \de \ll 1,
\\
&
\begin{aligned}
\| u_{+,0} - u_{-,0} \|_{H^s}
+
\| v_{+,0} - v_{-,0} \|_{H^s}
&+
\| w_{+,0} - w_{-,0} \|_{H^s}
\\
&\quad
= 2 ( N^{-s} + 1) \de \ll 1
\end{aligned}
\end{aligned}
\]
for $s \ge 0$.
Moreover,
$h_\pm$ satisfies
\[
\left\{
\begin{aligned}
&h_\pm''(t) = -N^2 h_\pm(t) \big( 2h_\pm(t)^2 - (1+\de^2) N^{-2s} \big), && t>0, \\
&(h_\pm(0),h_\pm'(0)) = (\pm \de N^{-s}, \pm \de N^{1-2s}).
\end{aligned}
\right.
\]
Thus,
the same argument above implies \eqref{funif3b} for $s<1$.
Namely, the flow map fails to be locally uniformly continuous for $\al-\ga=0$ and $0 \le s<1$.
With Theorem \ref{thm:IP2} and the result in Subsection \ref{subsec:al-ga},
we obtain Theorem \ref{thm:uni1} (ii) for $(\be+\ga)(\al-\ga)=0$.
\end{rem}

\mbox{}

\noindent
{\bf 
Acknowledgements.}
%\section*{acknowledgements}
H.H.~was supported by JSPS KAKENHI  Grant number JP21K13825.
S.K.~was supported by JSPS KAKENHI  Grant number JP24K16945.
M.O.~was supported by JSPS KAKENHI  Grant number JP23K03182.
The authors would like to thank the anonymous referees for helpful comments


\begin{thebibliography}{99}
%\bibitem{Aka}
%K. Akase,
%{\itshape Well-posedness for a nonlinear Schr\"odinger equation with quadratic derivative nonlinearities for bounded primitive initial data},
%arXiv:2312.16234 [math.AP].


%\bibitem{Be06}I. Bejenaru, {\itshape Quadratic nonlinear derivative Schr\"odinger equations. Part I},
%Int. Math. Res. Pap. {\bfseries 2006} (2006), 84pp.

\bibitem{Be08}I. Bejenaru, {\itshape Quadratic nonlinear derivative Schr\"odinger equations. Part II},
Trans. Amer. Math. Soc. {\bfseries 360} (2008), no. 11, 5925--5957.

\bibitem{BHHT09}
I. Bejenaru, S. Herr, J. Holmer, and D. Tataru,
{\itshape On the 2D Zakharov system with $L^2$ Schr\"odinger data},
Nonlinearity {\bfseries 22} (2009), no. 5, 1063--1089.


\bibitem{BL01}H. Biagioni and F. Linares, {\itshape Ill-posedness for the derivative Schr\"odinger and generalized Benjamin-Ono equations},
Trans. Amer. Math. Soc. {\bfseries 353} (2001), no. 9, 3649--3659.

\bibitem{Bo93}J. Bourgain, {\itshape Fourier transform restriction phenomena for certain lattice
  subsets and applications to nonlinear evolution equations. I. Schr\"odinger equations},
  Geom. Funct. Anal. {\bfseries 3} (1993), no. 2, 107--156. 

\bibitem{Bo13}J. Bourgain, {\itshape Moment inequalities for trigonometric polynomials with spectrum
in curved hypersurfaces},
  Israel J. Math. {\bfseries 193} (2013), no. 1, 441--458. 

\bibitem{BD15}J. Bourgain and C. Demeter, {\itshape The proof of the $l^2$ decoupling conjecture},
Ann. of Math. (2) {\bfseries 182} (2015), no. 1, 351--389.

\bibitem{BGT02}
N. Burq, P. G\'erard, and N. Tzvetkov,
{\itshape An instability property of the nonlinear Schr\"odinger equation on $\mathbb S^d$},
Math. Res. Lett. {\bfseries 9} (2002), no. 2-3, 323--335.

 
%\bibitem{Chi95}H. Chihara, {\itshape Local existence for semilinear Schr\"odinger equations},
%Math. Japon. {\bfseries 42} (1995), no. 1, 35--51. 

\bibitem{Chi99}H. Chihara, {\itshape Gain of regularity for semilinear Schr\"odinger equations},
Math. Ann. {\bfseries 315} (1999), no. 4, 529--567. 

\bibitem{Chi02}
H. Chihara,
\textit{The initial value problem for Schr\"odinger equations on the torus},
Int. Math. Res. Not. 2002, no. 15, 789--820.

\bibitem{Ch}M. Christ, {\itshape Illposedness of a Schr\"odinger equation with derivative nonlinearity}, preprint
(\url{https://math.berkeley.edu/~mchrist/preprints.html}).

\bibitem{CGKO17}
J. Chung, Z. Guo, S. Kwon, and T. Oh,
{\itshape Normal form approach to global well-posedness of the quadratic derivative nonlinear Schr\"odinger equation on the circle},
Ann. Inst. H. Poincar\'e C Anal. Non Lin\'eaire 34 (2017), no.5, 1273--1297.


\bibitem{CCT03}
M. Christ, J. Colliander, and T. Tao,
{\itshape Asymptotics, frequency modulation, and low regularity ill-posedness for canonical defocusing equations},
Amer. J. Math. {\bfseries 125} (2003), no. 6, 1235--1293.

\bibitem{CC04}M. Colin and T. Colin, {\itshape On a quasilinear Zakharov system describing laser-plasma interactions},
  Differential Integral Equations, {\bfseries 17} (2004), no. 3-4, 297--330. 

%\bibitem{CKSTT01}J. Colliander, M. Keel, G. Staffilani, H. Takaoka, and T. Tao,
%{\itshape Global well-posedness result
%for Schr\"odigner equations with derivative},
%  SIAM J. Math. Anal. {\bfseries 33} (2001), no. 3, 649--669. 

\bibitem{CKSTT02}J. Colliander, M. Keel, G. Staffilani, H. Takaoka, and T. Tao,
{\itshape A refined global well-posedness result
for Schr\"odigner equations with derivative},
SIAM J. Math. Anal. {\bfseries 34} (2002), no. 1, 64--86.

%\bibitem{Gr}A. Gr\"unrock, {\itshape On the Cauchy - and periodic boundary value problem for a certain class of derivative nonlinear Schr\"odinger equations}, 
%preprint (arXiv:0006195v1 [math.AP]). 

\bibitem{GTV97}
J. Ginibre, Y. Tsutsumi, and G. Velo,
{\itshape On the Cauchy problem for the Zakharov system},
J. Funct. Anal. {\bfseries 151} (1997), no. 2, 384--436.

\bibitem{Gru00}
A. Gr\"unrock,
\textit{On the Cauchy- and periodic boundary value problem for a certain class of derivative nonlinear Schr\"odinger equations},
arXiv:math/0006195.

\bibitem{HHK09}M. Hadac, S. Herr, and H. Koch, {\itshape Well-posedness and scattering for the KP-II equation in a critical space},
  Ann. Inst. H. Poincar\'e CAnal. Non lin\'eaire. {\bfseries 26} (2009), no. 3, 917--941.
  
\bibitem{HHK10}M. Hadac, S. Herr, and H. Koch, {\itshape Errantum to ``Well-posedness and scattering for the KP-II equation in a critical space'' [Ann. I. H. Poincar\'e--AN26 (3) (2009) 917--941]}, Ann. Inst. H. Poincar\'e CAnal. Non lin\'eaire. {\bfseries 27} (2010), no. 3, 971--972. 

\bibitem{HKNV}
B. Harrop-Griffiths, R.Killip, M. Ntekoume, and M. Visan,
\textit{Global well-posedness for the derivative nonlinear Schr\"odinger equation in $L^2(\R)$},
arXiv:2204.12548.


\bibitem{He06}S. Herr, {\itshape On the Cauchy problem for the derivative nonlinear Schr\"odinger 
equation with periodic boundary condition},
 Int. Math. Res. Not. {\bfseries 2006} (2006), 33 pp.
  
  
\bibitem{HTT11}S. Herr, D. Tataru, and N. Tzvetkov, {\itshape Global well-posedness of the energy-critical nonlinear Schr\"odinger equation with small initial data in $H^{1}(\T^{3})$},
  Duke. Math. J. {\bfseries 159} (2011), no. 2, 329--349. 
  
\bibitem{Hi}H. Hirayama, {\itshape Well-posedness  and scattering for a system of quadratic derivative
nonlinear Schr\"odinger equations with low regularity initial data}, Commun. Pure Appl. Anal. {\bfseries 13} (2014), no. 4, 1563--1591. 

\bibitem{HK}H. Hirayama and S. Kinoshita, 
{\itshape Sharp bilinear estimates and its application to 
a system of quadratic derivative nonlinear Schr\"odinger equations}, 
Nonlinear Anal. {\bfseries 178} (2019), 205--226. 

\bibitem{HKO2020}H. Hirayama, S. Kinoshita, and M. Okamoto, 
{\itshape Well-posedness for a system of quadratic deivative 
nonlinear Schr\"odinger equations with radial initial data}, 
Ann. Henri Poincar\'e {\bfseries 21} (2020), 2611--2636.

\bibitem{HKO2021}
H.~Hirayama, S.~Kinoshita, and M.~Okamoto,
\textit{Well-posedness for a system of quadratic derivative nonlinear Schr\"odinger equations in almost critical spaces},
J. Math. Anal. Appl. {\bfseries 499} (2021), no. 2, Paper No. 125028, 29 pp.

\bibitem{HKO22}
H.~Hirayama, S.~Kinoshita, and M.~Okamoto,
\textit{A remark on the well-posedness for a system of quadratic derivative nonlinear Schr\"odinger equations},
Commun. Pure Appl. Anal. {\bfseries 21} (2022), no. 10, 3309--3334.

\bibitem{IKO16}
M. Ikeda, N. Kishimoto, and M. Okamoto,
{\itshape
Well-posedness for a quadratic derivative nonlinear Schr\"odinger system at the critical regularity},
J. Funct. Anal. {\bfseries 271} (2016), no. 4, 747--798.

\bibitem{KPV98}C. Kenig, G. Ponce, and L. Vega, {\itshape Smoothing effects and local existence theory for the generalized nonlinear Schr\"odinger equations},
  Invent. Math. {\bfseries 134} (1998), no. 3, 489--545.  
%

\bibitem{KNV23}
R. Killip, M. Ntekoume, M.Vi\c{s}an,
\textit{On the well-posedness problem for the derivative nonlinear Schr\"{o}dinger equation}
Anal. PDE \textbf{16} (2023), no. 5, 1245--1270.

\bibitem{KV16}
R. Killip and M. Vi\c{s}an,
\textit{Scale invariant Strichartz estimates on tori and applications},
Math. Res. Lett. \textbf{23} (2016), no. 2, 445--472.

\bibitem{KS21}
S. Kinoshita and R. Schippa,
\textit{Loomis-Whitney-type inequalities and low regularity well-posedness of the periodic Zakharov-Kuznetsov equation},
J. Funct. Anal. {\bfseries 280} (2021), no. 6, 108904.

\bibitem{Kishi13}
N. Kishimoto, \textit{Local well-posedness for the Zakharov system on the multidimensional torus}, J. Anal. Math. {\bfseries 119} (2013), 213--253.

\bibitem{KT18}H. Koch and D. Tataru, {\itshape Conserved energies for the cubic nonlinear 
Schr\"odinger equation in one dimension},
Duke Math. J. {\bfseries 167} (2018), no. 17, 3207--3313.  

\bibitem{KoOk1}
T. Kondo and M. Okamoto,
\textit{Norm inflation for a higher-order nonlinear Schr\"odinger equation with a derivative on the circle},
Partial Differ. Equ. Appl. {\bfseries 6} (2025), no. 2., Paper No. 11, 14pp. 

\bibitem{KoOk2}
T. Kondo and M. Okamoto,
\textit{Well- and ill-posedness of the Cauchy problem for semi-linear Schr\"odinger equations on the torus},
arXiv:2501.04205.

\bibitem{LO24}
R. Liu and T. Oh,
{\itshape Sharp local well-posedness of the two-dimensional periodic nonlinear Schr\"{o}dinger equation with a quadratic nonlinearity $|u|^2$}, Math. Res. Lett. {\bfseries 31} (2024), no.1, 255--277. 


%
\bibitem{MWX11}C. Miao, Y. Wu, and G. Xu, {\itshape Global well-posedness for Schr\"odinger equation with derivative in $H^{\frac 12}(\R )$}, J. Differential. Equations {\bfseries 251} (2011), no.8, 2164--2195.

\bibitem{Oh09}T. Oh, {\itshape Diophantine conditions in well-posedness theory of coupled KdV-Type systems: local Theory},
Int. Math. Res. Not. IMRN {\bfseries 2009}, no. 18, 3516--3556.

%\bibitem{Oz98}T. Ozawa, {\itshape Finite energy solutions for the Schr\"odinger equations with
%quadratic nonlinearity in one space dimension},
%Funkcial. Ekvac. {\bfseries 41} (1998), no. 3, 451--468.

%\bibitem{St07}A. Stefanov, {\itshape On quadratic derivative Schr\"odinger equations in one space dimension}, 
%Trans. Amer. Math. Soc. {\bfseries 359} (2007), no. 8, 3589--3607.

%\bibitem{Tak99}H. Takaoka, {\itshape Well-posedness for the one dimensional Schr\"odinger equation with the derivative nonlinearity},
%Adv. Differential. Equations {\bfseries 4} (1999), no. 4, 561--580.

\bibitem{Tak01}H. Takaoka, {\itshape Global well-posedness for Schr\"odinger equations with derivative in a nonlinear term and data in low-order Sobolev spaces},
Electron. J. Differential Equations {\bfseries 2001}, No. 42, 23 pp.

\bibitem{Wa} Y. Wang, {\itshape Periodic nonlinear Schr\"odinger equation in critical $H^{s}(\T^{n})$ spaces},
SIAM J. Math. Anal. {\bfseries 45} (2013), no. 3, 1691--1703.

\bibitem{Wi10}Y. Win, {\itshape Global Well-Posedness of the Derivative Nonlinear Schr\"odinger Equations on $\T$},
Funkcial. Ekvac. {\bfseries 53} (2010), no. 1, 51--88.
\end{thebibliography}
\end{document}